\documentclass[12pt, a4paper]{amsart}

\usepackage{amsmath,amssymb,amsfonts,amsthm}
\usepackage{setspace}
\usepackage{hyperref}
\usepackage{aliascnt}
\usepackage{comment}
\usepackage{pictexwd,dcpic}
\usepackage[usenames,dvipsnames]{color}
\usepackage{graphicx}
\usepackage{stmaryrd}
\usepackage{enumerate}
\theoremstyle{plain}
\newcommand\macDisplayText{m.a.c.}
\newcommand\mecDisplayText{m.e.c.}

\usepackage{xstring}

\usepackage{titlecaps}

\newif\ifmacdot
\IfEndWith{\macDisplayText}{.}{\macdottrue}{\macdotfalse}
\newif\ifmecdot
\IfEndWith{\mecDisplayText}{.}{\mecdottrue}{\mecdotfalse}

\newcommand\mac\macDisplayText
\newcommand\mec\mecDisplayText

\newtheorem{thrm}{Theorem}

\newtheorem{theorem}{Theorem}[subsection]
\newaliascnt{proposition}{theorem}
\newaliascnt{lemma}{theorem}
\newaliascnt{corollary}{theorem}
\newaliascnt{claim}{theorem} 
\newaliascnt{fact}{theorem}
\newaliascnt{observation}{theorem}
\newaliascnt{conjecture}{theorem}
\newaliascnt{definition}{theorem}
\newaliascnt{example}{theorem}
\newaliascnt{examples}{theorem}
\newaliascnt{question}{theorem}
\newaliascnt{remark}{theorem}
\newaliascnt{property}{theorem}
\newaliascnt{construction}{theorem}
\newaliascnt{setting}{theorem}
\newaliascnt{problem}{theorem}

\theoremstyle{plain}
\newtheorem{proposition}[proposition]{Proposition}
\newtheorem{lemma}[lemma]{Lemma}
\newtheorem{corollary}[corollary]{Corollary}
\newtheorem{claim}[claim]{Claim} 
\newtheorem{fact}[fact]{Fact}

\theoremstyle{definition}
\newtheorem{definition}[definition]{Definition}
\newtheorem{example}[example]{Example}
\newtheorem{examples}[examples]{Examples}
\newtheorem{question}[question]{Question}
\newtheorem{conjecture}[conjecture]{Conjecture}
\newtheorem{problem}[problem]{Problem}
\newtheorem{remark}[remark]{Remark}

\aliascntresetthe{proposition}
\aliascntresetthe{lemma}
\aliascntresetthe{corollary}
\aliascntresetthe{claim} 
\aliascntresetthe{fact}
\aliascntresetthe{observation}
\aliascntresetthe{conjecture}
\aliascntresetthe{definition}
\aliascntresetthe{example}
\aliascntresetthe{examples}
\aliascntresetthe{question}
\aliascntresetthe{remark}
\aliascntresetthe{property}
\aliascntresetthe{construction}
\aliascntresetthe{setting}
\aliascntresetthe{problem}

\def\Aut{\mathop{\rm Aut}\nolimits}
\def\Alt{\mathop{\rm Alt}\nolimits}
\def\Th{\mathop{\rm Th}\nolimits}
\def\tp{\mathop{\rm tp}\nolimits}
\def\Min{\mathop{\rm Min}\nolimits}

\newcommand{\Rad}{\mathrm{Frac}(R)}
\def\bil{\mathop{\rm bil}\nolimits}
\def\qe{\mathop{\rm qe}\nolimits}
\def\rings{\mathop{\rm rings}\nolimits}
\newcommand{\mon}[1]{\langle {Z}^{#1}\rangle}

\renewcommand{\emptyset}{\varnothing}
\renewcommand{\phi}{\varphi}
\renewcommand{\epsilon}{\varepsilon}
\renewcommand{\longrightarrow}{\to}
\renewcommand{\longmapsto}{\mapsto}
\renewcommand{\Leftrightarrow}{\iff}

\title[Multidimensional asymptotic classes]{Multidimensional asymptotic classes}
\author{Sylvy Anscombe, Dugald Macpherson, Charles Steinhorn and Daniel Wolf}
\thanks{\today\\This research was funded by EPSRC grant EP/K020692/1.
The third author was partially funded by Simons Foundation Mathematics and Physical Sciences Collaboration Grant for Mathematicians, award \#524012.
The fourth author (born Daniel Wood) was funded by the Leeds School of Mathematics through a Graduate Teaching Assistantship.
Part of the present work forms part of his PhD thesis \cite{WolfPhD}.}
\address{Universit\'{e} Paris Cit\'{e} and Sorbonne Universit\'{e}, CNRS, IMJ-PRG, F-75013 Paris, France}
\email{sylvy.anscombe@imj-prg.fr}
\address{School of Mathematics, University of Leeds, Leeds LS2 9JT, United Kingdom}
\email{h.d.macpherson@leeds.ac.uk}
\address{Department of Mathematics \& Statistics, Vassar College, 124 Raymond Avenue, Poughkeepsie, New York 12604}
\email{steinhorn@vassar.edu}
\address{\emph{Formerly of the} School of Mathematics, University of Leeds, Leeds LS2 9JT, United Kingdom}
\email{dwolfeu@gmail.com}

\begin{document}
\begin{abstract}
We develop a general framework (multidimensional asymptotic classes, or \mac{}s) for handling classes of finite first order structures with a strong uniformity condition on cardinalities of definable sets: The condition asserts that definable families given by a formula $\phi(\bar{x},\bar{y})$ should take on a fixed number $n_\phi$ of approximate sizes in any $M$ in the class, with those sizes varying with $M$. The prototype is the class of all finite fields, where the uniformity is given by a theorem of Chatzidakis, van den Dries and Macintyre. It inspired the development of asymptotic classes of finite structures, which this new framework extends. 
 
The underlying theory of \mac{}s is developed, including preservation under bi-interpretability, and a proof that for the \mac{} condition to hold  it suffices to consider formulas $\phi(x,\bar{y})$ with $x$ a single variable. Many examples of \mac{}s are given, including 2-sorted structures $(F,V)$ where $V$ is a vector space over a finite field $F$ possibly equipped with a bilinear form, and an example arising from representations of quivers of finite representation type. We also give examples and structural results for multidimensional {\em exact} classes (\mec{}s), where the definable sets take a fixed number of precisely specified cardinalities, which again vary with $M$.

We also develop a notion of infinite {\em generalised measurable} structure, whereby definable sets are assigned values in an ordered semiring. We show that any infinite ultraproduct of a \mac{} is generalised measurable, that values can be taken in an ordered {\em ring} if the \mac{} is a \mec{}, and explore model-theoretic consequences of generalised measurability. Such a structure cannot have the strict order property, and stability-theoretic properties can be read off from the measures in the semiring.

\end{abstract}
\maketitle

\tableofcontents

\section{Introduction}
\label{section:introduction}

A classical theorem of Lang and Weil (see \cite[Theorem 1]{LW54}) states that the number of $\mathbb{F}_{q}$-rational points of an absolutely irreducible variety $V\subseteq\mathbb{P}^{n}$ defined over $\mathbb{F}_q$ is approximately equal to $q^{r}$, where $\mathbb{F}_{q}$ is the finite field of order $q$ and $V$ is of algebraic dimension $r$. `Approximately' here means that the difference $\big||V(\mathbb{F}_{q})|-q^{r}\big|$ is bounded by a constant multiple of $q^{r-\frac{1}{2}}$. This constant depends only on certain information about $V$: namely $n$, $r$, and the degrees of the polynomials defining $V$. Using `big $O$ notation', we may re-write this approximation (sacrificing some information) as an asymptotic statement:
\begin{equation}
\big||V(\mathbb{F}_{q})|-q^{r}\big|=O(q^{r-\frac{1}{2}})\quad\text{as $q\longrightarrow\infty$}.
\end{equation}
Later work of Chatzidakis, van den Dries, and Macintyre (see \cite[Main Theorem]{CvdDM92}) extends these asymptotics to definable sets in finite fields, in the language of rings $L_{\mathrm{ring}}$. More precisely, for any $L_{\mathrm{ring}}$-formula $\phi(\bar{x};\bar{y})$,
there is a finite set $D\subseteq\{0,\ldots,|\bar{x}|\}\times\mathbb{Q}^{>0}$ such that for each prime power $q$ and each $\bar{b}\in\mathbb{F}_{q}^{|\bar{y}|}$ either $\phi(\mathbb{F}_{q}^{|\bar{x}|};\bar{b})$ is empty or there exists $(d,\mu)\in D$ such that
\begin{equation}
\label{eq:2}
\big||\phi(\mathbb{F}_{q}^{|\bar{x}|};\bar{b})|-\mu q^{d}\big|=O(q^{d-\frac{1}{2}})\quad\text{as $q\longrightarrow\infty$}.
\end{equation}
In this case, $d$ is known as the `dimension' and $\mu$ as the `measure' of the set $\phi(\mathbb{F}^{|\bar{x}|}_{q};\bar{b})$.
Moreover, the set of $\bar{b}$ such that \eqref{eq:2} holds for a specified $(d,\mu)$  is uniformly definable by an $L_{\mathrm{ring}}$-formula without parameters.

With a change of perspective, in \cite{MS08} Macpherson and Steinhorn turned these asymptotic results into a definition: Roughly speaking, a class $\mathcal{C}$ of finite $L\/$-structures is called a {\em $1$-dimensional asymptotic class} if it satisfies the Chatzidakis--van den Dries--Macintyre Theorem. In \cite{Elwes07}, Elwes generalised this further to study `$N$-dimensional asymptotic classes'; see \autoref{eg:N-dim} below for details. By a theorem of Ryten \cite[Theorem 1.1.1]{ryten}, for any fixed Lie type $\tau$ (possibly twisted) the collection of all finite simple groups of Lie type $\tau$ is an asymptotic class. 

In \cite{MS08} the authors also introduced the notion of a {\em measurable} structure -- this is  an infinite structure such that dimension-measure pairs $(d,\mu)$ can be assigned to definable sets, in the manner that follows for pseudofinite fields from the theorem of Chatzidakis, van den Dries and Macintyre. Measurability has content from the viewpoint of model-theoretic generalised stability theory -- measurable structures are supersimple of finite SU-rank. Elwes in \cite{Elwes07} (see also \cite[Proposition 3.9]{EM08}) noted among other results that any ultraproduct of an asymptotic class is measurable, and also in \cite[Proposition 6.5]{Elwes07} that any stable measurable structure is one-based. \cite{MS08} gives a range of examples of measurable structures. 

Aspects of this work were followed up by Garcia, Macpherson and Steinhorn in \cite{GMS15}. That paper considers the Hrushovski--Wagner notion of {\em pseudofinite dimension} (see \cite{hrushovski-wagner, hrush-pseud}), identifying conditions on this dimension which ensure that an ultraproduct of finite structures is supersimple (or simple, or stable) and showing that the conditions that imply supersimplicity hold for asymptotic classes. The present paper revisits some examples from \cite{GMS15}, but we do not fully explore the connections between our work here and that of \cite{GMS15}. 

The notion of {\em asymptotic class} is very restrictive, implying that ultraproducts have  finite SU-rank, and with nearly all the known examples closely related to finite fields. In addition, there are very simple examples which fail the definition because of their many-sorted nature. One such, in a language $L$ with a single unary predicate $P$, would be the collection of all finite structures $(M,P)$. This is not an asymptotic class since $P$ can pick out an arbitrarily sized subset of $M$, but this is the only obstruction.

In this paper we develop a considerably broader framework. We still consider classes of finite structures, and still impose that  for any formula $\phi(\bar{x},\bar{y})$ determining a family of definable sets in each structure parameterized by $\bar{y}$, there is a uniform finite bound on the number of (approximate) sizes of these sets in each structure, together with a corresponding definability clause, analogous to definability of Morley rank or degree. In the definition of a {\em multidimensional asymptotic class} (\mac{}) of finite structures, we allow different parts of the finite structures -- e.g.\ sorts, or definable sets, or coordinatising geometries -- to vary independently, and we do not specify the form of the functions determining cardinalities. We obtain extra model-theoretic information, e.g.\ supersimplicity of the limit theory, when these functions are known. We separate out the regime of a multidimensional {\em exact} class (\mec{}) when the cardinalities of definable sets are given {\em exactly}, rather than just asymptotically. When the cardinalities of definable sets are determined by polynomials, we talk of a {\em polynomial} \mac{} (or \mec{}).
We also develop the notion of a {\em generalised measurable} infinite structure, the appropriate analogue of a measurable structure, and draw connections to \mac{}s and \mec{}s. 

We now summarise the main results, with fuller definitions and more detailed and precise statements appearing later in the paper. 

\begin{thrm}\label{interp}
Let $\mathcal{C}$ be a class of finite structures.
\begin{enumerate}[(i)]
\item (One variable criterion, {\rm\autoref{thm:projection.lemma}}) If all formulas $\phi(x,\bar{y})$ satisfy the definition of a \mac{} (or \mec{}) for $\mathcal{C}$   then so do all formulas $\phi(\bar{x},\bar{y})$, that is, 
$\mathcal{C}$ is a \mac{} (respectively \mec{}). 
\item (Interpretability, {\rm\autoref{thminterpmac}(i)}) If the class $\mathcal{D}$ is uniformly interpretable in 
$\mathcal{C}$ and $\mathcal{C}$ is a m.a.c (respectively \mec{}) then $\mathcal{D}$ is a weak \mac{} (respectively weak \mec{}). 
\item (Bi-interpretability,  {\rm\autoref{thminterpmac}(ii)}) In (ii), the word `weak' can be dropped if $\mathcal{C}$ and $\mathcal{D}$ are uniformly bi-interpretable.
\end{enumerate}
\end{thrm}

The next theorem gives two key motivating examples.  

\begin{thrm}\label{exquiver}~
\begin{enumerate}[(i)]
\item {\rm (\autoref{quiver-theorem}.)} Let $Q$ be a quiver of finite representation type, and let $\mathcal{C}(Q)$ be the set of all finite structures
$(F, FQ, M)$, where $F$ is a finite field, $FQ$ is the path algebra of $Q$ over $F$, and $M$ is a finite $FQ$-module, viewed in a 3-sorted language where  $F$ and $FQ$ both carry (copies of) the language of rings, $M$ carries the language of groups, and there are function symbols for the maps $F \times FQ \longrightarrow FQ$, $F \times M \longrightarrow M$, $FQ \times M \longrightarrow M$. Then $\mathcal{C}(Q)$ is a weak polynomial \mac{}.
\item {\rm (\autoref{bilinear-granger}.)}  Let $\mathcal{C}_{\bil}$ be the collection of all $L_{\bil}$ structures $(V,F)$ where $V$ is a finite-dimensional vector space over the finite field $F$, equipped with a non-degenerate alternating bilinear form $\beta$. Let $R=\mathbb{Q}(\mathbf{F})[\mathbf{V}]$.  Then $\mathcal{C}_{\bil}$ is an $R$-m.a.c..
\end{enumerate}
\end{thrm}

The following theorem provides a wide range of examples of multidimensional {\em exact} classes. 

\begin{thrm}\label{exactex}
The following are multidimensional exact classes (polynomial \mec{}s in (i), (ii) and (iv)).
\begin{enumerate}[(i)]
\item {\rm (\autoref{pillaysm}.)} For any pseudofinite strongly minimal set $M$, any class of finite structures whose non-principal ultraproducts are all elementarily equivalent to $M$.
\item {\rm(\autoref{modules}.)} The class of all finite abelian groups.
\item {\rm (\autoref{gaifman}.)} For any $d\in \mathbb{N}$, the class of all finite graphs of degree at most $d$.
\end{enumerate}
In addition
\begin{enumerate}[(i)]
\item[(iv)] {\rm (\autoref{wolf-thesis}.)} for  any finite language $L$ and $d\in \mathbb{N}$, the class of all finite $L$-structures $M$ such that $\Aut(M)$ has at most $d$ orbits on $M^4$ is a \mec{}, after expansion to a finite extension $L'\supseteq L$ that does not change the automorphism groups. 
\end{enumerate}
\end{thrm}
Note that part (iv) involves a correction to \cite[Corollary 4.4.2]{Wolf}, discussed below before \autoref{wolf-thesis}.

We note that~(i) has been extended to pseudofinite uncountably categorical structures in \cite{vanabel}. As a partial converse to (ii) above we obtain the following.

\newtheorem*{propn}{\autoref{prpn:groups-sol}}
\begin{propn}
{\it If $\mathcal{C}$ is a \mec{} of finite groups, then there is $d\in \mathbb{N}$ such that each $G\in \mathcal{C}$ has a (uniformly definable across $\mathcal{C}$) soluble radical $R(G)$ of index at most $d$, and $R(G)/F(G)$ has derived length at most $d$, where $F(G)$ is the Fitting subgroup of $G$.}
\end{propn}


We also prove the following result, and conjecture that it holds for all finite relational languages (the right-to-left direction follows from Theorem~\ref{exactex}(iv) and the work of Lachlan). 
\begin{thrm} {\rm (\autoref{homog1}} and {\rm \autoref{homog2}.)} \label{conj-graphs}
Let $M$ be a countably infinite homogeneous graph. Then $M$ is elementarily equivalent to an ultraproduct of a \mec{} if and only if it is stable. 
\end{thrm}

Regarding generalised measurability, we introduce the notion of (totally ordered) measuring semiring and show in \autoref{monomialise} that any measuring semiring may be replaced by a homomorphic image which is a `monomial' measuring semiring $\mathbb{R}\mon{D}$ consisting of monomials $r{Z}^d$ for $d\in D$ and $r\in \mathbb{R}^{\geq0}$. Here ${Z}$ is an indeterminate, and the set $D$ of `dimensions' of definable sets is a `tropical' semiring. This `monomialisation' has useful consequences, e.g. in ensuring that generalised measurability is inherited by $M^{{\rm eq}}$ (Proposition~\ref{meq}). The following result links generalised measurability to \mac{}s. and \mec{}s.
\begin{thrm}~
\begin{enumerate}[(i)]
\item {\rm (\autoref{macgm}.)} If $\mathcal{C}$ is a \mac{} then any ultraproduct of $\mathcal{C}$ is generalised measurable.
\item {\rm (\autoref{ring}.)} If $\mathcal{C}$ is a \mec{} then any ultraproduct of $\mathcal{C}$ is ring-measurable, that is, generalised measurable with values in an ordered ring (an integral domain). 
\end{enumerate}
\end{thrm}

We also explore model-theoretic consequences of generalised measurability, and a sample of results is the following.
A theory $T$ has the {\em strict order property} if some model  of $T$ has an interpretable partial order containing an infinite totally ordered subset.  For the notion of {\em functional unimodularity}, see Definition~\ref{func-un}. (This was called just {\em unimodularity} in \cite{Elwes07} -- certain confusions related to the usage in \cite{hrush-unimod} were clarified later in \cite{kestner-pillay}.) 

\begin{thrm}~
\begin{enumerate}[(i)]
\item {\rm (\autoref{theory}.)} Generalised measurability is preserved by elementary equivalence.
\item {\rm (\autoref{NSOP}} and {\rm \autoref{unimod}.)} If $M$ is generalised measurable then ${\rm Th}(M)$ does not have the strict order property and is functionally unimodular.
\item {\rm (\autoref{injsur}.)} If $M$ is ring-measurable then any definable function from a definable set in $M$ to itself is injective if and only if it is surjective.
\item {\rm (\autoref{supersimple}.)} If $M$ is generalised measurable with well-ordered set of `dimensions' $D$ (e.g.\ if $M$ is an ultraproduct of a polynomial \mac{}) then ${\rm Th}(M)$ is supersimple.
\end{enumerate}
\end{thrm}

Here is the overall structure of the paper. We introduce the basic definitions around multidimensional asymptotic and exact classes in Section 2 and then, referring ahead for full details, give a brief overview of some key examples. We also consider certain structures related to finite fields which we half-expected to yield examples, and show that they do not.
Section 2 also contains the proof of Theorem~\ref{interp}. The two key examples in Theorem~\ref{exquiver} are considered in Section 3, along with multisorted structures of the form $(F,V_1,\ldots,V_t)$ where $F$ is a finite field and the $V_i$ are finite-dimensional $F$-vector spaces. Multidimensional {\em exact} classes are explored in detail in Section 4, which includes results yielding Theorem~\ref{exactex}, Proposition~\ref{prpn:groups-sol}, and Theorem~\ref{conj-graphs}. The general theory around generalised measurability is developed in Sections 5 and 6, with Section 6 focussing on stability-theoretic consequences. We conclude with some open questions in Section 7.

\subsection{Notation, conventions, and model-theoretic background}
Throughout this paper,  $L$ is a first-order language, $L$-structures are usually denoted by $M$ or $N$, and `definable' means `definable with parameters'.
For a tuple $\bar{x}$ of variables and a set $A$, $A^{|\bar{x}|}$ denotes the set of $|\bar{x}|$-tuples from $A$.  For a structure $M$, let $\mathrm{Def}(M)$ denote the collection of sets in $M$ which are definable with parameters. A \em definable family \rm in $M$ is a set 
$\{\phi(M^{|\bar{x}|};\bar{b}) :\bar{b}\in M^{|\bar{y}|}\}$, for some formula $\phi(\bar{x};\bar{y})$. 

The paper makes occasional reference to concepts from generalised stability theory, in particular stability, simplicity, and NSOP${}_1$. These are increasingly broad notions of model-theoretic tameness for first order theories. For stability, simplicity and other model-theoretic background, \cite{tent-ziegler} can be used as a general source. More detail on simple theories can be found in \cite{kim} or \cite{wagner2}. The general theory around NSOP${}_1$ has been developed more recently, and \cite{ramsey} provides an introduction.

\section{Multidimensional asymptotic and exact classes}

\subsection{Key definitions} \label{section:macs}

\medskip

For a class $\mathcal{C}$ of $L$-structures and a tuple $\bar{y}$ of variables, we denote by $(\mathcal{C},\bar{y})$ the set $\big\{(M,\bar{a}):M\in\mathcal{C},\bar{a}\in M^{|\bar{y}|}\big\}$ of pairs consisting of a structure in $\mathcal{C}$ and a $\bar{y}$-tuple from that structure. A partition $\Pi$ of $(\mathcal{C},\bar{y})$ is said to be {\em $\emptyset$-definable} if for each part $\pi\in\Pi$ there exists an $L$-formula $\phi_{\pi}(\bar{y})$ without parameters such that
\begin{equation}
\phi_{\pi}(M^{|\bar{y}|}) = \big\{\bar{b}\in M^{|\bar{y}|}:(M,\bar{b})\in\pi\big\}
\end{equation}
for each $M\in\mathcal{C}$. We say that a partition is {\em finite} if it has finitely many parts. 

The following is the main definition of this paper.

\begin{definition}\label{def:mac}
Let $R$ be any set of functions $\mathcal{C}\longrightarrow\mathbb{R}^{\geq0}$. A class $\mathcal{C}$ of finite $L$-structures is an {\em $R$-multidimensional asymptotic class} (or an {\em $R$-\mac{}} for short) if for every formula $\phi(\bar{x};\bar{y})$ there is a finite $\emptyset$-definable partition $\Pi=\Pi_\phi$ of $(\mathcal{C},\bar{y})$ and an indexed set $H_{\Pi}:=\{h_{\pi}\in R:\pi\in\Pi\}$ such that for $(M,\bar{b})\in\pi$ we have
\begin{equation}\label{eq:3}
\big||\phi(M^{|\bar{x}|};\bar{b})|-h_{\pi}(M)\big|=o(h_{\pi}(M))
\end{equation}
as $|M|\longrightarrow\infty$.
The functions $h_{\pi}$ are called the {\em measuring functions} of $\phi(\bar{x};\bar{y})$. When $R$ is understood, we just say $\mathcal{C}$ that is a {\em \mac{}}.
\end{definition}

The `little $o$-notation' in \autoref{eq:3} means that for every real number $\epsilon>0$ there exists a natural number $C$ such that, for all $(M,\bar{b})\in\pi$, if $|M|>C$ then
\begin{equation}
\big||\phi(M^{|\bar{x}|};\bar{b})|-h_{\pi}(M)\big|\leq\epsilon h_{\pi}(M).
\end{equation}

\begin{remark}
The elements of $R$ are called {\em measuring functions}. Often we will assume that $R$ is closed under the addition and multiplication of functions.
\end{remark}

\begin{definition}\label{def:mec}
If in \autoref{def:mac} we have $|\phi(M^{|\bar{x}|};\bar{b})|=h_{\pi}(M)$ for all  $M\in {\mathcal C}$ and $\bar b\in M^{|\bar{x}|}$ with $(M,\bar{b})\in \pi$, then we say that $\mathcal{C}$ is an {\em exact $R$-\mac{}}, or an $R$-\mec{} (multidimensional {\em exact} class).
In this case the measuring functions may be chosen to have codomain ${\mathbb Z}^{\geq 0}$. 
\end{definition}

\begin{remark} 
We say that $\mathcal{C}$  is a {\em weak} \mac{} if \autoref{def:mac} holds but without requiring that the partition $\Pi$ is $\emptyset$-definable (or even parameter-definable).
Analogously, we say that $\mathcal{C}$ is a {\em weak \mec{}} if \autoref{def:mec} holds, but without requiring the $\emptyset$-definability of the partition $\Pi$.
We do not explore intermediate notions given by demanding that $\Pi$ is parameter-definable, or definable by formulas of a certain complexity.
\end{remark}

\begin{lemma}\label{constants}~
\begin{enumerate}[(i)]
\item Let $\mathcal{C}$ be an $R$-\mac{} (respectively $R$-\mec{}) of $L$-structures,  let $L'$ be an extension of $L$ by constants, and for each $M \in \mathcal{C}$ let $M'$ be an $L'$-expansion of $M$. Put $\mathcal{C}':=\{M':M\in \mathcal{C}\}$. Then $\mathcal{C}'$ is an $R$-\mac{} (respectively $R$-\mec{}).

\item The assertion of (i) holds with \mac{} and \mec{} replaced by weak \mac{} and weak \mec{} respectively.

\item Suppose $L'=L\cup\{c\}$ where $c$ is a constant symbol not in  $L$, and let $\mathcal{C}$ be a class of $L$-structures. Let $\phi(x)$ be an $L$-formula and suppose that for $M\in \mathcal{C}$, $\Aut(M)$ is transitive on $\phi(M)$. Let $\mathcal{C}'$ consist of $L'$-expansions of members of $\mathcal{C}$ with $c$ interpreted by a realisation of $\phi$. If $\mathcal{C}'$ is a m.a.c. (respectively m.e.c.), then so is $\mathcal{C}$. 
\end{enumerate}
\end{lemma}

\begin{proof} This is immediate. In (iii), we keep the definability clause in the definition of m.a.c./m.e.c., since any use of $c$ in a defining formula can be replaced by a quantifier relativised to $\phi(M)$. 
\end{proof} 

For an initial sense of the content of the $R$-\mac{} definition, note the following. 

\begin{lemma}\label{nonmac}
Let $\mathcal{C}$ be a class of finite structures, and suppose there is a formula $\phi(\bar{x},\bar{y})$ and $r\in {\mathbb R}$ with $r>1$ such that the following holds: For all $n\in {\mathbb N}$ there is $M_n \in \mathcal{C}$ and $\bar{a}_1,\ldots,\bar{a}_n \in M_n^{|\bar{y}|}$ such that $|\phi(M_n^{|\bar{x}|},\bar{a}_{i+1})| > r|\phi(M_n^{|\bar{x}|},\bar{a}_i)|$ for each $i=1,\ldots,n-1$. Then $\mathcal{C}$ is not a weak $R$-\mac{} for any $R$.
\end{lemma}

\begin{proof} Under the given assumption, put $X_i(M_n):=\phi(M_n^{|\bar{x}|},\bar{a}_i)$, and let $\epsilon>0$.
By the pigeon-hole principle, if the assumption holds then there are some $h_\pi \in R$ and for arbitrarily large $n$ distinct $i_n,j_n$ such that $|X_{j_n}(M_n)|>r|X_{i_n}(M_n)|$ and
\begin{equation}\label{chain} 
\big||X_{i_n}(M_n)|-h_\pi(M_n)\big| \leq\epsilon h_\pi(M_n) \mbox{~~and}
\end{equation}
\begin{equation}
\big||X_{j_n}(M_n)|-h_\pi(M_n)\big| \leq \epsilon h_\pi(M_n).
\end{equation} Thus by the triangle inequality
$\big| |X_{j_n}(M_n)|-|X_{i_n}(M_n)|\big|\leq 2\epsilon h_\pi(M_n)$, so $(r-1)|X_{i_n}(M_n)|\leq 2\epsilon h_\pi(M_n)$. Since
 $r$ is fixed and $\epsilon$ is arbitrarily small, this contradicts \eqref{chain}. 
\end{proof}

\begin{remark} \label{nonex} It follows immediately from \autoref{nonmac}, applied to the formula $x<y$, that the class of finite total orders is not a weak \mac{}. In fact, by \autoref{mac-nsop}, no weak \mac{} can have an ultraproduct with the strict order property. Likewise, the class of all finite structures consisting of a set equipped with an equivalence relation $E$ is not a weak \mac{} (consider the formula $Exy$). However, given a fixed bound on the class size or number of classes one obtains a weak \mac{}. For a stronger version of this last statement, see \autoref{examples-macs}(2)(c).
\end{remark}


\subsection{Examples of \mac{}s and \mec{}s}

\medskip

There are several special cases of m.a.c.s, corresponding to particular choices of the set $R$ of functions, which are worth picking out in more detail.

\begin{example}
\label{eg:N-dim}
As indicated in the introduction, the above definition, \autoref{def:mac}, generalises the notions of {\em $1$-dimensional asymptotic class} and {\em $N$-dimensional asymptotic class} introduced in \cite{MS08} and \cite{Elwes07}, respectively. We can reformulate the definition from \cite{Elwes07} in our framework as follows. For $N\in\mathbb{N}$, we let $\mathbb{R}M^{\frac{1}{N}}$ be the set of functions $\mathcal{C}\longrightarrow\mathbb{R}^{\geq0}$ given by \[m_{r,i}:M\longmapsto r|M|^{\frac{i}{N}}\] for some $r\in\mathbb{R}^{\geq0}$ and some $i\in\mathbb{N}$. A class $\mathcal{C}$ is an $N$-dimensional asymptotic class (in the sense of \cite{Elwes07}) if and only if it is an $\mathbb{R}M^{\frac{1}{N}}$-\mac{} with the additional requirement that, for every formula $\phi(\bar{x};\bar{y})$ with $\bar{x}$ an $m$-tuple of variables, the corresponding measuring functions are of the form $m_{r,i}$ for $\frac{i}{N}\leq m$. In this context, if $m_{r,i}$ is the measuring function corresponding to a definable set, then we say that $r$ is the {\em measure} and that $\frac{i}{N}$ is the {\em dimension} of that set.

The original notion of a $1$-dimensional asymptotic class (in the sense of \cite{MS08}) is essentially the special case of the above with $N=1$. By the main theorem of \cite{CvdDM92}, the class of all finite fields is a 1-dimensional asymptotic class. Furthermore,  if $p$ is prime and $m,n\in {\mathbb N}$ with $(m,n)=1$ and $m>1$, $n\geq 1$, then let $\mathcal{C}_{m,n,p}$ be the collection of all difference fields $(\mathbb{F}_{p^{kn+m}}, x\longmapsto x^{p^k})$. Then by Theorem 3.5.8 of \cite{ryten}, $\mathcal{C}_{m,n,p}$ is a 1-dimensional asymptotic class, and Ryten shows that it follows that the collection of all finite simple groups of any fixed Lie type is an asymptotic class. In  the original definition of {\em 1-dimensional asymptotic class}, and in the results of \cite{CvdDM92} on finite fields and of \cite{ryten} on $\mathcal{C}_{m,n,p}$, there is stronger information on the error term -- it takes the form  given in \cite{CvdDM92}.

\end{example}

\begin{definition}\label{polynomialmac}~
\begin{enumerate}[(i)]
\item If $\mathcal{C}$ is an $R$-\mac{} then we say that the $L$-formula $\delta(\bar{x},\bar{y})$ is {\em balanced} if for every $\epsilon\in \mathbb{R}^{>0}$ there is $C \in \mathbb{N}$ such that for any $M\in \mathcal{C}$ with $|M|>C$ and any $\bar{a},\bar{b}\in M^{|\bar{y}|}$ with $\delta(M^{|\bar{x}|},\bar{a})$ and $\delta(M^{|\bar{x}|},\bar{b})$ nonempty, we have
\[\big| |\delta(M^{|\bar{x}|},\bar{a})|-|\delta(M^{|\bar{x}|},\bar{b})|\big|<\epsilon |\delta(M^{|\bar{x}|},\bar{b})|.\]
We say $\delta(\bar{x},\bar{y})$ is {\em exactly balanced} if $|\delta(M^{|\bar{x}|},\bar{a})|=|\delta(M^{|\bar{x}|},\bar{b})|$ whenever $\bar{a},\bar{b}\in M^{|\bar{y}|}$ with $\delta(M^{|\bar{x}|},\bar{a})$ and $\delta(M^{|\bar{x}|},\bar{b})$ are nonempty.
\item Suppose in \autoref{def:mac} that $\mathcal{C}$ is an $R$-\mac{} satisfying the following condition: There are balanced  $L$-formulas $\delta_1(\bar{x}_1,\bar{y}_1),\ldots, \delta_k(\bar{x}_k,\bar{y}_k)$ and $N_1,\ldots,N_k \in \mathbb{N}^{>0}$, and for each $h\in R$ a polynomial $P_h\in \mathbb{R}[X_1,\ldots,X_k]$,  such that for every $M \in \mathcal{C}$ there are $\bar{a}_1,\ldots,\bar{a}_k$ in $M$ with $ \delta_i(M^{|\bar{x}_i|},\bar{a}_i)\neq \emptyset$ for each $i$, such that 
$$h(M)=P_h(|\delta_1(M^{|\bar{x}_1|},\bar{a}_1)|^{\frac{1}{N_1}},\ldots, |\delta_k(M^{|\bar{x}_k|},\bar{a}_k)|^{\frac{1}{N_k}}).$$ Then we say that $\mathcal{C}$ is a {\em polynomial} $R$-\mac{}, or {\em polynomial \mac{}}. The notion of {\em polynomial \mec{}} is defined similarly, but with `balanced' replaced by `exactly balanced'.
\end{enumerate}
\end{definition}

\begin{remark}~
\begin{enumerate}[1.]
\item In view of \autoref{thminterpmac} we sometimes slightly extend the above terminology, allowing the formulas $\delta(\bar{x},\bar{y})$ to be $L^{{\rm eq}}$-formulas using quotients by equivalence relations uniformly $\emptyset$-definable across $\mathcal{C}$. The point here is that in the key example of envelopes of a smoothly approximable structure (see \autoref{wolf-thesis}), the above functions $h$ will be polynomials in the cardinalities of certain `Lie geometries' which may live in $M^{{\rm eq}}$ but are uniformly definable across the class and could be added as additional sorts to the members of $\mathcal{C}$. We do not labour this point. 

In all the examples of polynomial \mac{}s or \mec{}s which we consider in this paper, the polynomials are over ${\mathbb Q}$ (and in some cases over ${\mathbb Z}$). 

\item The condition that $\delta$ is balanced is easily arranged in a \mac{}. Let $\Pi$ be the $\emptyset$-definable partition associated with $\delta(\bar{x},\bar{y})$, let $\pi\in \Pi$, and let $\phi_\pi(\bar{y})$ be the corresponding formula defining $\pi$. Then the formula $\delta'(\bar{x},\bar{y})$ of form
$\delta(\bar{x},\bar{y})\wedge \phi(\bar{y})$ is balanced.

\end{enumerate}
\end{remark}

\begin{example}
\label{eg:natural.mac}

We explain how the notion of {\em polynomial $R$-\mac{}} stems from \autoref{eg:N-dim}. The elements of $\mathbb{R}M^{\frac{1}{N}}$ are monomial functions of $|M|^{\frac{1}{N}}$. We may recast this by thinking of these functions as `formal' monomials in a new variable $X$. These formal monomials are then represented as functions $\mathcal{C}\longrightarrow\mathbb{R}^{\geq0}$ by mapping $X$ to $|M|^{\frac{1}{N}}$. In particular, any asymptotic class is a polynomial \mac{} In many cases one can dispense with the exponents $\frac{1}{N}$. For example, for the 3-dimensional asymptotic class of groups ${\rm SL}_2(q)$ we may choose the formula $\delta(x,\bar{y})$ to define the (1-dimensional) group of upper unitriangular matrices. 

In our original approach we  generalised this viewpoint to consider `polynomial' functions in several variables with non-negative real exponents (or just positive integer exponents). Each of the  variables is mapped to the cardinality of a definable set (or in some cases a sort) in the finite model. We allow $L$ to be multi-sorted. 
For each sort $s$ 
in $L$ 
we let $X_{s}$ be a new variable symbol. Let $R_{L}$ be the field of fractions of the ring
\[\mathbb{R}[X_{s}^{\mathbb{R}^{\geq0}}:s\text{ is a sort of }L].\]
We have not developed this extra generality (beyond the notion of polynomial m.a.c) since it does not  seem to be forced by natural examples.

\end{example}


\begin{example}\label{examples-macs}
The following are examples of \mac{}s and \mec{}s which are discussed in this paper, and which (apart from the first) extend beyond the earlier notion of asymptotic class. The intention here, and in Subsection 2.3, is to help the reader build an intuition for the key concepts of \mac{} and \mec{}, before the lengthy proofs of Theorems~\ref{thm:projection.lemma} and \ref{thminterpmac}. The more intricate examples in (c) and (d) below are worked out in more detail in Section 3 below, and the examples of \mec{}s are mostly given in detail in Section 4. 

\begin{enumerate}[(1)]
\item Examples of \mac{}s

\begin{enumerate}[(a)]

\item Any asymptotic class is a \mac{}, as noted above. In particular, the collection of finite fields is a 1-dimensional asymptotic class so forms a \mac{}. As noted in \cite[Example 3.4]{MS08}, another example of a 1-dimensional asymptotic class is the class of {\em Paley graphs} $P_q$. Here $q$ is a prime power with $q \equiv 1 ~~({\rm mod}~~ 4)$, and $P_q$ has as vertex set the finite field ${\mathbb F}_q$, with vertices $a,b$ adjacent if and only if $a-b$ is a square. The key fact here -- see \cite[Theorem 13.10]{bollobas} -- is that if $U,W$ are disjoint sets of vertices of $P_q$ with 
$m:=|U\cup W|$, and $v(U,W)$ denotes the number of vertices not in $U \cup W$ adjacent to all vertices of $U$ and to none of $W$, then
\[
|v(U,W)-2^{-m}q| \leq \tfrac{1}{2}(m-2+2^{-m+1})q^{\frac{1}{2}} + \tfrac{m}{2}.
\]
Any non-principal ultraproduct of distinct Paley graphs is elementarily equivalent to the random graph. There is no \mec{} of Paley graphs -- see 2(d) below.

\item The collection ${\mathcal V}_F$ of 2-sorted structures $(V,F)$ where $V$ is a finite vector space over a finite field $F$ is a polynomial \mac{}. Ultraproducts are supersimple but may have SU-rank $\omega$.  The functions in the set $R$ are polynomials in two variables over ${\mathbb Q}$. See  \cite[Theorem 4.1]{GMS15}, or \autoref{GMSgen} for a more general result.

\item The collection of 2-sorted structures $(V,F, \beta)$ where $V$ is a finite vector space over a finite field $F$, and $\beta\colon V\times V \longrightarrow F$ is a non-degenerate alternating bilinear form, is a \mac{}. An ultraproduct where the field is infinite and the vector space is infinite-dimensional will not have simple theory but will be NSOP${}_1$. The functions in $R$ are rational in $|V|$ and $|F|$. See Section 3.3 for details.

\item Let $Q$ be a quiver of finite representation type. Then the collection of all 3-sorted structures $({\mathbb F}_q, {\mathbb F}_q Q, V)$ is a weak \mac{}, where ${\mathbb F}_q$ is a finite field, ${\mathbb F}_q Q$ is the corresponding path algebra, and $V$ is a finite
${\mathbb F}_q Q$-module. Details are given in Section~\ref{finiterep}. Again, the corresponding functions are polynomials in several variables over ${\mathbb Q}$, the variables corresponding to the field and to the indecomposable ${\mathbb F}Q$-modules. 

\item Fix a positive integer $d$ and let $L:=L_{{\rm rings}} \cup\{P_1,\ldots,P_d\}$ where the $P_i$ are unary predicates. Consider the collection $\mathcal{C}_d$ of all finite residue rings ${\mathbb Z}/n{\mathbb Z}$, where $n$ has form $n=p_1^{l_1} \cdot \ldots \cdot p_d^{l_d}$, where $p_1<\ldots < p_d$ are distinct primes and $0\leq l_i \leq d$. Note that the number of prime divisors of $n$ and the exponents of these primes are bounded, but not the primes themselves. View each member of $\mathcal{C}_d$ as an $L$-structure, with $P_i$ picking out the canonical subring of form ${\mathbb Z}/p^{l_i}{\mathbb Z}$. Then by \cite[Proposition 3.3.4]{bello}, $\mathcal{C}_d$ is a polynomial \mac{}, and it is easily seen that the collection of reducts to $L_{{\rm rings}}$ is a weak \mac{}. (The statement in \cite{bello} is slightly different, since there the exponents are fixed and not just bounded, but it is easily checked that the above assertions hold.) 

\end{enumerate}

\item Examples of \mec{}s

\begin{enumerate}[(a)]

\item By a result of Pillay (see \autoref{pillaysm} below) if $M$ is a pseudofinite strongly minimal set, then there is a polynomial \mec{} all of whose ultraproducts are elementarily equivalent to $M$, in which the corresponding functions are given by polynomials in one variable over ${\mathbb Z}$.

\item By a theorem of Wolf using earlier work of Cherlin and Hrushovski, if $L$ is any countable first order language and $d$ is a natural number, then there is finite $L'\supseteq L$ such that we may form a polynomial \mec{} in $L'$ containing, for each finite $L$-structure $M$ with at most $d$ 4-types, an expansion $M'$ of $M$ to $L'$ with the same automorphism group as $M$. See \autoref{wolf-thesis}. In particular, if $M$ is a smoothly approximable structure then the collection of all finite `envelopes' of $M$ is a \mec{} (after expanding the language without changing automorphism groups).

\item As a special case of (b), let $\mathcal{C}_d$ be the class of all structures $(M,E)$ where $E$ is an equivalence relation on $M$ with at most $d$ different sizes of equivalence classes. Then $\mathcal{C}_d$ is a polynomial \mec{}, after expansion by unary predicates picking out the union of the equivalence classes of given size.

\item It is shown in \autoref{homog2} (and \autoref{homog1}) that if $M$ is a homogeneous graph in the sense of Fra\"iss\'e, then there is a \mec{} with an ultraproduct elementarily equivalent to $M$ if and only if $M$ is stable. We conjecture (see \autoref{mec-conj}(i)) that the corresponding statement holds  for any finite relational language -- the right-to-left direction follows from Wolf's result above and earlier work of Lachlan and co-authors.

\item By \autoref{modules} (ii), the collection of all finite abelian groups is a \mec{} We do not have a clean description of the corresponding functions, apart from the special case of the collection $\mathcal{C}$ of all finite homocyclic groups, that is groups of the form $({\mathbb Z}/p^n{\mathbb Z})^m$, where $p$ is prime and $m,n$ are positive integers (see Proposition 4.4.2 of \cite{GMS15}). As a very partial converse (\autoref{prpn:groups-sol}), if $\mathcal{C}$ is a \mec{} of finite groups, then there is a number $d$ such that the groups $G\in \mathcal{C}$ have a soluble uniformly definable normal subgroup $R$ whose quotient by its Fitting subgroup has derived length at most $d$, and with $|G:R|\leq d$. 

\item If $L$ is a finite relational language, and $d\in \mathbb{N}$, then the collection of all finite $L$-structures such that each element lies in at most $d$ tuples satisfying a relation is a \mec{} See \autoref{gaifman}. Again, we do not have a description of the functions giving cardinalities.

\item By \autoref{weakmec-fields}, there is no {\em weak} \mec{} consisting of arbitrarily large finite fields.

\end{enumerate}

\end{enumerate}
\end{example}

The following lemma gives a tool for constructing further \mac{}s and \mec{}s.

\begin{lemma}\label{product}
Let $\mathcal{C}_1,\ldots,\mathcal{C}_k$ be classes of finite structures in disjoint languages $L_1,\ldots, L_k$ respectively. Let $L=L_1\cup \ldots \cup L_k$, a multi-sorted language with sorts $S_1,\ldots,S_k$, with the symbols of each $L_i$ restricted to $S_i$. Let $\mathcal{C}$ be the class of $L$-structures $M=M_1 \cup \ldots \cup M_k$ (disjoint union), where each $M_i$ is the restriction of $M$ to $S_i$ and lies in $\mathcal{C}_i$. If each $\mathcal{C}_i$ is an $R_i$-\mac{} (respectively $R_i$-m.e.c), then $\mathcal{C}$ is an $R$-\mac{} (respectively $R$-\mec{}), where $R$ is the set of functions $\mathcal{C} \longrightarrow \mathbb{R}$ generated from the $R_i$ by addition and multiplication.
\end{lemma}

The construction of $R$ here  needs a little elucidation. We view $R_i$ as a set of functions $\mathcal{C}\longrightarrow \mathbb{R}$, by putting $f_i(M_1\cup \ldots \cup M_k)=f_i(M_i)$ for $f_i\in R_i$.  

\begin{proof} See \cite[Lemma 2.4.2]{Wolf} for the \mec{} case. The \mac{} argument is essentially the same.
\end{proof}

\subsection{Non-examples associated with finite fields}

The class of finite fields forms the motivating example of a $1$-dimensional asymptotic class, so it is natural to ask what additional structure may be carried by a multidimensional asymptotic class of expansions of finite fields. An obvious example of such an expansion is the 1-dimensional asymptotic class $\mathcal{C}_{m,n,p}$ from \autoref{eg:N-dim}.  Below we show that another function of interest on a finite (prime) field, the discrete logarithm, does not furnish examples of \mac{}s.

\begin{enumerate}[1.]

\item {\em Discrete logarithm in finite fields.}
The multiplicative group $\mathbb{F}_{q}^{\times}$ of a finite field $\mathbb{F}_{q}$ is cyclic.
Fix a generator $a$ of $\mathbb{F}_{q}^{\times}$.
For integers $k,l$, we have $a^{k}=a^{l}$ if and only if $k$ and $l$ are congruent modulo $q-1$.
Thus one may define the {\em discrete logarithm} with base $a$ to be the isomorphism
\begin{align*}
\log_{a}\colon\mathbb{F}_{q}^{\times}&\longrightarrow C_{q-1}\\
a^{l}&\longmapsto l,
\end{align*}
where $C_{q-1}$ is the cyclic group of order $q-1$,
perhaps identified with the additive group of integers modulo $q-1$.

There are at least two obvious ways to turn this into a first-order structure expanding $\mathbb{F}_{q}$. In the first (i) we write the codomain of the logarithm as a separate sort, and in the second (ii) we choose some identification of $\{0,...,q-2\}$ with a subset of $\mathbb{F}_{q}$, at least in the case that $q$ is prime.

\begin{enumerate}[(i)]

\item For each prime power $q$ consider the two-sorted structure $\big(\mathbb{F}_{q},C_{q-1};\log_{a}\big)$, where $\log_{a}\colon\mathbb{F}_{q}^{\times}\longrightarrow C_{q-1}$ is as above. We view the first sort $\mathbb{F}_{q}$ as a field, and the second sort $C_{q-1}$ as a group. Consider the class
\[\mathcal{C}_{1}:=\big\{(\mathbb{F}_{q},C_{q-1};\log_{a}):\text{$q$ a prime power, $a$ a generator}\big\}.\]
However there is nothing new here because the logarithm is only re-writing the multiplication in $\mathbb{F}_{q}^{\times}$. In fact these structures are uniformly bi-interpretable with finite fields, which form an asymptotic class. The class $\mathcal{C}_1$ is a polynomial m.a.c..


\item The second approach, for each {\em prime} $q$, is to identify the codomain $C_{q-1}=\{0,...,q-2\}$ with the subset $\{0,...,q-2\}\subseteq\mathbb{F}_{q}$.
This is somewhat artificial since the group structure on the codomain does not match the additive group of $\mathbb{F}_{q}$.
Consider the one-sorted structures $(\mathbb{F}_{q};\log_{a})$ and the class
\[\mathcal{C}_{2}:=\big\{(\mathbb{F}_{q};\log_{a}):\text{$q$ a prime, $a$ a generator}\big\}.\]

\begin{claim}
There is a formula $\phi(x,y)$ which uniformly defines a total ordering on each structure in $\mathcal{C}_{2}$.
\end{claim}
\begin{proof}
To keep this argument as clear as possible, for integers $x,y$, we write $\bar{x},\bar{y}$ for their residues modulo $q$. Thus $\mathbb{F}_{q}=\{\bar{0},...,\overline{q-1}\}$.
Consider the structure $(\mathbb{F}_{q};\log_{a})\in\mathcal{C}_{2}$. Then $\log_{a}(\bar{x})=\bar{y}$ if and only if $\bar{a}^{y}=\bar{x}$. We define a new operation on $\mathbb{F}_{q}^{\times}$ by:
\begin{align*}
\oplus\colon\{0,...,q-2\}\times\{0,...,q-2\}&\longrightarrow\mathbb{F}_{q}\\
(\bar{x},\bar{y})&\longmapsto \log_{a}(\bar{a}^{x}\cdot\bar{a}^{y}).
\end{align*}
This operation is definable in $(\mathbb{F}_{q};\log_{a})$. Suppose that $x+y\geq q$. Write $x+y=q+l$ where $l\in\{0,...,q-2\}$. Then $\bar{a}^{x+y}=\bar{a}^{q+l}=\bar{a}^{1+l}$, by Fermat's Little Theorem. Thus
\[\bar{x}\oplus\bar{y}
=\log_{a}(\bar{a}^{x}\cdot\bar{a}^{y})\\
=\log_{a}(\bar{a}^{x+y})\\
=\log_{a}(\bar{a}^{1+l})\\
=\overline{1+l}\\
\neq\bar{x}+\bar{y}\]
whereas if $x+y<q$ then $\bar{x}\oplus\bar{y}=\bar{x}+\bar{y}$. Since $\oplus$ is definable, we may define the set
\[\Delta:=\{(a,b)\in\mathbb{F}_{q}^{\times2}:a\oplus b=a+b\}\]
in the language with the logarithm. Now
\[\exists c\;((a,c)\in\Delta\wedge(b,c)\notin\Delta)\]
defines the order $\bar{1}<\bar{2}<...<\overline{q-1}$ on $\mathbb{F}_{q}^{\times}$. We can easily add $\bar{0}<\bar{1}$ to this ordering. This definition is uniform in $(\mathbb{F}_{q};\log_{a})$, as required.
\end{proof}

\end{enumerate}

The conclusion is that $\mathcal{C}_{1}$ is too simple to be interesting and that $\mathcal{C}_{2}$ is too wild to fit into our context. Indeed, by \autoref{nonex}, or more generally \autoref{mac-nsop}, $\mathcal{C}_2$  cannot form a weak \mac{}.

\item \emph{Lie algebras.} It is well-known that the class of finite general linear groups is model-theoretically wild if no bound is imposed on dimension (see e.g \cite[Section 6.3]{bunina}). One might hope that the corresponding class of Lie algebras would be tame, on the grounds that the Lie algebras are a linearisation of the groups. However, we observe the following.

\begin{proposition}
Let $\mathcal{C}$ be a collection of Lie algebras of the form ${\rm gl}_n(q)$, for a fixed prime power $q$ and unbounded $n$. Then $\mathcal{C}$ is not a weak \mac{}.

\end{proposition}

\begin{proof}  Let $A_r=(a_{ij})\in {\rm gl}_n(q)$ with $a_{11}=\ldots =a_{rr}=1$, and with all other $a_{ij}$ zero. It is easily checked that
if $B=(b_{ij})\in {\rm gl}_n(q)$, then $AB=BA$ if and only if $b_{ij}=0$ whenever $i\leq r$ and $j>r$ or $j\leq r$ and $i>r$. 
It follows that if $C(A_r)=\{X:[A_r,X]=0\}$, then $|C(A_r)|=q^{r^2}.q^{(n-r)^2}$. 
Thus, $\frac{|C(A_{r})|}{|C(A_{r+1})|}=q^{2n-4r-2}>q$ provided $n$ is large enough. The result now follows from \autoref{nonmac}.
\end{proof}

\item \emph{Expansions of finite fields.} If $\mathcal{C}$ is a weak \mac{} class of expansions of finite fields, then it is not possible uniformly to define small subsets of the fields of increasing size. We formulate this more precisely for finite fields $\mathbb{F}_q$ with $q\equiv 1 \mbox{~ mod~} 4$. Suppose that $\mathcal{C}$ is a class of finite structures which expand such finite fields, and that there is a formula $\phi(x,\bar{y})$ such that for every $d\in \mathbb{N}$ there is a structure $M\in \mathcal{C}$ of size $q>d$ and $\bar{a}\in M^{|\bar{y}|}$ satisfying that if $m:=|\phi(M,\bar{a})|$, then $d\leq m$ and
\[
2^{-m}q> \tfrac{1}{2}(m-2+2^{-m+1})q^{\frac{1}{2}} + \tfrac{m}{2}.
\]
Let $X:=\phi(M,\bar{a})$. Then for any subset $S$ of $X$, by the statement from \cite{bollobas} mentioned in \autoref{examples-macs}(1)(a), there is $a\in M$ such that for all $b\in X$, $a-b$ is a square (in the field structure on $M$) if and only if $b\in S$. Thus, the subsets of $X$ are uniformly definable, and it follows easily that $\mathcal{C}$ has an ultraproduct with the strict order property, so cannot be a weak \mac{} by \autoref{mac-nsop} below.

In particular, it follows that the collection of all pairs of finite fields is not a weak m.a.c, and hence, by considering the fixed field,  the collection of all finite difference fields is not a weak \mac{} -- some restrictions as in the definition of $\mathcal{C}_{m,n,p}$ are needed. Such phenomena, e.g. undecidability of the theory of pairs of finite fields, are well-known – see for example Section 4 of \cite{Chatzidakis97}.

\end{enumerate}

\subsection{Projection Lemma}

Here we use  a fibering argument to show that to check whether a class of structures is a \mac{} or \mec{}, it suffices to consider formulas $\phi(x,\bar{y})$ where $x$ is a single variable.  We argue as in Theorem 2.1 from \cite{MS08} and Lemma 2.2 from \cite{Elwes07} but for the $R$-\mac{} context, where $R$ is a set of functions $\mathcal{C}\longrightarrow\mathbb{R}^{\geq0}$. Let $\langle R\rangle$ denote the ring generated by $R$ under the usual addition and multiplication operations for real-valued functions.

\begin{theorem}\label{thm:projection.lemma}~
\begin{enumerate}[(i)]

\item Let $\mathcal{C}$ be a class of $L$-structures. Suppose that $\mathcal{C}$ satisfies the definition of an $R$-\mac{} (\autoref{def:mac}) for formulas $\phi(x;\bar{y})$ where $x$ is a singleton. Then $\mathcal{C}$ is an $\langle R\rangle$-\mac{}.

\item \cite[Lemma 2.3.1]{Wolf} The assertion of (i) holds with \mec{}s in place of \mac{}s.

\end{enumerate}
\end{theorem}

\begin{proof}
We just prove (i), which is a bit more intricate than (ii) (given in \cite{Wolf}). 

We proceed by induction on the length of $\bar{x}$.
The base case is our assumption.
Consider a formula $\phi(\bar{x}y;\bar{z})$, with $\bar{x}$ a non-empty tuple and $y$ a single variable.
By our inductive hypothesis, we may assume that the definition applies to $\phi(\bar{x};y\bar{z})$.
Let $\Pi$ be the finite partition of $(\mathcal{C},y\bar{z})$ corresponding to $\phi(\bar{x};y\bar{z})$, let $H_{\Pi}=\{h_{\pi}\in \langle R\rangle:\pi\in\Pi\}$ be the corresponding measuring functions and let 
$\Gamma_{\Pi}=\{\gamma_{\pi}(y\bar{z}):\pi\in\Pi\}$ be the indexed set of formulas 
which define $\Pi$.

For each $\pi\in\Pi$, we examine the family of sets defined by $\gamma_{\pi}(y;\bar{z})$. We apply our hypothesis to obtain a finite partition $\Psi_{\pi}$ of $(\mathcal{C},\bar{z})$, corresponding measuring functions $L_{\Psi_{\pi}}=\{l_{\psi}\in R:\psi\in\Psi_{\pi}\}$, and an indexed set of formulas 
$\Delta_{\Psi_{\pi}}=\{\delta_{\psi}(\bar{z}):\psi\in\Psi_{\pi}\}$ that define $\Psi_{\pi}$.

Each $\Psi_{\pi}$ is a partition of $(\mathcal{C},\bar{z})$. For each choice function $f\colon\Pi\longrightarrow\bigcup_{\pi\in\Pi}\Psi_{\pi}$, put 
\[\Omega_{f}:=\bigcap_{\pi\in\Pi}f(\pi)\]
and let $\Omega:=\{\Omega_{f}:f\text{ a choice function}\}$. Then $\Omega$ is the unique coarsest partition of $(\mathcal{C},\bar{z})$ such that $\Omega$ refines $\Psi_{\pi}$ for each $\pi\in\Pi$. If we view elements of $\prod_{\pi\in\Pi}\Psi_{\pi}$ as choice functions, we see that there is a bijection between $\prod_{\pi\in\Pi}\Psi_{\pi}$ and $\Omega$ given by $f\longmapsto\Omega_{f}$. 

Since each $\Psi_{\pi}$ is finite and $\Pi$ is finite, $\Omega$ is finite. Furthermore, $\Omega$ is definable: For $\Omega_{f}\in\Omega$, write $\alpha_{f}(\bar{z})$ for the formula 
\[\bigwedge_{\pi\in\Pi}\delta_{f(\pi)}(\bar{z}).\]
Then for each $(M,\bar{c})\in(\mathcal{C},\bar{z})$ we have that 
\[(M,\bar{c})\in\Omega_{f}\hbox{\ if and only if\ } M\models\alpha_{f}(\bar{c}).\]
For each $\pi\in\Pi$, let $\chi_{\pi}(\bar{x}y;\bar{z}):=\phi(\bar{x}y;\bar{z})\wedge\gamma_{\pi}(y;\bar{z})$ and note that $\phi(M^{|\bar{x}y|};\bar{c})=\bigsqcup_{\pi\in\Pi}\chi_{\pi}(M^{|\bar{x}\bar{y}|};\bar{c})$.

Let $\Omega_{f}\in\Omega$. We aim to understand the size of the set $\phi(M^{|\bar{x}y|};\bar{c})$ for $(M,\bar{c})\in\Omega_{f}$ as $|M|\longrightarrow\infty$. 

Let $\pi\in\Pi$. We abbreviate $h_{\pi}:=h_{\pi}(M)$ and $l_{f(\pi)}:=l_{f(\pi)}(M)$. Let $\epsilon>0$. By our inductive hypothesis, there exists $N_{\pi}\in\mathbb{N}$ such that for all $(M,\bar{c})\in f(\pi)$ with $|M|\geq N_{\pi}$ we have 
\[\big||\gamma_{\pi}(M;\bar{c})|-l_{f(\pi)}\big|<\epsilon l_{f(\pi)},\]
and for all $b\in M$ such that $(M,b\bar{c})\in\pi$ -- that is, 
$M\models\gamma_{\pi}(b\bar{c})$ -- we also have
\[\big||\phi(M^{|\bar{x}|};b\bar{c})|-h_{\pi}\big|<\epsilon h_{\pi}.\]
As $\chi_{\pi}(M^{|\bar{x}y|};\bar{c})$ is fibered over $\gamma_{\pi}(M;\bar{c})$, we have 
\[|\chi_{\pi}(M^{|\bar{x}y|};\bar{c})|
=\sum_{b\in\gamma_{\pi}(M;\bar{c})}|\phi(M^{|\bar{x}|};b\bar{c})|.\]
Thus
\[(1-\epsilon)^{2}h_{\pi}l_{f(\pi)}<|\chi_{\pi}(M^{|\bar{x}y|};\bar{c})|
<(1+\epsilon)^{2}h_{\pi}l_{f(\pi)}.\]
Since $(1-\epsilon)^{2},(1+\epsilon)^{2}\longrightarrow1$ as $\epsilon\longrightarrow0$, 
we have 
\[\big||\chi_{\pi}(M^{|\bar{x}y|};\bar{c})|-h_{\pi}l_{f(\pi)}|=o(h_{\pi}l_{f(\pi)}).\]

Finally, put $m_{f}:=\sum_{\pi\in\Pi}h_{\pi}l_{f(\pi)}$. Then 
\[\big||\phi(M^{|\bar{x}y|};\bar{c})|-m_{f}\big|=o(m_{f})\]
for $(M,\bar{c})\in\Omega_{f}$ as $|M|\longrightarrow\infty$, as required.
\end{proof}

\begin{remark} \label{onvevarpolymac}
\autoref{thm:projection.lemma} holds if `$R$-\mac{}' (respectively `$R$-\mec{}') is replaced by `polynomial $R$-\mac{}' (respectively `polynomial $R$-\mec{}') throughout. The understanding here is that the same formulas $\delta_1,\ldots,\delta_k$ (as in 
\autoref{polynomialmac})
are used for the 1-variable condition and the general condition. The proof is essentially as above. This is noted in a special case in the proof of Theorem 4.1 of \cite{GMS15}. 
\end{remark}



\subsection{Interpretability and bi-interpretability}

We prove two results from \cite{WolfPhD}. We assume familiarity with the usual model-theoretic notion of an $L'$-structure $N$ being $A$-interpretable in an $L$-structure $M$. This means essentially that $N$ is isomorphic to a structure with domain a quotient by an $A$-definable equivalence relation of an $A$-definable subset of a cartesian power of $M$, with the relations of $N$ induced by $A$-definable relations of $M$.

The structures $M$ and $N$ are {\em bi-interpretable} if $N$ is $\emptyset$-interpretable in $M$ via an interpretation (an isomorphism) $f\colon N \longrightarrow N^*$ (so $N^*$ lives on a quotient of a power of $M$), and $M$ is $\emptyset$-interpretable in $N$ via an interpretation $g\colon M\longrightarrow M^*$. We require in addition that if $f$ induces the isomorphism $f^*\colon M^* \longrightarrow M^{**}$ where $M^{**}$ is an $L$-structure interpreted in $N^*$, and $g$ induces the corresponding isomorphism $g^*\colon N^* \longrightarrow N^{**}$, then the isomorphisms $f^* \circ g\colon M \longrightarrow M^{**}$ and $g^* \circ f\colon N \longrightarrow N^{**}$ are $\emptyset$-definable in $M$ and $N$ respectively. Slightly weakening this, we say that $N$ is {\em weakly bi-interpretable in $M$} if we drop the assumption that $f^*\circ g$ is definable in $M$ but keep that $g^*\circ f$ is $\emptyset$-definable in $N$.

Let $\mathcal{C}$ be a class of $L$-structures and $\mathcal{C}'$ be a class of $L'$-structures. We say that $\mathcal{C}'$ is {\em parameter-interpretable} in $\mathcal{C}$ if there is an injection $\alpha\colon \mathcal{C}' \longrightarrow \mathcal{C}$ such that each $N\in \mathcal{C}'$ is parameter-interpretable in $\alpha(N)$ {\em uniformly}, i.e.\ through fixed formulas giving the interpretations. Likewise, we say that $\mathcal{C}$ and $\mathcal{C}'$ are {\em bi-interpretable} if there is a bijection $\alpha\colon \mathcal{C}'\longrightarrow \mathcal{C}$ such that each $N\in \mathcal{C}'$ is bi-interpretable with $\alpha(N)$, without use of parameters and uniformly as $N$ ranges through $\mathcal{C}'$. Here we say that $\mathcal{C}'$ is weakly bi-interpretable in $\mathcal{C}$ if we just require that  each $N\in \mathcal{C}'$ is uniformly weakly bi-interpretable in $\alpha(N)$. 

In these definitions, we replace the word `interpretable' with `definable' if no quotienting is involved.

\begin{definition}\label{defnrad}
Let $R$ be a set of functions from a class $\mathcal{C}$ to $\mathbb{R}^{\geq 0}$. We define 
$\Rad$ to be the set of functions from $\mathcal{C}$ to $\mathbb{R}^{\geq 0}$ given by 
\[
\Rad:=\left\{\sum_{i=1}^n\frac{g_i}{h_i}:
g_i,h_i\in R, n\in {\mathbb N}^{>0}\text{\ and\ }h_i(M)\neq 0\text{\ for all\ }M\in\mathcal{C}\right\},
\]
where for all $M\in\mathcal{C}$ we put 
\[
\left(\sum_{i=1}^n\frac{g_i}{h_i}\right)(M ):=\sum_{i=1}^n\frac{g_i(M)}{h_i(M)},
\]
and we define $\frac{g_i(M)}{h_i(M)}:=0$ if $h_i(M)=0$. 
\end{definition}
It is convenient to introduce the notion of positive-definiteness for $R$-\mac{}s:
 
\begin{definition}Let $\Pi$ be a finite partition of a class of finite $L$-structures $(\mathcal{C},\bar{y})$ as in the context of \autoref{def:mac}\label{pos-def-defn}. For $\pi\in\Pi$ let $\pi^M:=\{\bar{a}\in M^{|\bar{y}|}:(M^{\bar{x}},\bar{a})\in \pi\}$. The measuring function $h_\pi$ associated with the formula $\phi(\bar{x},\bar{y})$ is \emph{positive-definite} if
\begin{equation}\label{pos-def-eq}
\phi(M^{|\bar{x}|},\bar{a})=\emptyset ~\text{ for all }~\bar{a}\in \pi^M \Leftrightarrow h_\pi(M)=0
\end{equation}
for all $M\in\mathcal{C}$ with $\pi^M\neq\emptyset$.
\end{definition}

Measuring functions for a weak \mac{} are eventually positive-definite, in the sense of the following lemma. 

\begin{lemma}\label{pos-def-lem}
Suppose that the class of finite $L$-structures $\mathcal{C}$ is a weak $R$-\mac{}. Let $\phi(\bar{x},\bar{y})$ be an $L$-formula and let $\Pi=\Pi_\phi$ be a partition of $(\mathcal{C},\bar{y})$ corresponding to $\phi$, with measuring functions $\{h_\pi:\pi\in\Pi\}\subseteq R$. Then for each $\pi\in\Pi$ there exists $Q_\pi\in {\mathbb N}^{>0}$ such that 
\eqref{pos-def-eq} holds for all $M\in\mathcal{C}$ with $|M|>Q_\pi$ and $\pi^M\neq\emptyset$.
\end{lemma}

\begin{proof} Let $\pi\in\Pi$ and let $n:=|\bar{x}|$.

We first prove that the left-to-right direction of \eqref{pos-def-eq} eventually holds. For a contradiction, suppose that it never holds, i.e.\ for every $Q\in {\mathbb N}^{>0}$ there exists some $M_Q\in\mathcal{C}$ with $|M_Q|>Q$ and $\pi^{M_Q}\neq \emptyset$ such that 
$\phi(M_Q^n,\bar{a})=\emptyset$ for all $\bar{a}\in \pi^{M_Q}$ but $h_\pi(M_Q)\neq 0$. Let $\epsilon=\frac{1}{2}$. Then
\[
\left| |\phi(M_Q^n,\bar{a})|-h_\pi(M_Q)\right| = |0-h_\pi(M_Q)|
= h_\pi(M_Q)
>\epsilon h_\pi(M_Q).
\]
Since this holds for all $Q\in {\mathbb N}^{>0}$, it follows that \eqref{eq:3} does not hold for $\pi$ contradicting that $\mathcal{C}$ is a weak $R$-\mac{}. So there exists $Q_{\pi 1}\in {\mathbb N}^{>0}$ above which the left-to-right direction of \eqref{pos-def-eq} holds.

We now prove that the right-to-left direction of \eqref{pos-def-eq} eventually holds. For a contradiction, suppose that for every $Q\in {\mathbb N}^{>0}$ there exists some $M_Q\in\mathcal{C}$ with 
$|M_Q|>Q$ such that the right-to-left direction of \eqref{pos-def-eq} fails for $M_Q$. Thus $h_\pi(M_Q)=0$ but there is $\bar{a}\in \pi^{M_Q}$ such that $\phi(M,\bar{a})\neq\emptyset$. Let $\epsilon=1$. Then
\[
\left| |\phi(M_Q^n,\bar{a})|-h_\pi(M_Q)\right| = |\phi(M_Q^n,\bar{a})|
> 0
=\epsilon h_\pi(M_Q).
\]
Since this holds for all $Q\in {\mathbb N}^{>0}$, clause \eqref{eq:3} does not hold for $\pi$, again contradicting that $\mathcal{C}$ is a weak $R$-\mac{}. So there exists $Q_{\pi 2}\in {\mathbb N}^{>0}$ for which the right-to-left direction of \eqref{pos-def-eq} holds.

Taking $Q_\pi:=\max\{Q_{\pi 1},Q_{\pi 2}\}$ yields the lemma.
\end{proof}

\begin{theorem}\label{thminterpmac}
Let $\mathcal{C}$ and $\mathcal{C}'$ be, respectively, classes of $L$- and $L'$-structures.
\begin{enumerate}[(i)]

\item If $\mathcal{C}'$ is parameter-interpretable in $\mathcal{C}$ and $\mathcal{C}$ is an $R$-\mac{} (resp.\ -\mec{}) in $L$, then $\mathcal{C}'$ is a weak $\Rad$-\mac{} (resp.\ -\mec{}) in $L'$.

\item If $\mathcal{C}'$ is weakly bi-interpretable in $\mathcal{C}$ and $\mathcal{C}$ is an $R$-\mac{} (resp.\ -\mec{}) in $L$, then $\mathcal{C}'$ is a  $\Rad$-\mac{} (resp.\ -\mec{}) in $L'$.

\item If $\mathcal{C}'$ is parameter-definable in $\mathcal{C}$ and $\mathcal{C}$ is a weak $R$-\mac{} (resp.\ -\mec{}) in $L$, then $\mathcal{C}'$ is a weak $R$-\mac{} (resp.\ -\mec{}) in $L'$.

\end{enumerate}
\end{theorem}

\begin{remark}\label{parameter_issue_rem}~
\begin{enumerate}[1.]

\item There is a small abuse of notation: If $\alpha\colon \mathcal{C}'\longrightarrow \mathcal{C}$ gives an interpretation, an element $s \in R$ (a function $\mathcal{C}\longrightarrow \mathbb{R}^{\geq 0}$) is identified with the corresponding function $\mathcal{C}'\longrightarrow \mathbb{R}^{\geq 0}$, where
$s(M):=s(\alpha(M))$ for $M\in \mathcal{C}'$.

\item Many of the key ideas in the proof are due to Elwes, although our approach is slightly more direct than that in \cite{Elwes07}. Elwes \cite{Elwes07} stated a slightly stronger version of (ii), assuming only what he calls $\emptyset_{\mathcal{C}'}$-bi-interpretability, that is, he only requires that no parameters from the structures in $\mathcal{C}'$ are needed. However, it would seem that both his proof and that of (ii) require the stronger hypothesis of weak bi-interpretability over $\emptyset$. We explain this point when it arises in our proof of~(ii).

\end{enumerate}
\end{remark}

We prove \autoref{thminterpmac} only for $R$-\mac{}s; the proof for $R$-\mec{}s is just a simpler version of the proof for $R$-\mac{}s, and details can by found in \cite{WolfPhD}.

\begin{proof}[Proof of \autoref{thminterpmac}.]
We use the notation and terminology given at the beginning of this section  throughout, with $M, M^*, f, g$, etc. For each $N\in \mathcal{C}'$ we assume that the isomorphic copy $N^*$ of $N$ lives (uniformly) on a quotient by a definable equivalence relation $E_N$ of a definable subset $X_N$ of some cartesian power $\alpha(N)^r$ of $\alpha(N)$. We drop the subscript $N$ if the context is clear. 

{\em Proof of~(i).} Let $\phi'(x_1,\ldots,x_n;y_1,\ldots,y_m)$ be an arbitrary $L'$-formula. We need to show that this formula satisfies the size clause \eqref{eq:3}. The overall strategy is to move into $\mathcal{C}$ via the interpretation, make size estimates there, and pull back these estimates into $\mathcal{C}'$. Making the size estimates in $\mathcal{C}$ forms the bulk of the argument. 

We first translate the $L'$-formula $\phi'$ into an $L$-formula $\phi$. Each $L'$-symbol has a uniform $L$-translation, so we replace each $L'$-symbol in $\phi'$ with its $L$-translation. We leave all boolean connectives as is and adapt quantifiers in accordance with the new variables; for example, if a variable $x$ in $\phi'$ is in the scope of a quantifier $\forall x$ and $x$ becomes $\bar{x}$ in $\phi$, then the quantifier $\forall x$ becomes $\forall\bar{x}$ in $\phi$. The resulting $L$-formula is $\phi(\bar{x}_1,\ldots,\bar{x}_n,\bar{y}_1,\ldots,\bar{y}_m)$, where $|\bar{x}_i|=|\bar{y}_i|=r$.

Observe that $E=E_N$ induces an equivalence relation on each power of $X_N$: Two tuples are equivalent if they are $E$-equivalent in each coordinate. We abuse notation and use $E$ to denote this induced equivalence relation. Since the $L$-translation of each $L'$-symbol defines an $E$-invariant subset of a power of $X_N$, we see that $\phi(\alpha(N)^{r\cdot n+r\cdot m})\subseteq X_N^{n+m}$ is a union of $E$-equivalence classes. Hence for all $N\in\mathcal{C}'$ and for all $a_1,\ldots,a_m\in N$ we have
\begin{equation}\label{isocopy}
f(\phi'(N^n,a_1,\ldots,a_m))=\phi(\alpha(N)^{r\cdot n},\tilde{f}(a_1),\ldots,\tilde{f}(a_m))/E,
\end{equation}
where $\tilde{f}$ is a choice function on the set $\{f(a):a\in N\}$, i.e.\ $\tilde{f}(a)$ is an arbitrary element of the equivalence class $f(a)$, and where
\begin{equation*}
f(\phi'(N^n,a_1,\ldots,a_m)):=\{(f(c_1),\ldots,f(c_n)):N\models\phi'(c_1,\ldots,c_n,a_1,\ldots,a_m)\}.
\end{equation*}

Note that under the assumption of  parameter-interpretability 
the $L$-translation might require translation parameters from the structures in $\mathcal{C}$; that is, the $L$-translation of $\phi'$ might actually be of the form $\phi(\bar{x}_1,\ldots,\bar{x}_n,\bar{y}_1,\ldots,\bar{y}_m,\bar{c}_N)$, where $\bar{c}_N$ is a tuple of parameters from $\alpha(N)$. However, by \autoref{constants}(i) we can extend $L$ to include constant symbols for these translation parameters, so we can ignore this issue, although it will come up again in the proof of (ii) below.

To show that $\phi'(x_1,\ldots,x_n;y_1,\ldots,y_m)$ satisfies the size clause \eqref{eq:3} we need to calculate the approximate size of $\phi'(N^n,a_1,\ldots,a_m)$. By \eqref{isocopy} we have
\begin{equation}\label{isocopy2}
|\phi'(N^n,a_1,\ldots,a_m)|=|\phi(\alpha(N)^{r\cdot n},\tilde{f}(a_1),\ldots,\tilde{f}(a_m))/E|.
\end{equation}
It thus suffices to calculate the right-hand side of \eqref{isocopy2}, using that $\mathcal{C}$ is an $R$-\mac{}.

We may safely assume that $\alpha$ is surjective, i.e.\ that $\mathcal{C}=\{\alpha(N):N\in\mathcal{C}'\}$, since $\{\alpha(N):N\in\mathcal{C}'\}$ is a subclass of $\mathcal{C}$ and thus is also an $R$-\mac{}.

We now calculate the approximate size of $\phi(\alpha(N)^{r\cdot n},\bar{b}_1,\ldots,\bar{b}_m)/E$ for varying $\alpha(N)\in\mathcal{C}$ and $(\bar{b}_1,\ldots,\bar{b}_m)\in\alpha(N)^{r\cdot m}$. We write 
$\bar{y}:=(\bar{y}_1,\ldots,\bar{y}_m)$ and $\bar{b}:=(\bar{b}_1,\ldots,\bar{b}_m)$. The equivalence classes in $\phi(\alpha(N)^{r\cdot n},\bar{b})$ are uniformly defined by the $L$-formula
\begin{equation*}
\tilde{\phi}(\bar{x}_1,\ldots,\bar{x}_n;\bar{b},\bar{d}_1,\ldots,\bar{d}_n):\equiv
	\phi(\bar{x}_1,\ldots,\bar{x}_n,\bar{b})\,\wedge
		\bigwedge_{1\leq i\leq n}E(\bar{x}_i,\bar{d}_i)
\end{equation*}
for tuples $(\bar{d}_1,\ldots,\bar{d}_n)\in X_N^n$ varying over equivalence classes. Note that the $L$-formula $E(\bar{v}_1,\bar{v}_2)$ that defines the equivalence relation $E$ on $X_N^n$ may require parameters from each $\alpha(N)\in\mathcal{C}$, but these can be subsumed into $\bar{c}_N$.

Since $\mathcal{C}$ is an $R$-\mac{}, the formula $\tilde{\phi}(\bar{x}_1,\ldots,\bar{x}_n;\bar{y},\bar{v}_1,\ldots,\bar{v}_n)$ gives rise to a finite partition $\Pi=\{\pi_1,\ldots , \pi_\ell\}$ of 
\[(\mathcal{C}, \bar{y},\bar{v}_1,\ldots,\bar{v}_n)
=\{(\alpha(N),\bar{b},\bar{d}_1,\ldots,\bar{d}_n): N\in\mathcal{C}',\bar{b}\in\alpha(N)^{r\cdot m},\bar{d}_i\in\alpha(N)^{r}\}\]
with measuring functions $h_1,\ldots,h_\ell\in R$ and defining $L$-formulas 
$$\{\theta_j(\bar{y},\bar{v}_1,\ldots,\bar{v}_n):1\leq j\leq \ell\}.$$

As $\phi$ is $E$-invariant, we assume that the $h_1,\ldots,h_\ell$ respect the equivalence relation; that is, for every $\bar{b}\in\alpha(N)^{r\cdot m}$, if $(\bar{d}_1,\ldots,\bar{d}_n)$ and $(\bar{d}_1',\ldots,\bar{d}_n')$ lie in the same equivalence class, then $(\alpha(N),\bar{b},\bar{d}_1,\ldots,\bar{d}_n)$ and $(\alpha(N),\bar{b},\bar{d}_1',\ldots,\bar{d}_n')$ belong to the same element of the partition $\Pi$. Indeed, if they do not, then since $$\tilde{\phi}(\alpha(N)^{r\cdot n},\bar{b},\bar{d}_1,\ldots,\bar{d}_n)=\tilde{\phi}(\alpha(N)^{r\cdot n},\bar{b},\bar{d}_1',\ldots,\bar{d}_n')$$ we can modify the partition so that they do.

For $j=1,\ldots, \ell$ let $Y_j(\alpha(N),\bar{b})$ denote the union of the equivalence classes in 
$\phi(\alpha(N)^{r\cdot n},\bar{b})$ that take the function $h_j$. By the assumption in the preceding  paragraph, $Y_j(\alpha(N),\bar{b})\cap Y_{j'}(\alpha(N),\bar{b})=\emptyset$ for $j\neq j'$. For each $j$ the set $Y_j(\alpha(N),\bar{b})$ is uniformly defined by the $L$-formula
\begin{equation*}
\tilde{\phi}_j(\bar{x}_1,\ldots,\bar{x}_n;\bar{y}):\equiv \phi(\bar{x}_1,\ldots,\bar{x}_n,\bar{y})\wedge\theta_j(\bar{b},\bar{x}_1,\ldots,\bar{x}_n).
\end{equation*}

\hypertarget{pi_ji_invariance} Since $\mathcal{C}$ is an $R$-\mac{}, each formula 
$\tilde{\phi}_j(\bar{x}_1,\ldots,\bar{x}_n;\bar{y})$ gives rise to a finite partition 
$\mathrm{P}_j=\{\rho_{j1},\ldots,\rho_{je_{j}}\}$ of 
$(\mathcal{C}, \bar{y})=\{(\alpha(N),\bar{b}):N\in\mathcal{C}',\bar{b}\in\alpha(N)^{r\cdot m}\}$ with corresponding measuring functions $g_{j1},\ldots,g_{je_{j}}\in R$. Again, we may assume that this partition respects the equivalence relation, i.e.\ that if $\bar{b}$ and $\bar{b}'$ lie in the same equivalence class, then $(\alpha(N),\bar{b})$ and $(\alpha(N),\bar{b}')$ lie in the same 
$\mathrm{P}_j$-class.

We now wish to show for $(\alpha(N),\bar{b})\in\rho_{jk}$ that $|Y_j(\alpha(N),\bar{b})/E|$ is approximately equal to $\frac{g_{jk}(\alpha(N))}{h_j(\alpha(N))}$. This is to be expected, since 
$g_{jk}(\alpha(N))$ is approximately equal to $|Y_j(\alpha(N),\bar{b})|$ and $h_j(\alpha(N))$ is approximately equal to the size of each equivalence class in $Y_j(\alpha(N),\bar{b})$. 

For brevity we now write $\bar{d}:=(\bar{d}_1,\ldots,\bar{d}_n)$. Let $R_N$ consist of a choice of one $\bar{d}\in X_N^n$ from each equivalence class in $X_N^n$, that is, a transversal for the equivalence relation induced by $E$ on $X_N^n$. Thus $X_N^n/E=\{\bar{d}/E:\bar{d}\in R_N\}$ and 
$|X_N^n/E|=|R_N|$. We do not claim that $R_N$ is $L$-definable. Put $S_j(\bar{b}):=R_N\cap\theta_j(\bar{b},\alpha(N)^{r\cdot n})$. For each $(\alpha(N),\bar{b})\in(\mathcal{C}, \bar{y})$ and $j=1,\ldots,\ell$ we have
\begin{equation}\label{partitionE_N}
Y_j(\alpha(N),\bar{b})=\bigcup_{\bar{d}\in S_j(\bar{b})}\tilde{\phi}(\alpha(N)^{r\cdot n},\bar{b},\bar{d}),
\end{equation}
where the union is disjoint since equivalence classes are disjoint.

Now fix $j\in\{1,\ldots,\ell\}$ and $k\in\{1,\ldots,e_j\}$. By the definitions of $\pi_j$ and $h_j$ we have 
for all $(\alpha(N),\bar{b},\bar{d})\in\pi_j$ that 
\begin{equation}\label{eqnogamma}
\left| |\tilde{\phi}(\alpha(N)^{r\cdot n},\bar{b},\bar{d})| \vphantom{\sum}- h_j(\alpha(N))\right| = o(h_j(\alpha(N)))
\end{equation}
as $|\alpha(N)|\longrightarrow\infty$. By the definitions of $\rho_{jk}$ and $g_{jk}$ we likewise have  for all $(\alpha(N),\bar{b})\in\rho_{jk}$ that  
\begin{equation}\label{eqnYj}
\left||Y_j(\alpha(N),\bar{b})| \vphantom{\sum}- g_{jk}(\alpha(N))\right|=o(g_{jk}(\alpha(N)))
\end{equation}
as $|\alpha(N)|\longrightarrow\infty$.

Now let $\epsilon>0$. Recall the $\epsilon$-$Q$ definition of little-o notation. Let $Q_1\in {\mathbb N}$ be such that \eqref{eqnogamma} holds for $\epsilon/2$ and 
$Q_2\in\mathbb{N}$ be such that \eqref{eqnYj} holds for $\epsilon/2$. Set $Q:=\max\{Q_1,Q_2\}$ and $t:=|Y_j(\alpha(N),\bar{b})/E|$. Notice that $t$ depends on $\alpha(N)$, $\bar{b}$ and $j$, which we suppress, and that $t=|S_j(\bar{b})|$.

Let $(\alpha(N),\bar{b})\in\rho_{jk}$ be such that $|\alpha(N)|>Q$.  Then
\begin{equation}\label{key1}
\begin{alignedat}{4}
\left||Y_j(\alpha(N),\bar{b})|\vphantom{\sum}-t\cdot h_j(\alpha(N))\right| 
=\,&\left|\sum_{\bar{d}\in S_j(\bar{b})}|\tilde{\phi}(\alpha(N)^{r\cdot n},\bar{b},\bar{d})|-t\cdot  h_j(\alpha(N))\right| \text{\ \ (by \eqref{partitionE_N})} & \\
=\,&\left|\sum_{\bar{d}\in S_j(\bar{b})}\left[\vphantom{\sum}|\tilde{\phi}(\alpha(N)^{r\cdot n},\bar{b},\bar{d})|- h_j(\alpha(N))\right]\right|  \text{\ \ (as $t=|S_j(\bar{b})|$)} & \\
\leq\,&\sum_{\bar{d}\in S_j(\bar{b})}\left|\vphantom{\sum}|\tilde{\phi}(\alpha(N)^{r\cdot n},\bar{b},\bar{d})|- h_j(\alpha(N))\right| & \\
\leq\,&\underbrace{\frac{\epsilon}{2}\cdot h_j(\alpha(N)) + \cdots + \frac{\epsilon}{2}\cdot h_j(\alpha(N))}_{\text{$t$ summands}}\\
=\,&t \cdot\frac{\epsilon}{2}\cdot h_j(\alpha(N))
\end{alignedat}
\end{equation}
where the penultimate step follows by applying \eqref{eqnogamma} to $Q$ and $\epsilon/2$ for each $\bar{d}\in S_j(\bar{b})$,  since $(\alpha(N),\bar{b},\bar{d})\in\pi_j$ if $\bar{d}\in S_j(\bar{b})$. 
As $(\alpha(N),\bar{b})\in\rho_{jk}$ and $|\alpha(N)|>Q$, we similarly apply \eqref{eqnYj} to $Q$ and 
$\epsilon/2$ to yield 
\begin{equation}\label{key2}
\left||Y_j(\alpha(N),\bar{b})|-g_{jk}(\alpha(N))\right|\leq  \frac{\epsilon}{2}\cdot g_{jk}(\alpha(N)).
\end{equation}
By \eqref{key1} and \eqref{key2}, applying the triangle inequality gives
\begin{equation}\label{maxeqn}
\begin{alignedat}{2}
\left| t\cdot h_j(\alpha(N)) - g_{jk}(\alpha(N))\right|
\leq &\;\frac{\epsilon}{2}\cdot t \cdot h_j(\alpha(N))+\frac{\epsilon}{2}\cdot g_{jk}(\alpha(N))\\
\noalign{\smallskip}
\leq &\;\epsilon\cdot\max\{t\cdot h_j(\alpha(N)), g_{jk}(\alpha(N))\}. 
\end{alignedat}
\end{equation}
for all $(\alpha(N),\bar{b})\in\rho_{jk}$ such that $|\alpha(N)|>Q$. 

\medskip
Let $\frac{g_{jk}}{h_j}(\alpha(N)):=\frac{g_{jk}(\alpha(N))}{h_j(\alpha(N))}$. We now wish to show that

\begin{equation}\label{littleoyj3}
\left| |Y_j(\alpha(N),\bar{b})/E| - \frac{g_{jk}}{h_j}(\alpha(N))\right| = o\left(\frac{g_{jk}}{h_j}(\alpha(N))\right)
\end{equation}
for all $(\alpha(N),\bar{b})\in\rho_{jk}$ as $|\alpha(N)|\longrightarrow\infty$. We first assume that $h_j(\alpha(N))\not=0$. The argument now divides into two cases. The adaptation needed for the case $h_j(\alpha(N))=0$ is handled afterwards. 

\medskip

\noindent \textsc{Case 1.} $t \cdot h_j(\alpha(N))\leq g_{jk}(\alpha(N))$. 

From \eqref{maxeqn} and the fact that  $t=|Y_j(\alpha(N),\bar{b})/E|$, we obtain 

\begin{equation*}
\left| |Y_j(\alpha(N),\bar{b})/E| - \frac{g_{jk}}{h_j}(\alpha(N))\right| 
\leq\epsilon \cdot \frac{g_{jk}}{h_j}(\alpha(N))
\end{equation*}
from which \eqref{littleoyj3} follows for all 
$(\alpha(N),\bar{b})\in\rho_{jk}^g:=\{(\alpha(N),\bar{b})\in\rho_{jk}:t \cdot h_j(\alpha(N))\leq g_{jk}(\alpha(N))\}$ 
as $|\alpha(N)|\longrightarrow\infty$. 
\textsc{End of Case 1.}

\medskip

\noindent \textsc{Case 2.} $g_{jk}(\alpha(N)) < t \cdot h_j(\alpha(N))$. 

From \eqref{maxeqn} we obtain 
\begin{equation*}
\left| |Y_j(\alpha(N),\bar{b})/E| - \frac{g_{jk}(\alpha(N))}{h_j(\alpha(N))}\right| 
\leq\epsilon \cdot |Y_j(\alpha(N),\bar{b})/E|.
\end{equation*}
Therefore
\begin{equation*}
\left| |Y_j(\alpha(N),\bar{b})/E| - \frac{g_{jk}}{h_j}(\alpha(N))\right| = o\left(|Y_j(\alpha(N),\bar{b})/E|
\right)
\end{equation*}
for all $(\alpha(N),\bar{b})\in\rho_{jk}^h:=\{(\alpha(N),\bar{b})\in\rho_{jk}:g_{ji}(\alpha(N)) < t \cdot h_j(\alpha(N))\}$ as $|\alpha(N)|\longrightarrow\infty$. By a straightforward argument (`little-o-exchange') we have
\begin{equation*}
\left| |Y_j(\alpha(N),\bar{b})/E| - \frac{g_{jk}}{h_j}(\alpha(N))\right| = o\left(\frac{g_{jk}}{h_j}(\alpha(N))\right)
\end{equation*}
for all $(\alpha(N),\bar{b})\in\rho_{jk}^h$ as $|\alpha(N)|\longrightarrow\infty$. \textsc{End of Case 2.}

\medskip

As $\rho_{jk}=\rho_{jk}^g \cup\rho_{jk}^h$, the two cases together establish \eqref{littleoyj3} in the case that $h_j(\alpha(N))\not=0$.


\medskip

We briefly address the  modifications necessary if $h_j(\alpha(N))=0$, to avoid dividing by zero. Applying \autoref{pos-def-lem}, by taking a larger $Q$ if necessary we may assume that $h_j$ is positive-definite for all $\alpha(N)$ with $|\alpha(N)|>Q$. If $h_j(\alpha(N))=0$ then by definition we have  
$\frac{g_{jk}}{h_j}:=0$. 
Then \eqref{littleoyj3} still holds, since for $|\alpha(N)|>Q$ we have that $h_j(\alpha(N))=0$ implies $|Y_j(\alpha(N),\bar{b})|=0$ and hence
$|Y_j(\alpha(N),\bar{b})/E|=0$. 

We finally can proceed to calculate the approximate size of 
$\phi(\alpha(N)^{r\cdot n},\bar{b})/E$. For each $j\in\{1,\ldots,\ell\}$ we have a finite partition 
$\mathrm{P}_j:=\{\rho_{jk}:1\leq k\leq e_j\}$ of $\mathcal{C}(\bar{y})$. We use these partitions to construct a single finite partition $\Phi$ of $\mathcal{C}(\bar{y})$. 
Define
\begin{equation*}
I:=\{(k_1,\ldots,k_\ell):1\leq k_j\leq e_j,1\leq j\leq \ell\} \text{\ and\ }  \rho_{(k_1,\ldots,k_\ell)}:=\bigcap_{j=1}^\ell\rho_{jk_{j}}
\end{equation*}
 for each $(k_1,\ldots,k_\ell)\in I$.
Then $\mathrm{P}:=\{\rho_{(k_1,\ldots,k_\ell)}:(k_1,\ldots,k_\ell)\in I\}$ forms a finite partition of 
$\mathcal{C}(\bar{y})$. We now show that this partition works.

We have
\begin{equation*}
\phi(\alpha(N)^{r\cdot n},\bar{b})=\bigcup_{j=1}^\ell Y_j(\alpha(N),\bar{b}).
\end{equation*}
Since each $Y_j(\alpha(N),\bar{b})$ is a union of $E$-equivalence classes and the 
$Y_j(\alpha(N),\bar{b})$ are pairwise disjoint, they form a partition of 
$\phi(\alpha(N)^{r\cdot n},\bar{b})/E$. 
Hence
\begin{equation}\label{Yjbsum}
|\phi(\alpha(N)^{r\cdot n},\bar{b})/E|=\sum_{j=1}^\ell \left|Y_j(\alpha(N),\bar{b})/E\right|.
\end{equation}

Let $\epsilon>0$ and $(k_1,\ldots,k_\ell)\in I$. For each $j\in\{1,\ldots,\ell\}$, let 
$Q(j)\in\mathbb{N}^{>0}$ be such that \eqref{littleoyj3} holds for $\epsilon/\ell$, where we take $g_{jk}:=g_{jk_{j}}$. Set $Q':={\rm max}\{Q(j):1\leq j\leq l\}$. Then for every $(\alpha(N),\bar{b})\in\rho_{(k_1,\ldots,k_\ell)}$ with $|\alpha(N)|>Q'$, 
\begin{equation*}
\begin{alignedat}{3}
\left||\phi(\alpha(N)^{r\cdot n},\bar{b})/E|-\sum_{j=1}^\ell\frac{g_{jk_{j}}}{h_j}(\alpha(N))\right|
=\,&\left|\sum_{j=1}^\ell \left|Y_j(\alpha(N),\bar{b})/E\right|-\sum_{j=1}^\ell\frac{g_{jk_{j}}}{h_j}(\alpha(N))\right| & \text{\ \ (by \eqref{Yjbsum})}\\ 
\leq\,&\sum_{j=1}^\ell \left| |Y_j(\alpha(N),\bar{b})/E|-\frac{g_{jk_{j}}}{h_j}(\alpha(N))\right| & & 
\\
\leq\,&\sum_{j=1}^\ell\frac{\epsilon}{\ell} \cdot \frac{g_{jk_{j}}}{h_j}(\alpha(N))\\
=\,&\,\epsilon \cdot \frac{g_{jk_{j}}}{h_j}(\alpha(N))
\end{alignedat}
\end{equation*}
where the penultimate step follows by applying \eqref{littleoyj3} to $Q'$ and $\epsilon/l$ for each $j\in\{1,\ldots,\ell\}$, since $(\alpha(N),\bar{b})\in\rho_{(k_1,\ldots,k_\ell)}\subseteq\rho_{jk_{j}}$. Hence 
\begin{equation}\label{fracsumfinal}
\left| |\phi(\alpha(N)^{r\cdot n},\bar{b})/E|-\sum_{j=1}^\ell\frac{g_{jk_{j}}}{h_j}(\alpha(N))\right| = o\left(\sum_{j=1}^\ell\frac{g_{jk_{j}}}{h_j}(\alpha(N))\right)
\end{equation}
for all $(\alpha(N),\bar{b})\in\rho_{(k_1,\ldots,k_\ell)}$ as $|\alpha(N)|\longrightarrow\infty$.

We now pull everything back to $\mathcal{C}'$. Put 
\begin{equation*}
\rho'_{(k_1,\ldots,k_\ell)}:=\{(N,a_1,\ldots,a_m)\in(\mathcal{C}',y_1,\ldots, y_m):(\alpha(N),\tilde{f}(a_1),\ldots,\tilde{f}(a_m))\in\rho_{(k_1,\ldots,k_\ell)}\}.
\end{equation*} 
Then $\mathrm{P}':=\{\rho'_{(k_1,\ldots,k_\ell)}: (k_1,\ldots,k_\ell)\in I\}$ is a finite partition of 
$(\mathcal{C}',y_1,\ldots, y_m)$. By our earlier assumption that each $\rho_{jk}$ respects the equivalence relation $E$, the set $\rho'_{(k_1,\ldots,k_\ell)}$ does not depend on the choice function 
$\tilde{f}$. 
We also define 
\begin{equation*}
\frac{g_{jk}}{h_j}(N):=\frac{g_{jk}}{h_j}(\alpha(N))
\end{equation*}
for $N\in\mathcal{C}'$. 
Then \eqref{isocopy2} and \eqref{fracsumfinal} together imply for every $(k_1,\ldots,k_\ell)\in I$ that
\begin{equation*}
\left| |\phi'(N^n,a_1,\ldots,a_m)|-\sum_{j=1}^\ell\frac{g_{jk_{j}}}{h_j}(N)\right| 
= o\left(\sum_{j=1}^\ell\frac{g_{jk_{j}}}{h_j}(N)\right)
\end{equation*}
for all $(N,a_1,\ldots,a_m)\in\rho'_{(k_1,\ldots,k_\ell)}$ as $|N|\longrightarrow\infty$. Hence $\mathcal{C}'$ is a weak $\Rad$-\mac{}, completing the proof of part (i) of the theorem.

{\em Proof of~(ii).}  Following on from the proof of~(i), we need to show that the partition $\mathrm{P}'$ of $(\mathcal{C}',y_1,\ldots, y_m)$ is 
$\emptyset$-definable. To this end, let 
$\rho'_{(k_1,\ldots,k_\ell)}\in\mathrm{P}'$. 
Since $\mathcal{C}$ is an $R$-\mac{}, each $\rho_{jk}$ is 
$\emptyset$-definable and hence the intersection $\rho'_{(k_1,\ldots,k_\ell)}$ also is 
$\emptyset$-definable. 
Thus the partition $\mathrm{P}$ of $(\mathcal{C},\bar{y}_1,\ldots, \bar{y}_m)$ is 
$\emptyset$-definable. Let 
$\{\psi_{(k_1,\ldots,k_\ell)}(\bar{y}_1,\ldots, \bar{y}_m):\rho_{(k_1,\ldots,k_\ell)}\in\mathrm{P}\}$ be the set of defining $L$-formulas. Note that this is where the point of \autoref{parameter_issue_rem} comes into play. In the proof of~(i) we could ignore the translation parameters by applying \autoref{constants}. We cannot do that here, since expanding $\mathcal{C}$ by constant symbols might prevent it from being $\emptyset$-interpretable in $\mathcal{C}'$. Thus, to guarantee that each $\{\psi_{(k_1,\ldots,k_\ell)}(\bar{y}_1,\ldots, \bar{y}_m)$ is parameter-free, it appears that we need to assume that no translation parameters are required in the interpretation of 
$\mathcal{C}'$ in $\mathcal{C}$.

For brevity we now write $\rho$ and $\rho'$ for $\rho_{(k_1,\ldots,k_\ell)}$ and 
$\rho'_{(k_1,\ldots,k_\ell)}$, respectively, and $\psi$ for $\psi_{(k_1,\ldots,k_\ell)}$, since the subscript 
$(k_1,\ldots,k_\ell)$ no longer plays a role.
For all $(N,\bar{a})\in(\mathcal{C}',y_1,\ldots, y_m)$ we have
\begin{equation}\label{interpdefclause1}
\begin{alignedat}{3}
(N,\bar{a})\in\rho' & \Leftrightarrow (\alpha(N),\tilde{f}(\bar{a}))\in\rho & \text{\ \ (by the definition of $\rho'$)} \\
& \Leftrightarrow \alpha(N)\models\psi (\tilde{f}(\bar{a}))& \text{\ \ (since $\psi$ defines $\rho$)} \\
& \Leftrightarrow \alpha(N)^*\models\psi(g \circ \tilde{f}(\bar{a}))& \text{\ \ (since $g$ is an isomorphism)}.
\end{alignedat}
\end{equation}
Since $\alpha(N)^*$ is the $\emptyset$-interpretation of $\alpha(N)$ in $N$, we can find a parameter-free $L'$-translation $\psi'(\bar{y}')$ of $\psi(\bar{y})$ such that
\begin{equation}\label{interpdefclause2}
\alpha(N)^*\models\psi(g \circ \tilde{f}(\bar{a})) \Leftrightarrow N\models\psi'(\tilde{g} \circ \tilde{f}(\bar{a})),
\end{equation}
where $\tilde{g}$ is a choice function for $g$ in the same way that $\tilde{f}$ is a choice function for $f$. Since the isomorphism $g^*f\colon N \longrightarrow N^{**}$ is uniformly $\emptyset$-definable across $\mathcal{C}'$, we can find a parameter-free $L'$-formula $\psi''(y_1,\ldots,y_m)$ such that
\begin{equation}\label{interpdefclause3}
N\models\psi'(\tilde{g} \circ \tilde{f}(\bar{a})) \Leftrightarrow N\models\psi''(\bar{a}).
\end{equation}
Together \eqref{interpdefclause1}, \eqref{interpdefclause2} and \eqref{interpdefclause3} yield
\begin{equation*}
(N,\bar{a})\in\rho' \Leftrightarrow N\models\psi''(\bar{a}).
\end{equation*}
So $\rho'$ is definable, as required.

{\em Proof of~(iii).}  This is a straightforward special case of the proof of (i), with no quotienting involved, and we omit the details. 
\end{proof}

\section{Three examples of multidimensional asymptotic classes}

We discuss here three examples of (weak) \mac{}s that seem enlightening. The first, systems of vector spaces over a field, gives an example of dimensions ranging independently. It also is used in the second example, finite modules for the path algebra over a finite field for a quiver of finite representation type. The third example, vector spaces over a finite field with a bilinear form, has particular interest as ultraproducts of this class do not in general have simple theory. 

\subsection{Families of vector spaces.} The example in this subsection is a slight generalisation of the family of 2-sorted structures $(\mathbb{F},V)$ considered in Theorem 4.1 of \cite{GMS15}. We consider it partly because the extra generality seems to be needed for the quiver representations in the next subsection. Additionally, the proof in \cite{GMS15} rests on a quantifier elimination result from \cite{Granger99} which we later realised is not quite correct -- so we take the opportunity to state the correct QE result and indicate how the proof of \cite[Theorem 4.1]{GMS15} should be corrected. We do not quite see how to extract our result from \cite{GMS15} and Feferman--Vaught. 

\begin{theorem}\label{GMSgen}
Let $\mathcal{C}_d$ be the collection of $(d+1)$-sorted structures of the form $(F,V_1,\ldots, V_d)$, where $F$ is a finite field in the language $L_{{\rm rings}}$ of rings, and each $V_i$ is a  finite dimensional $F$-vector space, each carrying a distinct copy of the language of groups and with distinct function symbols for scalar multiplication $F\times V_i \longrightarrow V_i$. Then $\mathcal{C}_d$ is a polynomial $R$-\mac  and the polynomial functions in $R$ have the form $h=p(\bf{F},\bf{V_1},\ldots,\bf{V_d})$, so that if $M=(F,V_1,\ldots,V_d)\in \mathcal{C}_d$ then $h(M)=p(|F|,|V_1|,\ldots,|V_d|)$.
\end{theorem}

We consider the structures of $\mathcal{C}_d$ in the language $L_{vs}^d$ as described in the theorem, but in addition for each $i\in \{1,\ldots,d\}$ and each $n\geq 1$ an  $n$-ary relation symbol $\theta_n^i(v_1,\ldots,v_n)$ in the $V_i$ sort, and for each such $i,n$ and each $j \in \{1,\ldots,n\}$ an $n+1$-place function symbol $\lambda^i_{n,j}$. 
We interpret 
$\theta^i_n$ so that $(M,V_1,\ldots,V_d)\models \theta^i_n(v_1,\ldots,v_n)$ if and only if $v_1,\ldots,v_n \in V_i$ and  $v_1,\ldots,v_n$ are linearly independent. Furthermore, if $(M,V_1,\ldots,V_d)\models \theta^i_n(v_1,\ldots,v_n)$ and $w=\sum_{k=1}^n a_kv_k$ with $a_1,\ldots,a_k\in F$, then $\lambda^i_{n,j}(v_1,\ldots,v_n,w)=a_j$, and 
we define $\lambda^i_{n,j}(v_1,\ldots,v_n,w)=0_F$ otherwise.  Observe that the $\theta_n^i(v_1,\ldots,v_n)$ and $\lambda^i_{n,j}$ are definable in the original language $L_{vs}^d$. Let $T_{vs}^d(F)$ be the theory of structures $(F,V_1,\ldots,V_d)$ in this language where each $V_i$ is an infinite-dimensional vector space over the field $F$.

\begin{proposition}\label{granger}
The theory  $T_{vs}^d(F)$ is complete and has quantifier elimination relative to the field sort; that is, any formula is equivalent modulo $T_{vs}^d(F)$ to one whose only quantifiers are in the field sort.
\end{proposition}


Note that in \cite[Lemma 4.1]{GMS15}, resting on \cite{Granger99}, this  is stated with $d=1$ without the function symbols 
$\lambda^i_{n,j}$. However, it is clear that these (or some replacement) are needed. A proof incorporating the function symbols also is given in the forthcoming \cite{Chernikov-Hempel23}. 

\medskip

{\em Proof of \autoref{GMSgen}.}
 For clarity, we just give the argument for $d=1$; the argument for $d>1$ is essentially the same. As in \autoref{onvevarpolymac}---see also the end of the proof of 
 \cite[Theorem 4.1]{GMS15}---it suffices to consider formulas with a single variable (and possibly parameters). It remains to explain how the proof of \cite[Lemma 4.2]{GMS15} needs to be adjusted to take into account that the language includes the function symbols $\lambda^i_{n,j}$. Consider a set  defined by a formula $\phi(x,\bar{b}\bar{c})$, where, in the formula 
 $\phi(x,\bar{v}\bar{y})$, $\bar{v}$ ranges over the vector space sort and $\bar{y}$ over the field sort.

Suppose first that $x$ ranges over the vector space sort. As in \cite[Lemma 4.2]{GMS15} we may suppose that $\phi$ is a conjunction of quantifier-free formulas and field formulas. Using the $\theta_n$ formulas, we may assume that $\bar{v}$ is linearly independent and reduce to two possibilitiess: where  $x$ is in the span of $\bar{v}$, and where it is not. In the first case, we may identify the span of $\bar{v}$ with $F^n$, and the original formula reduces to a field formula defining a set with cardinality approximately a polynomial in $F$, with the definability clause also holding. In the second case, the $\lambda$-functions play no role and the argument from \cite{GMS15} applies.

Now assume that $x$ ranges through the field sort (Case 2 in the proof of \cite[Lemma 4.2]{GMS15}). The proof is essentially as given there, except that polynomials $p(x,\bar{y})$ have to be replaced by terms involving also the $\lambda$-functions. In each case, such formulas can be replaced by field formulas, and the definability clause also holds.  $\Box$

\subsection{A \mac{} from representation theory.}\label{finiterep}
For background on the relevant representation theory of quivers, see for example \cite{bernstein} or \cite{ringel}. 

 A {\em quiver} is a directed graph $Q$, and we assume that $Q$ has vertex set $\{v_1,\ldots,v_n\}$. Given a field $k$, the {\em path algebra} $kQ$ is an associative algebra with basis the set of all directed paths in $Q$, including a trivial path $e_i$ at each vertex $v_i$. Multiplication is defined on the basis by composition of paths, a product taking value 0 if the composition is undefined.
 
A {\em representation} of $Q$ (with base field $k$) consists of a $k$-vector space $V_i$ for each $i=1,\ldots,n$ and a $k$-linear map 
$\rho\colon V_i \longrightarrow V_j$ for each arrow $v_i\to v_j$.  
The {\em dimension vector} of the representation is the sequence $(r_1,\ldots,r_n)$ where ${\rm dim}(V_i)=r_i$ for each $i$.
There is a natural notion of {\em direct sum} of two representations, and a representation of $Q$ is {\em indecomposable} if it cannot be written as a direct sum of two non-trivial representations.  
It is well-known that the category of representations of $Q$ over $k$ is equivalent to the category of left $kQ$-modules. Indeed, given a representation of $Q$ with $k$-vector spaces $V_i$ for each $i$, one forms the $k$-vector space $V=\oplus_{i=1}^n V_i$. Left multiplication by $e_i$ corresponds to the linear map $V\longrightarrow V$ given by projection $V\longrightarrow V_i$ composed with inclusion $V_i \longrightarrow V$, and given a directed path $\alpha\colon v_i \to v_j$ with corresponding linear map $\rho\colon V_i \longrightarrow V_j$, given by composing maps coming from arcs, left multiplication by $\alpha$ is the linear map induced by $\rho$ on $V_i$ and by the 0-map on the other $V_k$. 
 
The quiver $Q$ has {\em finite representation type} if it has just finitely many isomorphism types of indecomposable representations. This property depends only on $Q$ and not on $k$. By a famous theorem of Gabriel \cite{gabriel},  a quiver has finite representation type if and only if it is a finite disjoint union of oriented copies  of the Dynkin diagrams $A_n, D_n, E_6,E_7, E_8$. Each indecomposable is associated with a unique positive root of the `Tits form' (see \cite{gabriel}, e.g.) of the quiver.  
 
 We now suppose $Q$ is a quiver of finite representation type. Then the indecomposables of $Q$ each have a fixed dimension vector which is independent of the base field $k$. Furthermore, the linear maps $V_i \longrightarrow V_j$ are defined over $\mathbb{Z}$, so are independent of the field $k$; that is, we may identify the $V_i$ with finite powers of $k$ in such a way that the linear maps are given by matrices over $\mathbb{Z}$. This can be seen from the treatment in \cite{bernstein}: it is shown there that the indecomposable representations can be constructed from 1-dimensional simple modules (for a possibly differently oriented version of the quiver) by applying a sequence of reflection functors, and this process preserves the property of being defined over $\mathbb{Z}$. It also follows from Theorem 1 of \cite{crawley}, which says that for each positive root of the Tits form of the quiver there is a unique (up to isomorphism) representation of $Q$ by a finitely generated free $\mathbb{Z}$-module, with the property that over any field, it gives the (unique) indecomposable representation of that dimension vector over that field.\footnote {The context in \cite{crawley} is more general, but one takes $R=\mathbb{Z}$, and for a Dynkin quiver positive roots are exactly {\em real Schur roots}, and furthermore working over a field, indecomposable modules are exactly {\em exceptional modules}.}

 We view a representation of $Q$ as a 3-sorted structure $(k,kQ, V)$, where $k$ is the base field and $V$ is the corresponding $k$-vector space and $kQ$-module. We view  $(k,kQ,V)$ as a structure in a 3-sorted language $L_Q$ with 
sorts $S_F$ (for the field ${\mathbb F}_q)$, $S_A$ (for the algebra ${\mathbb F}_qQ$), and $S_V$ (for the vector space $V$).
The language $L_Q$ has disjoint copies of $L_{{\rm rings}}$ for the sorts $S_F$ and $S_A$, the language $L_{{\rm groups}}$ of groups on $S_V$, a function symbol $S_F \times S_A\longrightarrow S_A$ for scalar multiplication on the algebra, a function symbol $S_F \times S_V$ for scalar multiplication (by field elements) on $V$, and a function symbol $S_A\times S_V \longrightarrow S_V$ for the left action of the algebra sort on the vector space sort.

\begin{theorem} \label{quiver-theorem}  Let $Q$ be a quiver of finite representation type. Then the collection
\[\mathcal{C}(Q):=\{({\mathbb F}_q, {\mathbb  F}_q Q, V): \textup{$q$  prime power},\; \textup{$V$  finite-dimensional ${\mathbb F}_q Q$-module}\}\]
is a weak polynomial \mac{}.
\end{theorem} 

\begin{proof}
Let $W_1(k),\ldots,W_t(k)$ denote the indecomposable representations of $Q$ determined by corresponding positive roots of the Tits form of $Q$. Suppose $W_i(k)$ has dimension vector $(r_{i1},\ldots,r_{in})$ and put $s_i:=r_{i1}+\ldots +r_{in}=\dim_k W_i(k)$. For each $i=1,\ldots,t$ and prime power $q$ let $T_i(q)=({\mathbb F}_q,{\mathbb F}_q Q, W_i({\mathbb F}_q))$. Put $\mathcal{C}_i:=\{T_i(q): q \mbox{~prime power}\}$, a family of finite $L_Q$-structures.

\medskip

{\em Claim.} (i) The structure $T_i(q)$ is interpretable in  the field ${\mathbb F}_q$, uniformly in $q$, and includes the home sort ${\mathbb F}_q$. 

\hskip 37pt (ii) The class $\mathcal{C}_i$ is an $N$-dimensional weak asymptotic class (i.e. asymptotic class without the definability clause), where $N=1+s_i+\dim_{{\mathbb F}_q}({\mathbb F}_q Q)$.

\medskip

{\em Proof of Claim.} (i) The path algebra ${\mathbb F}_qQ$ has a fixed  dimension---the number of paths of $Q$---over 
${\mathbb F_q}$. This is finite for the quivers under consideration. The algebra  multiplication is uniformly definable on the basis, and so extends by linearity.  Thus, $({\mathbb F}_q Q, {\mathbb F}_q)$ is uniformly interpretable in ${\mathbb F}_q$. The vector space $W_i({\mathbb F}_q)$ has a fixed dimension vector and so is a sum of $n$ ${\mathbb F}_q$-vector spaces and thus is uniformly interpretable in ${\mathbb F}_q$. If we name parameters to equip the vector space at each $Q$-vertex with a basis then the linear maps between them determined by arrows are given uniformly by matrices over $\mathbb{Z}$. Thus the maps $W_i({\mathbb F}_q) \longrightarrow W_i({\mathbb F}_q)$ given by basis elements of ${\mathbb F}_q Q$ are also given uniformly. It follows by linearity that multiplication by arbitrary elements of $\mathbb{F}_qQ$ is definable uniformly.

(ii) This follows immediately from (i) and \cite[Lemma 3.7]{Elwes07}.

\medskip

An ${\mathbb F}_q Q$-module $V$ has form $W_{1}({\mathbb F}_q)^{l_1} \oplus \ldots \oplus W_{t}({\mathbb  F}_q)^{l_t}$. We first claim that if $W$ denotes one of the $W_i$, of dimension $s$, then the class
\[\mathcal{D}(Q)=\{({\mathbb F}_q, {\mathbb  F}_q Q, W({\mathbb F}_q)^l): q \mbox{~a prime power}, l\in {\mathbb N}\setminus\{0\}\}\]
is a \mac{}. Now the ${\mathbb F}_q Q$-module $W({\mathbb F}_q)^l$ is $\mathbb{F}_qQ$-isomorphic to $W(q)\otimes_{{\mathbb F}_q} U(l,q)$, where $U(l,q)$ is an $l$-dimensional vector space over ${\mathbb F}_q$, and the paths in $Q$ act as the identity map on $U(l,q)$: Indeed, if $(f_1,\ldots,f_l)$ is a basis of $U(l,q)$ then the map $(w_1,\ldots,w_l)\longmapsto w_1\otimes f_1+\ldots+w_l\otimes f_l$ provides the required $\mathbb{F}_qQ$-isomorphism.
By \autoref{examples-macs}(1)(b)---see also \autoref{GMSgen}---the class of structures $({\mathbb F}_q, U(l,q))$ is a polynomial \mac{} Furthermore, each element of $W(q)\otimes U(l,q)$ can be written as a sum of at most $s$ simple tensors, i.e., elements of form $w\otimes u$ where $w\in W(q)$ and $u\in U(l,q)$. The equivalence relation which expresses that two such sums of simple tensors are equal in $W(q) \otimes U(l,q)$ is determined by bilinearity conditions and so is definable, and it follows that the vector space $W(q)\otimes U(l,q)$ is interpretable in $({\mathbb F}_q, U(l,q))$.  Since the multiplication by paths in $Q$ is $\emptyset$-definable in $({\mathbb F}_q, U(l,q))$, it follows  from \autoref{thminterpmac} that $\mathcal{D}(Q)$ is a weak polynomial \mac{}.

For the general case, observe that by \autoref{GMSgen} the class of structures $(\mathbb{F}_q, V_1,\ldots,V_t)$, where each $V_i$ is a  finite-dimensional vector space over $\mathbb{F}_q$, is a polynomial \mac{}. Since  $W_{1}({\mathbb F}_q)^{l_1} \oplus \ldots \oplus W_{t}({\mathbb  F}_q)^{l_t}$ is uniformly interpretable in $(\mathbb{F}_q, U(l_1,q),\ldots,U(l_t,q))$, the result follows from \autoref{thminterpmac}.
\end{proof}

\vspace{.1\baselineskip}

\begin{remark}~
\begin{enumerate}[(i)]
\item In \autoref{quiver-theorem} (and the Corollary below), we suspect that `weak \mac{}' can be replaced by `\mac{}'. The difficulty is that in the claim we have mutual interpretability without parameters between the $T_i(q)$ and the $\mathbb{F}_q$, but it is not clear that we have parameter-free bi-interpretability, since the vector spaces at each quiver vertex do not naturally come with a uniformly definable basis. In the claim, if we identify the vector space at each vertex with a power of $\mathbb{F}_q$, so that the corresponding maps between them are given uniformly by matrices over $\mathbb{Z}$, then the class of corresponding structures of form $(\mathbb{F}_q,\mathbb{F}_qQ, \mathbb{F}_q^{s_i})$ does form a \mac{}.

\item The conclusion of \autoref{quiver-theorem} still holds if $V$ is expanded by unary predicates interpreted by the powers of indecomposables $W_i(\mathbb{F}_q)^{l_i}$.

\end{enumerate}
\end{remark}

\begin{corollary}\label{chainsubspace} Let $k$ be a positive integer and let $\mathcal{C}$ be the collection of all structures of form $({\mathbb F}_q, W_1,\ldots,W_k)$, where $q$ is a prime power and $W_1\leq W_2 \leq \ldots \leq W_k$ are finite-dimensional ${\mathbb F}_q$-vector spaces. We view these as two-sorted structures, with a field sort and a vector space sort expanded by predicates for subspaces $W_1,\ldots,W_{k-1}$ of $W_k$. Then $\mathcal{C}$ is a weak polynomial \mac{}.
\end{corollary}

\begin{proof} Let $Q$ be the quiver obtained by orienting the Dynkin diagram $A_k$ as a directed path $v_1\to \cdots \to v_k$, and let $\alpha_i$ be the arc $v_i\to v_{i+1}$. For each structure $({\mathbb F}_q, W_1,\ldots,W_k)$ as above, let $r_i:=\dim(W_i)$, and let $V$ be a representation of $Q$ over ${\mathbb F}_q$ with an $r_i$-dimensional vector space $V_i$ at vertex $v_i$ for each $i$ and with $\alpha_i$ corresponding to an embedding $V_i\longrightarrow V_{i+1}$. We view $V$ as an ${\mathbb F}_q Q$-module, a member of $\mathcal{C}(Q)$ as described in the proof of  \autoref{quiver-theorem}. Then $({\mathbb F}_q, W_1,\ldots,W_k)$ is uniformly interpretable in the structure $({\mathbb F}_q, {\mathbb F}_q Q, V)$ since we may identify $W_k$ with $V_k$ and $W_i$ with $\alpha_{k-1}\circ \cdots \circ \alpha_i(V_i)$. The result follows.
\end{proof}

\begin{remark} \label{quiverpoly} \rm An ultraproduct of members of $\mathcal{C}$ in \autoref{chainsubspace} has SU-rank at most $\omega^k$.
 \end{remark}

\subsection{Vector spaces with bilinear forms}
Let $V$ be a vector space over a field $F$, equipped with a bilinear form $\beta:V\times V \to F$. For $v\in V$ we put  $v^{\perp}=\{w\in V: \beta(v,w)=0\}$. The form $\beta$ is said to be
 {\em non-degenerate} if $(\forall v\in V)\ v^{\perp}=V\, \longrightarrow\, v=0$, and
 to be {\em alternating} if 
$\forall v\in V(\beta(v,v)=0)$.  We view vector spaces with a bilinear form as 2-sorted structures $(V,F)$ in a language $L_{\bil}$ which has the language of groups on the sort for $V$, the language of rings on the sort for $F$, a function symbol $F\times V \to V$ for scalar multiplication, and a function symbol $\beta:V\times V \to F$ for the bilinear form. The expanded language $L_{\bil,\qe}$ contains also---as in Section 3.1---for each $n\geq 1$, an $n$-ary relation symbol $\theta_n$, and, for each 
$n\geq 1$ and $j=1,\ldots, n$, function symbols $\lambda_{n,j}$. In structures $(V,F)$, the relation 
 $\theta_ n(v_1,\ldots, v_n)$ holds just of linearly independent $n$-tuples in the $V$-sort, and if $v_1,\ldots,v_n$ are linearly independent and $w=\sum_{i=1}^n a_iv_i$ for some $a_1,\ldots,a_n\in F$, then
$\lambda_{n,j}(v_1,\ldots,v_n,w)=a_j$, with the $\lambda_{n,j}$ taking value $0_F$ otherwise.

\begin{proposition} The theory of non-degenerate alternating bilinear forms has quantifier elimination in the vector space sort in the language $L_{\bil,\qe}$; that is, any formula is uniformly (in all such structures) equivalent to one with no vector space quantifiers.
\end{proposition}
\begin{proof}
 A version of this is claimed in Theorem 9.2.3 of \cite{Granger99} but is not quite correct in the form there, since the functions $\lambda_{n,j}$ are omitted but are clearly needed. With these functions, it is straightforward to adapt Granger's proof to obtain the result; a detailed proof is given in \cite{Chernikov-Hempel23}. See also Theorem 5.18 of \cite{snowden}, though some care is needed in interpreting results from that paper since, though the language is treated as 2-sorted, the authors do not think of the field as varying. 
\end{proof}

\begin{lemma}\label{beta} Let $V$ be a finite-dimensional vector space over the finite field $\mathbb{F}_q$, equipped with a non-degenerate bilinear form $\beta:V\times V \to \mathbb{F}_q$, and let $v_1,\ldots,v_m$ be linearly independent and $a_1,\ldots,a_m\in \mathbb{F}_q$. Then the set $\{v\in V:\bigwedge_{i=1}^m \beta(x,v_i)=a_i\}$ has cardinality $|V|/q^m$.
\end{lemma}
\begin{proof} The map $\bar{\beta}:V \to \mathbb{F}_q^m$ given by $\bar{\beta}(x)=(\beta(x,v_1),\ldots,\beta(x,v_m))$ is a group homomorphism, so it suffices to show that it is surjective. Extend $v_1,\ldots,v_m$ to a basis $v_1,\ldots,v_l$ of $V$.
The homomorphism $\beta^*:V\to \mathbb{F}_q^l$ given by $\beta^*(x)=(\beta(x,v_1),\ldots,\beta(x,v_l))$ is injective (by non-degeneracy) and so is surjective. It follows that $\bar{\beta}$ is also surjective, as required.

\end{proof}

\begin{theorem} \label{bilinear-granger} Let $\mathcal{C}_{\bil}$ be the collection of all $L_{\bil}$ structures $(V,F)$ where $V$ is a finite-dimensional vector space over the finite field $F$, equipped with a non-degenerate alternating bilinear form $\beta$. Let $R=\mathbb{Q}(\mathbf{F})[\mathbf{V}]$.  Then $\mathcal{C}_{\bil}$ is an $R$-m.a.c..
\end{theorem}



\begin{proof} We apply Theorem~\ref{thm:projection.lemma}(i), working in the richer language $L_{\bil,\qe}$ with quantifier elimination in the vector space sort. Consider first a formula $\phi(x,\bar{y}\bar{\xi})$, where $x$ and $\bar{y}=(y_1,\ldots,y_m)$ range through the $V$-sort and $\bar{\xi}=(\xi_1,\ldots,\xi_n)$ through the $F$-sort. We aim to compute the approximate cardinality of $\phi(V,u_1,\ldots,u_m,a_1,\ldots,a_n)$  in $(V,F)\in \mathcal{C}_{\bil}$, for $u_1,\ldots,u_m\in V$ and $a_1,\ldots,a_n\in F$. 
By adjusting $\phi$, we may suppose that  $\phi(x,y_1,\ldots,y_m,\xi_1,\ldots,\xi_n)$ determines whether or not $x$ lies in the linear span $\langle y_1,\ldots,y_m\rangle$. We may also assume that $\phi(x,y_1,\ldots,y_m,\xi_1,\ldots,\xi_n)$ implies that $y_1,\ldots,y_m$ are linearly independent; indeed, if say $y_1,\ldots,y_r$ span $\langle y_1,\ldots,y_m\rangle$ with $r<m$, we may replace $y_{r+1},\ldots,y_m$ each by a linear combination of the form  $\sum_{i=1}^r  \nu_i y_i$, where the $\nu_i$ for $i=1,\ldots , r$ are new field variables.   

Suppose first that $\phi(x,y_1,\ldots,y_m,\xi_1,\ldots,\xi_n) \to x\in \langle y_1,\ldots,y_m\rangle$. We may, definably over $u_1,\ldots,u_m$, identify $\langle u_1,\ldots,u_m\rangle$ with $F^m$, and define $\beta$ uniformly on $F^m$, putting
$\beta(\bar{c},\bar{d})=\sum_{i=1}^m\sum_{j=1}^m c_id_j\beta(u_i,u_j)$. Thus, using quantifier elimination in the sort $V$, the formula $\phi$ can in this case be replaced by an $L_{\rings}$-formula, and the result follows in this case from the fact that finite fields form a m.a.c. with functions taking values of form $\mu |F|^d$ as in \cite{CvdDM92}.

Suppose next that $\phi(x,y_1,\ldots,y_m,\xi_1,\ldots,\xi_n) \to x\notin \langle y_1,\ldots,y_m\rangle$. 
Using bilinearity of $\beta$, we may introduce a new field variable $\rho_i$ for each term 
$\beta(x,y_i)$ in $\phi$, and replace 
$\phi$ by a formula $\phi^*(x,\bar{y},\bar{\xi}, \bar{\rho})$ of the form
$\beta(x,y_{i_1})=\rho_1\wedge\ldots \wedge \beta(x,y_{i_r})=\rho_r \wedge \chi(x,\bar{y},\bar{\xi},\bar{\rho})$, where $\chi$ involves no terms of the form $\beta(x,y_i)$. Working in $(V,F)$, let $\bar{u}\in V^m$ be linearly independent, $\bar{a}\in F^n$ and  $\bar{b}\in F^r$. Since we have assumed that $\varphi$ implies that $x \in V\setminus \langle u_1,\dots ,u_m\rangle$, applying \autoref{beta} we see in this case that there are 
$(|V|-|F|^m)/|F|^r$ elements $x\in V$ satisfying $\beta(x,u_{i_1})=b_1\wedge\ldots \wedge \beta(x,u_{i_r})=b_r$. Moreover, as 
$\bar{b}$ ranges over  $F^r$ the sets of $x\in V$ satisfying $\beta(x,u_{i_1})=b_1\wedge\ldots \wedge \beta(x,u_{i_r})=b_r$ are pairwise disjoint. If we further replace each $\beta(y_i,y_j)$ for $i<j\leq m$ by a new variable $\tau_{ij}$, and each $\beta(y_i,y_i)$ by $0$ in $\chi(x,\bar{y},\bar{\xi},\bar{\rho})$, then our formula $\phi^*$ becomes $\phi^{**}(x, \bar{y},\bar{\xi},\bar{\rho},\bar{\tau})$ of the form
$$\beta(x,y_{i_1})=\rho_1\wedge\ldots \wedge \beta(x,y_{i_r})=\rho_r \wedge \bigwedge_{i<j} \beta(y_i,y_j)=\tau_{ij} \wedge \theta(x,\bar{y},\bar{\xi},\bar{\rho},\bar{\tau})$$ 
where $\theta$ is a formula in the 2-sorted vector space language for $(V,F)$. Now let $\bar{c}\in F^k$, where 
$k={m \choose 2}$, be such that $\bigwedge_{i<j} \beta(u_i,u_j)=c_{ij}$. Since for any two linearly independent $m+1$ tuples 
$\bar {w}_1, \bar{w}_2\in V^{m+1}$  there is an automorphism of 
$(V, F)$ that fixes $F$ pointwise taking $\bar{w}_1$ to $\bar{w}_2$, it follows that for each $\bar b\in F^m$, either 
$\theta(x,\bar{y},\bar{a},\bar{b},\bar{c})$ is satisfied by all such $\bar w$ (substituted for $(x,\bar{y})$) or it is satisfied by no such 
$\bar w$. By identifying the span of linearly independent $m+1$ tuples with $F^{m+1}$, e.g., it follows that there is a field formula $\theta^* (\bar{\xi}, \bar{\rho}, \bar{\tau})$ such that for all $\bar b\in F^m$, there exist linearly independent 
$v, u_1, \ldots, u_m\in V$ such that $\theta(v,\bar{u},\bar{a},\bar{b},\bar{c})$ holds in $(V,F)$ if and only if 
$\theta^* (\bar{a}, \bar{b}, \bar{c})$ is satisfied in $F$. Hence, there are $\mu$ and $d$ such that there are $\mu |F|^d$ tuples 
$\bar{b}\in F^r$ satisfying $\theta^* (\bar{a}, \bar{b}, \bar c)$, as in \cite{CvdDM92}. (Note that there are finitely many possible pairs $(d,\mu)$ depending definably on $\bar{c}$---and thus $\bar{u}$---and $\bar{a}$.) 
It follows that the number of realisations of our original formula $\phi(x,\bar{u},\bar{a})$ in this case is $\frac{(|V|-|F|^m)\mu |F|^d}{|F|^r}$. This has the required form.

Lastly, we consider formulas of the form $\phi(\rho,\bar{y},\bar{\xi})$, where $\rho$ ranges through the field sort. As before, we may reduce to the case when $\phi(\rho,\bar{y},\bar{\xi})$ implies that the entries $y_1,\ldots,y_m$ of $\bar{v}$  are linearly independent. This is handled as in the last paragraph but one, since, working over $\bar{v}$ and using quantifier elimination in the $V$-sort, we may identify $\langle y_1,\ldots,y_m\rangle$ with $\mathbb{F}_q^m$, and replace $\phi$ by an $L_{\rings}$-formula. 

It can be checked that the definability condition in the definition of m.a.c. holds in each step in the above argument, so holds for $\mathcal{C}_{\bil,\qe}$.

\end{proof}

\section{Multidimensional exact classes}\label{multiexactsection}

In this section we focus on  multidimensional {\em exact} classes. We give a range of examples, partial results, and conjectures. The examples suggest that there is a  connection between being a \mec{} and ultraproducts being one-based (in an appropriate sense).

\subsection{Smooth approximation and \mec{}s of homogeneous structures}

We first recall the notion of {\em smooth approximation} introduced by Lachlan and developed in \cite{KLM89} and 
\cite{Cherlin-Hrushovski03}. We say that the countably infinite structure $M$ is {\em smoothly approximable} if $M$ is $\omega$-categorical and $M=\bigcup _{i\in \omega} M_i$ is a union of a chain  of finite substructures $M_0\leq M_1 \leq \ldots \leq M$, where the $M_i$ are {\em homogeneous substructures} of $M$ in the sense that two tuples of $M_i$ lie in the same orbit of $\Aut(M)$ if and only if they lie in the same orbit of $\Aut(M)_{\{M_i\}}$, the subgroup of $\Aut(M)$ stabilising $M_i$ setwise.

The following result follows rather directly from the definition of smooth approximation.

\begin{proposition} \cite[Proposition 3.2.1]{Wolf}\label{Liesoft} Let $M$ be a structure smoothly approximated by a chain $(M_i)_{i\in \omega}$ of finite substructures. Then there is $R$ such that $\{M_i:i<\omega\}$ is an $R$-\mec{}.
\end{proposition}

Using further material from \cite{Cherlin-Hrushovski03}, \autoref{Liesoft} was  extended by Wolf (\cite{WolfPhD}, also \cite[Theorem 4.6.4]{Wolf}) to \autoref{wolf-thesis} below. Note that the statement and proof in \cite{Wolf} contain an inaccuracy, as one needs to allow the expansion to $L'$. For example, let $L$ have a single binary relation $E$, and let $\mathcal{C}$ be the collection of finite $L$-structures in which $E$ is interpreted by an equivalence relation with at most two class sizes; then the number of orbits on quadruples is bounded, but the definability clause in the \mec{} definition is not satisfied. Essentially, one needs to work in the language $\mathcal{M}^D$ of \cite[Definition 8.3.1]{Cherlin-Hrushovski03} with dimension quantifiers expressing that one Lie geometry has smaller dimension than another (and with Witt index quantifiers), and this can be done in a finite extension $L'$ of $L$ as above. In Wolf's proof, the error occurs in the second use of \cite[Proposition 4.4.3]{Cherlin-Hrushovski03} in the penultimate paragraph of the proof of \cite[Theorem 4.4.1]{Wolf}: \cite[Proposition 4.4.3]{Cherlin-Hrushovski03} requires that the models have `true dimensions' relative to the skeletal language, and this is ensured by the expansion to $\mathcal{M}^D$. In the statement of \cite[Theorem 4.4.1]{Wolf} one first expands the members of $\mathcal{C}(\mathcal{L},d)$ to $\mathcal{M}^D$ before finding the finite partition. See the corresponding discussion in the proof of \cite[Proposition 8.3.2]{Cherlin-Hrushovski03}.  

\begin{theorem} \label{wolf-thesis} Let $L$ be a finite first-order language, $d$ a natural number, and $\mathcal{C}=\mathcal{C}(L,d)$ the collection of all finite $L$-structures $M$ with at most $d$ 4-types (so $\Aut(M)$ has at most $d$ orbits on $M^4$). Then 
there is a finite language $L'\supseteq L$ and for each $M\in \mathcal{C}$ an $L'$-expansion $M'$ with the same automorphism group as $M$, such that
$\mathcal{C}':=\{M':M\in \mathcal{C}\}$ is a polynomial \mec{}.
\end{theorem}

The polynomials here are essentially polynomials in the cardinalities of certain coordinatising Lie geometries (formally, in numbers of the form  $(-\sqrt{q})^{d_E(J)}$, where $E$ is a finite `envelope' and $J$ is a (finite-dimensional approximation of a) Lie geometry of $E$ of dimension $d_E(J)$ over a field $\mathbb{F}_q$, or in the case of a geometry that is a pure set, $d_E(J)$ is the cardinality of the set). The geometries arising belong to a `standard system of geometries' in the sense of \cite[Definition 2.5.6]{Cherlin-Hrushovski03}. Wolf's results say in particular that if $M$ is smoothly approximated then the collection of all finite envelopes of $M$ is a polynomial \mec{}, after extension of the language as above.

See also Lemma 5.2.2 and the preceding pages of \cite{Cherlin-Hrushovski03}.

\begin{remark} In the last theorem the exponent 4 cannot be reduced to 3. For consider the collection of all finite Desarguesian projective planes, each viewed as a structure whose universe is the set of points, equipped with a ternary collinearity relation. By the Veblen--Young theorem, the Desarguesian plane over $\mathbb{F}_q$  has automorphism group P$\Gamma$L${}_3(q)$ in its natural action on projective space, which has a bounded number of orbits on triples (and two orbits on triples of distinct elements). However, by a classical fact from projective geometry, the Desarguesian plane over $\mathbb{F}_q$ uniformly {\em defines} the field $\mathbb{F}_q$, and so no infinite collection of such planes can be even a weak \mec{}, by Theorem~\ref{thminterpmac}(iii) in combination with \autoref{weakmec-fields} below.
\end{remark}

The above results of Wolf, and those which follow, make the following conjecture very natural. A general proof might require revisiting Lachlan's work on finite homogeneous structures, but under weaker assumptions than full homogeneity. Below, a countable structure over a relational language is {\it homogeneous} if every isomorphism between finite substructures extends to an automorphism. Part of the motivation is that the conjecture leads to natural questions about families of finite structures with arbitrarily high levels of combinatorial regularity but no assumptions on the automorphism group (rather in the way that distance-regularity for graphs is a combinatorial generalisation of distance-transitivity); see Section 7. 

\begin{conjecture}\label{mec-conj}~
\begin{enumerate}[(i)]
\item Let $M$ be a homogeneous structure over a finite relational language $L$. Then there is an m.e.c. with ultraproduct elementarily equivalent to $M$ if and only if $M$ is stable.

\item Let $M$ be an unstable homogeneous structure over a finite relational language. Then $M$ is not elementarily equivalent to any structure interpretable in an ultraproduct of a \mec{}.
\end{enumerate}
\end{conjecture}

The right-to-left direction of \autoref{mec-conj}(i) is true:

\begin{proposition} \label{homog1}
Let $M$ be a stable homogeneous structure over a finite relational language $L$. Then there is an \mec{} $\mathcal{C}$ with an infinite ultraproduct which is elementarily equivalent to $M$.
\end{proposition}

\begin{proof} It is well-known that any stable structure which is homogeneous over a finite relational language is $\omega$-stable
-- for example, definability of types and quantifier-elimination yield that there are countably many types over any countable model.
Hence $M$ is smoothly approximable by \cite[Corollary 7.4]{chl}. It follows that there is an \mec{} of envelopes of $M$, by the results of Wolf (essentially \autoref{Liesoft}). We may take this m.e.c. to consist of the members of a smoothly approximating sequence of homogeneous substructures $M_i$ of $M$, and every sentence in $\Th(M)$ holds in cofinitely many of the $M_i$. Hence any non-principal ultraproduct of $\mathcal{C}$ is elementarily equivalent to $M$, and the result follows.
\end{proof}

For the other direction of \autoref{mec-conj}(i), we have the limited evidence in the following result. Let $I_n$ (for $n\geq 3$) denote the 
digraph consisting of $n$ vertices with no arcs between them. The classification of homogeneous digraphs by Cherlin \cite{cherlin_98}  includes for each $n\geq 3$ the universal homogeneous $I_n$-free digraph $Q_n$. The `generic bipartite graph' is a homogeneous structure in a language with two relations, an equivalence relation $E$ for the bipartition, and a symmetric irreflexive graph relation $R$ which only holds between $E$-inequivalent pairs; it is the unique countably infinite such structure such that for any two finite disjoint subsets $A,B$ of one part of the bipartition, there is a vertex in the other part adjacent to all vertices of $A$ and to no vertices of $B$.

\begin{theorem}\label{homog2} Let $M$ be any of the following homogeneous structures. Then there is no \mec{} with an ultraproduct elementarily equivalent to $M$.
\begin{enumerate}[(i)]
\item Any homogeneous structure whose theory has the strict order property. 

\item Any unstable homogeneous  graph.

\item Any homogeneous tournament.

\item The digraph $Q_n$ for each $n\geq 3$.

\item The generic bipartite graph, with the two parts named by unary predicates.

\end{enumerate}
\end{theorem}

\begin{proof}
(i) This follows immediately from \autoref{mac-nsop}.

(ii) If $M$ is an unstable homogeneous graph, then by the Lachlan--Woodrow classification \cite{lw}  $M$ is the random graph $H$, the universal homogeneous $K_n$-free graph $H_n$ for some $n\geq 3$, or the complement of $H_n$. We shall suppose that $M$ is the random graph, as the same argument eliminates the other cases. Let $T=\Th(M)$, and suppose that $\mathcal{C}$ is a \mec{} with an ultraproduct $N\models T$.
After thinning out $\mathcal{C}$ we may suppose that all non-principal ultraproducts of $\mathcal{C}$ are elementarily equivalent -- that is, each element of $T$ holds of cofinitely many $P\in \mathcal{C}$.

For any formula $\phi(x,\bar{y})$ there is a finite set $E$ of functions $h\colon\mathcal{C}\longrightarrow {\mathbb R}$ and some formula $\psi_h(\bar{y})$ for each $h\in E$, such that for any $P\in \mathcal{C}$ and $h\in E$, if $\bar{a}\in P^{|\bar{y}|}$ then
\[
P\models \psi_h(\bar{a}) \Longrightarrow |\phi(P,\bar{a})|=h(P).
\]
Since $T$ has quantifier-elimination, there is a quantifier-free formula $\chi_h(\bar{y})$ and $\sigma \in T$ such that 
$\sigma \models \forall \bar{y} \, (\psi_h(\bar{y})\leftrightarrow \chi_h(\bar{y}))$. As $\sigma$ holds on cofinitely many members of  $\mathcal{C}$, provided we work in sufficiently large $P\models T$, we may assume that $\psi_h(\bar{y})$ is quantifier-free. Taking $\bar{y}=(y_1,\ldots,y_5)$ and $\phi(x,\bar{y})$ to say that $x$ is adjacent to each $y_i$, it follows that any sufficiently large $P\in \mathcal{C}$ is {\em 5-regular}: For any $A\subset P$ of size at most 5, the number of common neighbours of $A$ depends only on the isomorphism type of $A$.

By the note added in  proof at the end of \cite{cameron}, any finite 5-regular graph occurs in the list in \cite[Theorem 3.2]{cameron}: the pentagon, the line graph of $K_{3,3}$, a disjoint union of complete graphs all of the same size, or the complement of the latter. (We remark that these are exactly the finite homogeneous graphs, so include  the finite graphs smoothly approximating the stable homogeneous graphs). However, if $P$ is chosen sufficiently large, then  $P$ will satisfy appropriate extension axioms true of the random graph but not satisfied by these, a contradiction.

(iii) By \cite{lachlan}, there are three infinite homogeneous tournaments, namely a dense total order, the `local order', and the random tournament. The first two have the strict order property, so by (i) cannot be elementarily equivalent to an ultraproduct of an \mec{}. So let $M$ be the random tournament, and suppose for a contradiction that $M$ is elementarily equivalent to an ultraproduct of the m.e.c $\mathcal{C}$ of finite tournaments. 

A finite tournament $D$ is said to be {\em regular} if any two vertices have the same out-degree. By counting in two ways the set 
$S:=\{(x,y): x \to y\}$, we see that if $D$ is a finite regular tournament then its in-degree equals its out-degree, so $D$ has an odd number of vertices.  

Arguing as in (i), if $D\in \mathcal{C}$ is sufficiently large then the out-degree of a vertex depends just on the isomorphism type of the vertex, so $D$ is regular so has odd size. Fix distinct vertices $a,b\in D$ with $a \to b$. Define the sets
$E_1:=\{x\in D: a\to x \wedge b \to x\}$, $E_2:=\{x: a \to x \wedge x \to b\}$, $E_3:=\{x: b \to x \wedge x \to a\}$ and $E_4:=\{x: x \to a \wedge x \to b\}$. Arguing as above, provided $D$ is sufficiently large each of $E_1, E_2,E_3,E_4$ is non-empty and regular, so has odd size.
Since the vertex set of $D$ is $E_1 \cup E_2 \cup E_3 \cup E_4 \cup\{a,b\}$, $D$ has an even number of vertices, a contradiction.

(iv) For any digraph $D$ let $D^{||}(x)$ be the set of vertices $y$ distinct from $x$ and non-adjacent to $x$. 
Observe that for each vertex $x$ of $Q_n$, if $n>3$ we have $Q_n^{||}(x)\cong Q_{n-1}$, and  $Q_3^{||}(x)$ is isomorphic to the generic tournament. Now if $\mathcal{C}$ is an m.e.c with ultraproduct elementarily equivalent to $Q_n$, then the set of digraphs
\[\{D^{||}(x): D\in \mathcal{C}, x\mbox{~a vertex of~} D\}\]
is a \mec{} with ultraproduct elementarily equivalent to $Q_n^{||}(y)\cong Q_{n-1}$ for $y$ a vertex of $Q_n$. The result now follows by (iii) and induction.

(v) Let $B$ be the random bipartite graph, and let $\mathcal{C}$ be a \mec{} of finite graphs with an ultraproduct elementarily equivalent to $B$; we work in a language with a binary relation symbol for adjacency and two unary predicates giving a bipartition in $B$ and in members of $\mathcal{C}$. 
Then $\mathcal{C}$ can be thinned out so that, in the language of \cite{heinrich}, its members are  {\em 3-tuple regular}: If $\Gamma\in \mathcal{C}$ and $S$ and $S'$ are sets of size at most 3 of vertices of $\Gamma$ that induce isomorphic vertex-coloured subgraphs, then the number of vertices of $\Gamma$ which are adjacent to every vertex of $S$ is equal to the number adjacent to all vertices of $S'$. It follows from \cite[Lemma 4.8]{heinrich} that either there is a matching between the two parts, or there are no edges between the parts, or the bipartite complement of one of these situations occurs. In particular, such graphs cannot have ultraproduct elementarily equivalent to $B$. 
\end{proof}

\begin{remark} \rm 
\begin{enumerate}[1.]
\item The above result makes essential use of the definability clause in the definition of a \mec{}. Also, the argument in (iii) and (iv) makes use of a parity argument which appears to be unavailable for the universal homogeneous digraph, and for the families of digraphs determined by a collection of forbidden finite tournaments. These observations suggest \autoref{randomstr} from the final section. Since many random homogeneous structures are interpretable in pseudofinite fields, the following observation is relevant. 
\item A small adjustment of the proof of (iii) shows that if $D$ is a finite tournament which is {\em 3-regular} in the sense that the number of realisations of a quantifier-free 1-type over a tuple $\bar{a}$ of length at most 3 depends just on the quantifier-free type of $\bar{a}$, then $D$ is a single vertex or a directed 3-cycle. 
\end{enumerate}
\end{remark} 

\begin{proposition} \label{weakmec-fields}
There is no weak \mec{} consisting of finite fields.
\end{proposition}

\begin{proof} See the proposition on p.44 of \cite{sury}. It is noted that given a prime power $q$, for every $m$ with $(m,q)=1$ and
$|m|\leq 2\sqrt{q}$, there is an elliptic curve $E$ over $\mathbb{F}_q$ with isogeny class corresponding to $m$, and with $q+1-m$ rational points in ${\mathbb F}_q$. In particular, the number of distinct possible sizes of the number of rational points of an elliptic curve over ${\mathbb F}_q$ increases with $q$. Since elliptic curves are uniformly definable, the result follows.
\end{proof}

\subsection{\titlecap{\mecDisplayText{}s} of groups}

We first show that there is a strong structural restriction on \mec{}s of groups. If $G$ is a finite group, let $R(G)$ denote the soluble radical of $G$ (the largest soluble normal subgroup of $G$) and $F(G)$ the Fitting subgroup of $G$ (the largest nilpotent normal subgroup of $G$ -- so $F(G)\leq R(G)$). 

\begin{proposition} \label{prpn:groups-sol} Let $\mathcal{C}$ be a \mec{} of finite groups. Then:
\begin{enumerate}[(i)]
\item There is $d\in {\mathbb N}$ such that each $G\in \mathcal{C}$ has a uniformly $\emptyset$-definable soluble normal subgroup of index at most $d$.

\item There is $e\in {\mathbb N}$ such that for each $G\in \mathcal{C}$, the quotient $R(G)/F(G)$ has derived length at most $e$. 
\end{enumerate}
\end{proposition}

\begin{proof} (i) By a theorem of Wilson \cite{wilson}, there is a formula $\psi(x)$ which defines the soluble radical $R(G)$  in each finite group $G$. Our purpose is to show that $|G:R(G)|$ is bounded. For convenience of notation, we assume $R(G)=1$ for each $G\in \mathcal{C}$, though the argument could be given directly in $G$. 

We argue as in the proof of Theorem 4.15 (Claim 1) in \cite{mpseudo}. For $G\in \mathcal{C}$, let $S=S(G)$ be the socle of $G$, the direct product of the minimal normal subgroups of $G$. Then $S=S_1 \times \ldots \times S_t$, where the $S_i$ are non-abelian simple groups. By Theorem 1.5 of \cite{ls} (together with the Feit--Thompson theorem), there is a constant $c$ such that  each finite non-abelian simple group $H$ has an involution $h$ such that each element of $H$ is a product of exactly $c$ conjugates of $h$. It follows that $S$ is parameter-definable in $G$, uniformly as $G$ ranges over $\mathcal{C}$. Indeed, if $h_i$ is such an involution chosen in $S_i$, then $S$ is exactly the set of  products of $c$ conjugates (by elements of $G$) of $(h_1,\ldots,h_t)$. Furthermore, each $S_i$ is itself uniformly definable, as the set of products of $c$ conjugates (by elements of the definable group $S$) of $h_i$. 

The number $t$ is bounded as $G$ ranges through $\mathcal{C}$. Indeed, there is a definable partial order on $S$, where $g_1<g_2$ if and only if $C_S(g_1)<C_S(g_2)$. It is easily seen that with the $h_i$ as above,
\[C_S((h_1,1,\ldots,1))>C_S((h_1,h_2,1,\ldots,1))> C_S((h_1,h_2,h_3,1,\ldots,1) >\ldots,\]
so if $t$ is unbounded then $\mathcal{C}$ has an ultraproduct with the strict order property, contrary to \autoref{mac-nsop}. 

We shall view the alternating group $\Alt_n$ as having Lie rank $n$. There is a uniform bound on the Lie rank of $S$ as $G$ ranges through $\mathcal{C}$. This can be seen in many ways, and we do not give full details. For example, consider finite alternating groups. For each $n$, there is an element $g\in \Alt_n$ which consists of a single  cycle of length  $n$ or $n-1$ (according to parity) and at 
most one fixed point, and it can be seen that $C_{\Alt_n}(g)<C_{\Alt_n}(g^2)<\ldots$, so any infinite collection of distinct alternating groups has an ultraproduct with the strict order property, again contradicting \autoref{mac-nsop}. In the case of ${\rm PSL}_n(q)$, we may replace this $g$ by the image in ${\rm PSL}_n(q)$ of an element of ${\rm SL}_n(q)$ with a single large cycle on a basis. 

Thus, each $S_i$ is of bounded size or is a simple group of Lie type $G({\mathbb F}_q)$ of bounded Lie rank (bounded as $G$ ranges through $\mathcal{C}$). By results of Ryten (Theorems 5.2.4 and 5.3.3 and Proposition 5.4.6 of \cite{ryten}), the finite field ${\mathbb F}_q$ is uniformly parameter-interpretable in $G({\mathbb F}_q)$. By \autoref{weakmec-fields}, an infinite collection of distinct finite fields cannot form a weak \mec{}. It follows by \autoref{thminterpmac}(i) that the $S_i$ have bounded size, so $|S|$ is bounded as $G$ ranges through $\mathcal{C}$. 

Finally, for $G\in \mathcal{C}$, $G$ acts on $S$ by conjugation, and since $R(G)=1$, the kernel of the action is trivial; that is, $C_G(S)=1$. It follows that $|G|\leq |\Aut(S)|$, so $|G|$ is bounded, as required.

(ii) Following \cite[Definition 3]{camina} we define ${\rm crk}(G)$ to be the conjugacy rank of $G$; that is, the number of distinct sizes greater than 1 of conjugacy classes of $G$. Clearly in a \mec{} of groups there is a uniform bound on ${\rm crk}(G)$ for $G\in \mathcal{C}$. The result now follows immediately from \cite[Theorem A]{keller}.
\end{proof}

We conjecture that in the last proposition `soluble' can be replaced by `nilpotent of bounded class'. Note that by \cite[Theorem A]{keller}, there is a bound on the Fitting height of $R(G)$ for $G\in \mathcal{C}$.

\begin{remark}\rm 
In fact \autoref{prpn:groups-sol}  only requires that $\mathcal{C}$ be a {\em weak} \mec{} of finite groups. The only point in the above proof that needs attention is the use of \autoref{thminterpmac}(i) near the end of the proof of (i). However, Ryten in fact shows that given any family of finite simple groups of fixed Lie type, the corresponding finite fields are uniformly {\em definable} (in 1-space). The result then follows from \autoref{thminterpmac}(iii) and
\autoref{weakmec-fields} -- note that in our proof above, the move to the quotient modulo $R(G)$ was unnecessary. 
\end{remark}

For the theorem below, given a ring $R$ let $L_R=(+,-,0,(f_r)_{r\in R})$ be  the usual language for left $R$-modules, and $T_R$ be the theory of left $R$-modules. A {\em positive-primitive} (p.p.) formula is one of the form
$\exists \bar{w} \bigwedge_{i=1}^k \psi_i(\bar{x},\bar{w})$, where the $\psi_i$ are atomic. If $M$ is an $R$-module and
$\phi(\bar{x},\bar{y})$ is a p.p.\ formula without parameters, then $\phi(\bar{x},\bar{0})$ defines a subgroup of $M^n$ where $|\bar{x}|=n$, and any formula $\phi(\bar{x},\bar{a})$ defines a coset of it; such a subgroup is called a {\em p.p. definable subgroup}. Given p.p.\ formulas $\phi_1({x})$ and $\phi_2({x})$, defining subgroups $G_1$ and $G_2$ respectively of $M$, an {\em invariant sentence} is one expressing, for some $t\in {\mathbb N}$, that $|G_1:G_1 \cap G_2|\leq t$. The Baur quantifier-elimination theorem for modules asserts that 
given any $L_R$-formula $\phi(\bar{x})$ there is an $L_R$-formula $\psi(\bar{x})$ that is a boolean combination of p.p.\ formulas and invariant sentences such that $T_R \models \forall \bar{x}(\phi(\bar{x}) \leftrightarrow \psi(\bar{x}))$.

We thank Charlotte Kestner for a useful discussion leading to the next proof.

\begin{theorem}\label{modules}~
\begin{enumerate}[(i)]
\item Let $R$ be a ring and $\mathcal{C}$ be the set of all finite $R$-modules, in the language of $R$-modules. Then $\mathcal{C}$ is a \mec{}.

\item The collection of all finite abelian groups is a \mec{}.
\end{enumerate}
\end{theorem}

\begin{proof}
(i) The argument is analogous to those in \cite{kestner}. 
As usual, by \autoref{thm:projection.lemma}, to prove the theorem it suffices to consider formulas $\phi(x,\bar{y})$ where $x$ is a singleton. As noted, such a formula is equivalent in all $R$-modules to a fixed boolean combination of p.p.\ formulas and invariant sentences. In a particular model, the invariant sentences are either true or false, and there are finitely many possibilities for the truth values of the sentences, so we may suppose there are no sentences involved. Hence, since a conjunction of p.p.\ formulas is a p.p.\ formula, given any formula $\phi(x,\bar{y})$ we may suppose that it has the form
\[
\bigvee_{i=1}^t \big(\psi_i(x,\bar{y}) \wedge \big(\bigwedge_{j=1}^{r_i} \neg \psi_{ij}(x,\bar{y})\big)\big)
\]
where the $\psi_i$ and $\psi_{ij}$ are p.p. We may suppose that the disjuncts define disjoint sets, so the size of any set defined by $\phi(x,\bar{a})$ is the sum of the sizes of the disjuncts, and hence that there is only one disjunct; that is, $\phi$ has the form
\[\psi(x,\bar{y}) \wedge \big(\bigwedge_{j=1}^{r} \neg \psi_{j}(x,\bar{y})\big)\]
where $\psi$ and the $\psi_j$ are p.p.

Observe first that in a given module $M$, all sets of the form $\psi(M,\bar{a})$ have the same size, since they are all cosets of the group defined by $\psi(x,\bar{0})$. Furthermore, any conjunction of the form
$\psi_{j_1}(M,\bar{a}) \wedge \cdots \wedge \psi_{j_t}(M,\bar{a})$ is either empty or a coset of the p.p.-definable group
$\psi_{j_1}(M,\bar{0}) \wedge \cdots \wedge \psi_{j_t}(M,\bar{0})$, so such sets assume at most two sizes (and there is a formula in $\bar{y}$ that determines which size arises uniformly across all $R$-modules $M$). Thus, by inclusion--exclusion, the conditions for a \mec{} hold for formulas of the form $\bigwedge_{j=1}^{r} \neg \psi_{j}(x,\bar{y})$, and likewise for formulas
\[\psi(x,\bar{y}) \wedge \big(\bigwedge_{j=1}^{r} \neg \psi_{j}(x,\bar{y})\big).\]
The result follows (the definability clause for a m.e.c. is easily verified). 

(ii) This is immediate from (i).
\end{proof}

Regarding our question above whether any \mec{} of groups consists of nilpotent-by-bounded groups, we consider next extraspecial $p$-groups, noting that these furnish examples of \mec{}s of nilpotent class 2 groups, and that these classes are not abelian-by-bounded.

Let $p$ be an odd prime. An extraspecial $p$-group of exponent $p$ is a group $G$ of exponent $p$ such that $G'=Z(G)=\Phi(G)\cong C_p$, and $x^p=1$ for all $x\in G$; here $\Phi(G)$ denotes the Frattini subgroup of $G$ (the intersection of the maximal subgroups of $G$). Such an extraspecial group, if finite, has order $p^{2n+1}$ for some $n$, is determined up to isomorphism by the pair $(p,n)$, and is a central product of copies of a certain specific group of order $p^3$. 
It is noted in \cite[Proposition 3.11]{MS08} that for fixed $p$ the class of extraspecial $p$-groups of exponent $p$ forms a 1-dimensional asymptotic class
which smoothly approximates a countably infinite extraspecial $p$-group which was shown to be $\omega$-categorical by Felgner in \cite{felgner}. In particular, any countably infinite extraspecial $p$-group of exponent $p$ is $\omega$-categorical, smoothly approximated, and has SU-rank 1. 
For more on the model theory of extraspecial groups see Milliet \cite[Appendix A]{milliet}.

Let $G_{p,n}$ denote the unique extraspecial $p$-group of order $p^{2n+1}$ (and exponent $p$). Put
\[
\mathcal{C}_{{\rm ext}}:= \{G_{p,n}: p \mbox{~an odd prime}, n\in {\mathbb N}^{>0}\}.
\]
Also let $c$ be a contant symbol not in the language $L$ of groups, let $L'=L\cup\{c\}$, and let $\mathcal{C}_{{\rm ext}}'$ contain, for each  $G\in \mathcal{C}_{{\rm ext}}$, an $L'$-expansion of $G$ with $c$ interpreted by a non-identity element of $Z(G)$. 
Let $L_{\rm bil}$ be a 2-sorted language with a sort ${\bf K}$ carrying the language of rings, a sort ${\bf V}$ carrying the language of (additive) groups (for a vector space structure), a map ${\bf K}\times {\bf V} \longrightarrow {\bf V}$ for scalar multiplication, and a binary function symbol $\beta:{\bf V} \times {\bf V} \longrightarrow {\bf K}$. Let
$\mathcal{B}$ be the collection of 2-sorted finite $L_{\rm bil}$-structures $(V, {\mathbb F}_p)$ with $\beta$ interpreted by a symplectic form on $V$.

\begin{lemma} \label{extraspecial-bilinear}
The classes $\mathcal{C}_{{\rm ext}}'$ and  $\mathcal{B}$ are uniformly bi-interpretable.
\end{lemma}

\begin{proof} To interpret $\mathcal{C}_{{\rm ext}}'$ in $\mathcal{B}$, following Lemma A.6 of \cite{milliet}, for each $(V,{\mathbb F}_p)\in \mathcal{B}$, define a group $G$ with domain $V\times ({\mathbb F}_p,+)$, with group operation \[(u,a)*(v,b)=(u+v,a+b+\beta(u,v)).\]
Interpret the constant symbol $c$ by the element $(0,1)$.
Then $G$ is extraspecial, and all members of $\mathcal{C}_{{\rm ext}}$ arise in this way. 

For the other direction, let $G=G_{p,n}$ be finite  extraspecial of exponent $p$, and $(G,c) \in \mathcal{C}_{{\rm ext}}'$ be an expansion. Then $V=G/Z(G)$ is an elementary abelian $p$-group, so can be viewed as a vector space over $\mathbf{F}_p$. There is an equivalence relation $E$ definable on $V\setminus \{0\}$, where $uEu' \Leftrightarrow C_G(uZ(G))=C_G(u'Z(G))$. (Here $C_G(uZ(G))$ just means the subgroup of $G$ commuting with all elements of the set $uZ(G)$.) It can be checked that $uEu'$ holds if and only if the 1-spaces $\langle u\rangle$ and $\langle u' \rangle$ are equal, so we recover the domain of the projective space on $V$ as $(V\setminus \{0\})/E$. 

We now show that the field structure on $Z(G)$ is uniformly definable over the parameter $c$ naming a generator for $Z(G)$. Indeed, we may identify such a parameter with the element $1$ of $\mathbb{F}_p$. Then, using that we know uniformly the 1-spaces of $V=G/Z(G)$, for $x\in G\setminus Z(G)$ and non-identity $a\in Z(G)$ we can identify $ax$ with the unique  $x'$ in the 1-space $\langle x\rangle$ such that $[x',y]=a$ whenever $[x,y]=c$. We may then identify the product $ab$ (for $a,b\in Z(G)$) with the unique $d$ such that
$[ax,by]=d[x,y]$ for any $x,y\in G\setminus Z(G)$.

It can be checked that these mutual interpretations in fact give a bi-interpretation. 
\end{proof}

\begin{corollary}\label{interpfield}
Suppose $\mathcal{U}$ is an ultrafilter on $\mathcal{C}_{{\rm ext}}$ such that for all $e\in {\mathbb N}$,
$\{G_{p,n}:p>e\} \in \mathcal{U}$. Then an infinite field is definable in the ultraproduct $H=\prod_{\mathcal{C}_{{\rm ext}}} G_{p,n}\big/\mathcal{U}$.
\end{corollary}

\begin{proof}
This follows immediately from the uniform interpretation of  $\mathcal{B}$ in  $\mathcal{C}_{{\rm ext}}'$ and its proof -- note that the field lives on the centre of the extraspecial group. 
\end{proof}


\vspace{.1\baselineskip}

\begin{corollary}~
\begin{enumerate}[(i)]
\item For any odd prime $p$ the set $\mathcal{C}_{{\rm ext},p}:=\{G_{p,n}\in \mathcal{C}_{{\rm ext}}: n\in \mathbb{N}\}$ is a \mec{}.

\item Any subset $\mathcal{C}_0$ of $\mathcal{C}_{{\rm ext}}$ which contains groups $G_{p,n}$ for arbitrarily large odd primes $p$ is a \mac{} but not a weak \mec{}.
\end{enumerate}
\end{corollary}

\begin{proof}
(i) It is noted in the proof of \cite[Proposition 3.11]{MS08} that the class $\mathcal{C}_{{\rm ext},p}$ is a sequence smoothly approximating a countably infinite extraspecial $p$-group. The result now follows from \autoref{Liesoft}.

(ii)  The automorphism group of $G_{p,n}$ acts transitively on the non-identity elements of $Z(G)$ -- see e.g.\ \cite[Theorem 1]{winter}.  The fact that $\mathcal{C}_0$ is a \mac{} now follows from the bi-interpretability in \autoref{extraspecial-bilinear} together with Theorem~\ref{bilinear-granger} and \autoref{thminterpmac}, along with Lemma~\ref{constants}(iii). That $\mathcal{C}_0$ is not a weak \mec{} is a consequence of \autoref{interpfield}
 in combination with \autoref{thminterpmac}(iii) and \autoref{weakmec-fields}. 
\end{proof}

\begin{lemma}Let $\mathcal{U}$ be an ultrafilter on $\mathcal{C}_{{\rm ext}}$. Then the ultraproduct $G:=\prod_{\mathcal{C_{{\rm ext}}}}G_{p,n}\big/\mathcal{U}$ has NSOP${}_1$ theory, and has simple theory if and only if there is some
$e\in {\mathbb N}$ such that
$\{G_{p,n}: p>e \wedge n>e\} \not\in \mathcal{U}$. 
\end{lemma}

\begin{proof} By a result of Kestner and Ramsey (personal communication \cite{ramsey}), any ultraproduct $(V,K)$ of members of $\mathcal{B}$ has NSOP${}_1$ theory, so $G$ has NSOP${}_1$ theory by \autoref{extraspecial-bilinear}. If the underlying vector space $V$ has finite dimension then $G$ is interpretable in a  pseudofinite field so has simple (in fact supersimple finite rank) theory. Likewise, if $V$ is infinite dimensional and $K$ is finite, then $(V,K)$ (equipped with $\beta$)  is a smoothly approximable Lie geometry and supersimple of rank 1. 

On the other hand, suppose that $K$ is infinite and $V$ is infinite dimensional; this corresponds to the case when for each $e$ we have $\{G_{p,n}: p>e\wedge n>e\} \in \mathcal{U}$. Then the formula $\phi(x,yz)$ (with $x,y$ ranging through the vector space sort and $z$ through the field sort)  which says $[x,y]=z$ has the tree property. Indeed, let $\{f_i:i<\omega\}$ be a linearly independent subset of $V$ with $\beta(f_i,f_j)=0_K$ for all $i,j\in \omega$, and $\{a_i:i\in \omega\}$ be an infinite subset of $K$. If $\mu \in \omega^{<\omega}$ has length $n$, let $b_{\mu i}=(f_n,a_i)$ so $\phi(x,b_{\mu i})$ asserts $\beta(x,f_n)=a_i$.
Then for any $\eta \in \omega^\omega$ the set $\{\phi(x,b_{\eta|i}):i\in \omega\}$ is consistent, but for any $\mu \in \omega^{<\omega}$, the set
$\{\phi(x,b_{\mu i}):i\in \omega\}$ is 2-inconsistent. This shows that such $(V,K)$ is not simple. For the corresponding extraspecial groups, the formula $\psi(x,yz)$ which says $x^{-1}y^{-1}xy=z$ has the tree property for essentially the same reasons.
\end{proof}

We conclude this subsection by examining a particular case of \autoref{modules}, namely the class $\mathcal{C}_{{\rm hom}}$  of all finite homocyclic groups; that is, groups $(\mathbb{Z}/p^n\mathbb{Z})^m$ as $p$ ranges through primes and $n,m$ through positive integers. For this class, the precise functions giving cardinalities of definable sets were calculated in \cite{GMS15}. For nonnegative integers $d,k$, let $S(d,k)$ be the set of functions of the form $P(X,u,v)=\sum_{i=0}^k\sum_{j=-kd}^{kd} c_{ij}X^{u(iv+j)}$, where 
 $c_{ij} \in {\mathbb Z}$ for all $0\leq i\leq k$ and $-kd\leq j\leq kd$. By Szmielew's Theorem (see \cite[Theorem A.2.2]{hodges}), modulo the theory of abelian groups, every formula $\phi(\bar{x},\bar{y})$ is equivalent to a Boolean combination of formulas of the form $t(\bar{x},\bar{y})=0$ or $p^l|t(\bar{x},\bar{y})$, where $t$ is a term in the language of groups and $p$ is prime. We say that such a Boolean combination is in {\em standard form}. 
 
 \begin{proposition}[Proposition 4.4 of \cite{GMS15}]\label{abelianexact}
Let $\phi(\bar{x},\bar{y})$ be a formula in the language of groups in standard form. Let $d$ be the greatest integer $l$ such that for some prime $p$, either some subformula $p^l|t(\bar{x},\bar{y})$ occurs in $\phi$ or some term $t(\bar{x},\bar{y})$ occurring in 
$\phi$ has a coefficient divisible by $p^l$. Then 
\begin{enumerate}[(i)]
\item There  is a finite subset $F=F(\phi)$ of $S(d,r)$ (where $r=|\bar{x}|$) such that for each $G=({\mathbb Z}/p^n{\mathbb Z})^m \in \mathcal{C}$ and $\bar{a}\in G^s$, there is $P(X,u,v)=\sum_{i=0}^k\sum_{j=-kd}^{kd} c_{ij}X^{u(iv+j)} \in F$ with $c_{ij}=0$ whenever $in+j<0$, such that $|\phi(G^r,\bar{a})|=P(p,m,n)$.
 
\item For each such function $P\in F$ there is a formula $\phi_P$ such that for each $G=(\mathbb{Z}/p^n\mathbb{Z})^m\in \mathcal{C}$ and $\bar{a}\in G^s$ we have $G\models \phi_P(\bar{a})$ if and only if $|\phi(G^r,\bar{a})|=P(p,m,n)$.
\end{enumerate} 
\end{proposition}

For the class $\mathcal{C}_{{\rm hom}}$ it is interesting to consider the model theory of different ultraproducts. Let $\mathcal{U}$ be an ultrafilter on the set $J:=\{(p,n,m): p \mbox{~prime}, n,m\in \mathbb{N}^{>0}\}$. 
We say that  $p,n$ are {\em unbounded} on the ultrafilter $\mathcal{U}$ if for all $d\in \mathbb{N}$ there is $U_d\in \mathcal{U}$ such that if $(p,n,m)\in U_d$ then $p,n>d$; similarly for other subsets of the three coordinates. Note that if $p$ is unbounded and $n$ is unbounded then $p,n$ are (together)  unbounded. A variable is {\em bounded} if it is not unbounded. 

\begin{proposition}\label{homocyclic}
Let $\mathcal{U}$ be a non-principal ultrafilter on $J$ and $G_{\mathcal{U}}:=\prod\mathcal{C}_{{\rm hom}}\big/\mathcal{U}$.
\begin{enumerate}[(i)]
\item If $p,n$ are bounded on $\mathcal{U}$ but $m$ is unbounded, then $G_{\mathcal{U}}$ has $\omega$-categorical, $\omega$-stable and smoothly approximable theory.
\item If $p$ is bounded but $n,m$ are unbounded on $\mathcal{U}$ then $G_{\mathcal{U}}$ is stable unsuperstable.
\item If $p,m$ are bounded but $n$ unbounded on $\mathcal{U}$ then $G_{\mathcal{U}}$ is superstable but not $\omega$-stable.
\item If $p$ is unbounded but $n,m$ are bounded on $\mathcal{U}$ then $G_{\mathcal{U}}$ is $\omega$-stable.
\item If $p,m$ are unbounded but $n$ is bounded then $G_{\mathcal{U}}$ is $\omega$-stable.
\end{enumerate}
\end{proposition}

We omit the proof, which is elementary. 





\subsection{Other examples of \mec{}s}

We thank Dario Garc\'ia for drawing our attention to the following example. 

\begin{theorem}   \label{pillaysm}
Let $M$ be a pseudofinite strongly minimal set. Then there is a \mec{} whose ultraproducts are all elementarily equivalent to $M$. 
Furthermore, for  each formula $\phi(\bar{x},\bar{y})$ the functions $h_\phi(\bar{y})$ in the definition of \mec{} are polynomials over ${\mathbb Z}$ in the cardinality of the finite structure, whose degree is exactly the Morley rank of the corresponding set in the ultraproduct.
\end{theorem}

\begin{proof} This is immediate from \cite[Theorem 1.1]{pillaypseud}, though the latter is formulated in terms of non-standard cardinalities.
\end{proof}

The following result of van Abel builds on \autoref{pillaysm}, exploiting the way any uncountably categorical structure is  controlled by a strongly minimal set. Note that a version of this for totally categorical structures  is already implied by \autoref{Liesoft} and \autoref{wolf-thesis}  above (using the fact that totally categorical structures are smoothly approximated, by \cite[Corollary 7.4]{chl}).  
\begin{theorem}[Proposition 5.5 of  van Abel \cite{vanabel}] \label{abel}
Let $T$ be a pseudofinite uncountably categorical theory and let $\mathcal{C}$ be a class of finite structures all of whose non-principal ultraproducts satisfies $T$. Then $\mathcal{C}$ is a polynomial \mec{}.
\end{theorem}
\begin{proof} The term `polynomial m.e.c.' is not used in \cite{vanabel}, so some elucidation is needed. The essential point is that, working in models of $T$,  if $\theta(x,\bar{d})$ defines a strongly minimal set for every $\bar{d}$ realising the isolated type $q(\bar{y})$, then in members of $\mathcal{C}$ the formula $\theta(x,\bar{y})$ is exactly balanced in the sense of 
Definition~\ref{polynomialmac}. This holds essentially by \cite[Corollary 5.8]{vanabel}.
\end{proof}

Fix a finite relational language $L$. If $d$ is a positive integer, $M$ is an $L$-structure and $a\in M$, we say that $a$ has {\em degree} $d$ if there are $d$ pairs $(R,\bar{a})$ where $\bar{a}$ is a tuple of $M$ containing $a$, $R$ is a relation symbol of $L$, and $M\models R\bar{a}$. We say that the $L$-structure $M$ has degree at most $d$ if all elements of $M$ have degree at most $d$. Let $\mathcal{C}_d^{{\rm gaif}}$ be the collection of finite $L$-structures
of degree at most $d$. 

\begin{theorem}  \label{gaifman}
The class $\mathcal{C}_d^{{\rm gaif}}$ is a \mec{}.
\end{theorem}

\begin{proof}
We only sketch the proof. The basic idea is to use the Gaifman Locality Theorem from \cite{gaifman}, which we briefly describe. See also \cite[Section 2.5]{flum}.

We first introduce some standard notation. The {\em Gaifman graph} $G(M)$ of $M$ is the simple graph with vertex set $M$, with two vertices adjacent if and only if there is a tuple containing both of them and satisfying a relation (so the degree of an element of $M$ is bounded in terms of its degree as a vertex of $G(M)$). If $a,b\in M$, then the {\em distance} $d(a,b)$ is the length of a shortest path in $M$ with endpoints $a,b$. For each $a\in M$ and $e\in {\mathbb N}$, define the {\em sphere of radius $e$ around $a$} to be 
\[
S_e(a):=\{x\in M: d(a,x)\leq e\}.
\]
There is a finite set $\psi^d_{e,1}(x),\ldots,\psi^d_{e,n_e}(x)$ of quantifier-free $L$-formulas such that if $M$ is an $L$-structure of degree at most $d$, and $a\in M$, then for some $i$ the sentence $\psi^d_{e,i}(a)$ describes the atomic diagram of the $L(a)$-structure $(S_e(a),a)$.

If $M$ is an $L$-structure, $\bar{a}=(a_1,\ldots,a_n)\in M^n$ and $k\in {\mathbb N}$, let $S_k(\bar{a}):=S_k(a_1)\cup\ldots\cup S_k(a_n)$. For each $L$-formula $\phi(\bar{x})$ and $k\in {\mathbb N}$, there is a formula $\phi^{S_k}(\bar{x})$, called a {\em local formula},  such that for each $L$-structure $M$ and $\bar{a}\in M^n$, \[M\models \phi^{S_k}(\bar{a}) \mbox{~if and only if~} S_k(\bar{a}) \models \phi(\bar{a}).\]
 
The formula $\phi^{S_k}(\bar{x})$ is obtained from $\phi(\bar{x})$ by relativising all quantifiers to $S_k(\bar{x})$. A {\em basic local sentence} has the form
\[
\exists x_1\cdots \exists x_m\bigwedge_{1\leq i<j\leq m}d(x_i,x_j)>2r \wedge \phi^{S_r}(x_i).
\]
Gaifman's Locality Theorem asserts that every first order $L$-sentence is logically equivalent to a boolean combination of basic local sentences, and that every formula $\phi(\bar{x})$ is logically equivalent to a boolean combination of local formulas and basic local sentences. As with the proof of \autoref{modules}, the proof now reduces to handling formulas $\phi(x,\bar{y})$ which are conjunctions of local formulas. 

We may suppose that such a formula $\phi(x,\bar{y})$ has the form $\sigma(x) \wedge \tau(x,\bar{y})$, where $\sigma(x)$ is a local formula (so a partial description of a radius $r$ neighbourhood of $x$) and for $\tau(x,\bar{y})$ there is $e\in \mathbb{N}$ such that for any $M\in \mathcal{C}_d$ and $\bar{a}\in M^{|\bar{y}|}$, either $|\tau(M,\bar{a})|\leq e$ or $|M\setminus \tau(M,\bar{a})|\leq e$. Clearly at most $e$ sizes of sets $\phi(M,\bar{a})$ arise from $\bar{a}$ such that $|\tau(M,\bar{a})|\leq e$. Likewise, at most $e$ sizes of $\phi(M,\bar{a})$ arise from $\bar{a}$ with $|M\setminus \tau(M,\bar{a})|\leq e$. This gives the bound on the number of possible sizes, and the definability clause follows similarly. 
\end{proof}


\section{Generalised measurable structures}\label{genmeassect}


In this and the next section we shift our focus to infinite structures, in particular to ultraproducts of \mac{}s and \mec{}s.

In \cite{MS08} the notion of a measurable structure was introduced. A structure $M$ is {\em measurable}\footnote{Some authors have renamed this 
{\em MS-measurable} to avoid conflict with other definitions of measurability. The definitions in \cite{MS08} and \cite{Elwes07} have an additional clause built in to ensure supersimplicity, but it is shown in \cite{EM08} that this clause is unnecessary.} if there is a function $h=(\mathrm{dim},\mathrm{meas})\colon\mathrm{Def}(M)\longrightarrow(\mathbb{N}\times\mathbb{R}^{>0})\cup\{(0,0)\}$ such that the following hold:

\begin{enumerate}[(i)]

\item If $X\in\mathrm{Def}(M)$ is finite then $h(X)=(0,|X|)$.

\item For every formula $\phi(\bar{x};\bar{y})$ there is a finite set 
$D_{\phi}\subseteq(\mathbb{N}\times\mathbb{R}^{>0})\cup\{(0,0)\}$ such that

\begin{enumerate}[(a)]
\item for all $\bar{b}\in M^{|\bar{y}|}$, $h(\varphi(M^{|\bar{x}|};\bar{b}))\in D_{\phi}$, and
\item for all $(d,\mu)\in D_{\phi}$, the set $\{\bar{b}\in M^{|\bar{y}|} : h(\phi(M^{|\bar{x}|};\bar{b}))
=(d,\mu)\}$ is $\emptyset$-definable.
\end{enumerate}

\item Let $X,Y\in\mathrm{Def}(M)$ and let $f\colon X\longrightarrow Y$ be a definable surjection.
As guaranteed by~(ii), there is a positive integer $r$ and $(d_{1},\mu_{1}),\ldots,(d_{r},\mu_{r})\in(\mathbb{N}\times\mathbb{R}^{>0})\cup\{(0,0)\}$ such that if 
$Y_{i}:=\{\bar{b}\in M^{|\bar{y}|} : h(f^{-1}(\bar{b}))=(d_{i},\mu_{i})\}$,
then $Y=Y_{1}\cup...\cup Y_{r}$ is a partition of $Y$ into non-empty disjoint $\emptyset$-definable sets.
Let $h(Y_{i})=(e_{i},\nu_{i})$ for $i=1,...,r$,
and let $c:=\max\{d_{1}+e_{1},...,d_{r}+e_{r}\}$,
where we suppose that this maximum is attained by the values $d_{1}+e_{1},...,d_{s}+e_{s}$.
Then
$h(X)=(c,\mu_{1}\nu_{1}+...+\mu_{s}\nu_{s})$.
\end{enumerate}

The intuition is that `$\mathrm{meas}$' is a measure and `$\mathrm{dim}$' is a dimension, so the function $h$ combines measure and dimension into one. As noted in the introduction, any infinite ultraproduct of an asymptotic class is measurable by \cite[Lemma 5.4]{MS08}.

In this section we extend this idea by allowing functions $h$ that can take values in more general algebraic structures; 
however we keep the intuition that $h$ combines measure and dimension into one.

\subsection{Measuring semirings}

The definition of `measurable', above, makes implicit use of a natural algebraic structure on the codomain $(\mathbb{N}\times\mathbb{R}^{>0})\cup\{(0,0)\}$ of the measuring function.

 \begin{example}\label{realmonomial}
 We note that $(\mathbb{N}\times\mathbb{R}^{>0})\cup\{(0,0)\}$ admits natural addition and multiplication, as well as a compatible ordering.
The multiplication and ordering are clear:
There is an associative multiplication $(d_1,\mu_1)\cdot (d_2,\mu_2)=(d_1+d_2, \mu_1\mu_2)$, with identity $(0,1)$, and there is a total ordering given by the lexicographic product of the usual orders on $\mathbb{N}$ and $\mathbb{R}^{>0}$, with $(0,0)$ the smallest element.
Addition is a little less obvious:
 \[
 (d,r) + (e,s) =
 \left\{\begin{array}{ll}(d,r+s)&\text{ if }d=e\\(d,r)&\text{ if }d>e\\(e,s)&\text{ if }d<e.\end{array}\right.
 \]
This defines an associative addition for which $(0,0)$ is the identity element and over which the multiplication distributes.
Moreover, both addition and multiplication are compatible with the ordering in the usual sense.
Writing each pair $(d,\mu)$ as $\mu{Z}^d$, where ${Z}$ is a new indeterminate, we may identify
$(\mathbb{N}\times\mathbb{R}^{>0})\cup\{(0,0)\}$
with the {\em real monomial semiring}
$\mathbb{R}\mon{}$ which is the set $\{\mu{Z}^{d}\mid\mu\in\mathbb{R}^{>0},d\in\mathbb{N}\}\cup\{0{Z}^{0}\}$
equipped with
`max-plus' addition,
standard multiplication,
and the total ordering give by
\begin{align*}
	\mu_1Z^{d_1}<\mu_2 Z^{d_2} \quad\Leftrightarrow\quad \hbox{$d_1<d_2$ {\bf or} $d_1=d_2$ and $\mu_1<\mu_2$.}
\end{align*}
Note that $\mathbb{N}$ embeds into $\mathbb{R}\mon{}$ via the identification of $n$ with $nZ^0$. 
This is an example of a more general class of `monomial semirings' that will be defined in \autoref{mon}.
 \end{example}

Following this example we give the next definition (there will be variants in the literature). 

\begin{definition}\label{orderedsemiring}
We say that $S=(S,+,\cdot,0,1,<)$ is an \em ordered semiring \rm if: 
\begin{enumerate}
\item [\bf(OS1)] $(S,+,0)$ and $(S,\cdot,1)$ are commutative monoids,
\item [\bf(OS2)] $\cdot$ distributes over $+$,
\item [\bf(OS3)] $(S,<,0)$ is a totally ordered set with least element $0$,
\item [\bf(OS4)] $\forall x,y,z\;(x\leq y\to x+z\leq y+z)$,
\item [\bf(OS5)] $\forall x,y,z\;(x\leq y\to x\cdot z\leq y\cdot z)$,
\item [\bf(OS6)] $\forall x\; (0\cdot x=x\cdot 0=0)$, and
\item [\bf(OS7)]  $\forall x \forall y\; \big( (0<x \wedge 0<y) \to 0<xy\big)$.
\item [\bf(OS8)]  $0\neq 1$.
\end{enumerate}
For $x\in S$ and $n \in \mathbb{N}^{>0}$ we write  $nx$ as an abbreviation for $x+\cdots + x$ ($n$ times). We say that $a,b\in S$ are {\em of equal magnitude\/}, and write $a\sim b$, if $a\leq b\leq n\cdot a$ or $b\leq a\leq n\cdot b$ for some $n\in\mathbb{N}^{>0}$.
Observe that $\sim$ is an equivalence relation and that $\sim$-classes on $S$ are convex.
The quotient $D:=S/{\sim}$ admits an ordering, induced by $<$, with respect to which the minimum element is the $\sim$-class of $0$, which is usually denoted $-\infty$.
We will denote the quotient map by 
$d\colon S\longrightarrow D, $ and for $x\in S$ write $[x]_{\sim}$ for $d(x)\in D$. 
We call $d$ the \em $S$-dimension\rm, or just \em dimension \rm when there is no ambiguity,
and remark that $d$ is (weakly) monotone with respect to the orderings on $S$ and $D$.
\end{definition}

There are many familiar examples of ordered semirings.
The real monomial semiring, as defined above, is an ordered semiring.
Moreover, if $R$ is any totally ordered ring, then the non-negative cone in $R$, which we denote by $R^{\geq0}$, is naturally an ordered semiring.

\medskip

An arbitrary ordered semiring can be relatively wild.
For example it is possible that $x+x=x$ for some $x>0$.
To see this, consider a totally ordered abelian group $(\Gamma\cup\{\infty\},+,0,\leq)$ with additional element $\infty$
satisfying $\infty+a=a+\infty =\infty$ for all $a$, as in valuation theory.
Then $(\Gamma_{\geq}\cup\{\infty\},\min,+,\infty,0,>)$ is an ordered semiring, where
$\Gamma_{\geq0}=\{\gamma\in\Gamma\mid\gamma\geq0\}$,
and yet $\infty+\infty=\infty>0$.

\medskip


Our intention is to end up with an algebraic  structure in which a `generalised measure' may take its values. With this in mind, we wish to exclude some very wild behaviour while allowing enough freedom. Suppose that $Y$ and $Z$ are sets of equal `high dimension', that $X$ is a set of `low dimension', and  that $X,Y,Z$ are pairwise disjoint. We wish to allow -- but not impose -- that the (generalised) measure of $X\cup Y$ is equal to the measure of $Y$. On the other hand, we want to forbid that the measure of $X\cup Z$ is equal to the measure of $Y\cup Z$. These considerations motivate the following axiom.

%

\begin{definition}
An ordered semiring $S$ with dimension $d:S\longrightarrow D$ is a \em measuring semiring \rm if:
\begin{enumerate}
\item [\bf(MS)] $\forall x,y,z\in S\;\Big((x<y\wedge d(y)=d(z))\longrightarrow x+z<y+z\Big)$.
\end{enumerate}
\end{definition}

Later in \autoref{monomialise} we will show that we can always work with measuring semirings that abstractly have the monomial structure that will be introduced in \autoref{mon}.

This axiom {\bf(MS)} is implied by additive cancellation (i.e.~$x+z=y+z\implies x=y$). 
In fact it is strictly weaker than additive cancellation as demonstrated by the following example.

\begin{example}\label{eg:weaker.than.cancellation}
The real monomial semiring $\mathbb{R}\mon{}$, defined in \autoref{realmonomial},
is an ordered semiring,
but is not cancellative
as ${Z}+{Z}^3={Z}^3={Z}^2+{Z}^3$ but ${Z} \neq {Z}^2$. 
Nevertheless, $\mathbb{R}\mon{}$ is a measuring semiring.
To see this we note that $d(\sum_{i\leq m}a_{i}{Z}^{i})=m$ if $a_{m}\neq0$.
Then $\sum_{i\leq m}a_{i}{Z}^{i}<\sum_{j\leq n}b_{j}{Z}^{j}$
(where $a_{m}\neq0\neq b_{n}$)
implies that $m<n$ or ($m=n$ and $a_{m}<b_{n}$).
\end{example}

\begin{lemma}\label{lem:multiplication.by.n}
Let $S$ be a measuring semiring, $x,y\in S$, and $n,m\in\mathbb{N}\setminus\{0\}$. Then:
\begin{enumerate}[(i)]
\item $x<y\Longleftrightarrow nx<ny$,
\item $x\leq y\Longleftrightarrow nx\leq ny$, and
\item if $0<x$ and $0<n<m$, then $nx<mx$.
\end{enumerate}
\end{lemma}
\begin{proof}
The direction $\Rightarrow$ of {(ii)} is a straightforward induction using {\bf(OS4)}.
For $\Rightarrow$ of $(i)$, we suppose $x<y$.
Note that $d(y)=d(ny)$ for all $n>0$.
By {\bf(OS4)} and {\bf(MS)},
if $nx<ny$ for some $n>0$, then $nx+x\leq nx+y<ny+y$, 
which gives $(n+1)x<(n+1)y$,
proving the implication $\Rightarrow$ of (i) by induction.
The converse implications of both (i) and (ii) now follow.
For (iii), by {\bf(MS)} we have $nx<nx+(m-n)x$.
\end{proof}

\begin{proposition}
Let $S$ be a measuring semiring. 
Then
\[n\longmapsto n\cdot 1\]
is an embedding of ordered semirings
\[(\mathbb{N},+,\cdot,0,1,<)\longrightarrow(S,+,\cdot,0,1,<).\]
\end{proposition}
\begin{proof}
All that is required to prove is that the map is injective.
First note that $0<1$ by {\bf(OS3)} and {\bf(OS8)},
and then {\bf(MS)} yields $0<1<1+1<1+1+1<\ldots$. 
\end{proof}

Let $S=(S,+,\cdot,0,1,<)$ be a measuring semiring.
As noted above, $\sim$ is an equivalence relation with convex equivalence classes. Moreover, $\sim$ respects addition and multiplication:
\begin{align*}
x\sim y &\Longrightarrow x+z\sim y+z\\[.5em]
x\sim y &\Longrightarrow x\cdot z\sim y\cdot z,
\end{align*}
for all $x,y,z\in S$.

\begin{lemma}[`Rough cancellation'] \label{rough}
Let $(S,+,\cdot,0,1,<)$ be a measuring semiring and $x,y,z\in S$ with $0<x$. If $x\cdot y\sim x\cdot z$, then $y\sim z$. 
\end{lemma}
\begin{proof} We may assume $y,z\neq 0$. Indeed, if both are $0$, then $y\sim z$, while the assumptions $0<x$ and $x\cdot y\sim x\cdot z$ are incompatible with just one of $y,z$ being $0$. Now, for the contrapositive we suppose without loss that $ny\leq z$ for all $n$. Then $n(x\cdot y)\leq x\cdot z$ for all $n$. It cannot happen that $n(x\cdot y)=x\cdot z$ for some $n$; indeed, as $0<x \cdot y$ (by {\bf (OS7)}, since $0<x$ and $0<y$) we have $2n(x\cdot y)>n(x \cdot y)$ by \autoref{lem:multiplication.by.n}(iii), so $n(x\cdot y)=x\cdot z$ would imply $2n(x\cdot y) >x\cdot z$, a contradiction. Thus $n(x \cdot y)<x \cdot z$ for all $n$, so $x \cdot y \nsim x \cdot z$. 
\end{proof}

The quotient $D=S/{\sim}$ admits induced addition and multiplication. The induced addition is in fact $\max$: If $x\leq y$, then
\[
y\leq x+y\leq 2y
\] 
and so $y\sim x+y$. Thus the quotient has the form $$D=(D,\max,\oplus,-\infty,0_D,<),$$ where $0_D=[1]_{\sim}$. Note that $-\infty\oplus[x]_{\sim}=-\infty$ and $0_D\oplus[x]_{\sim}=[x]_{\sim}$ for all $x\in S$.

\begin{lemma}\label{cancel_groupembedding}
\begin{enumerate}[(i)]
\item The operation $\oplus$ in $D$ is cancellative.
\item The monoid $(D,\oplus,<)$ embeds in an ordered abelian group\\ $(\hat{D},\oplus,<)$.
\end{enumerate}
\end{lemma}

\vspace{.25\baselineskip}

\begin{proof}~
\begin{enumerate}[(i)]

\item This follows from the `rough cancellativity' of multiplication in $S$. If $[x]_{\sim}\oplus[y]_{\sim}=[x]_{\sim}\oplus[z]_{\sim}$ and $[x]_{\sim}\neq -\infty$, then $x\cdot y\sim x\cdot z$ and $y\sim z$ by \autoref{rough}, i.e.\ $[y]_{\sim}=[z]_{\sim}$.

\item Note that $\oplus$ is commutative since $+$ is. We put $\hat{D}=D\times D\big/{\equiv}$, where $(a,b)\equiv (a',b')$ if and only 
if $a\oplus b'=a'\oplus b$. It is routine to extend  $\oplus$ and $<$ to $\hat{D}$.
\end{enumerate}
\vspace{-\baselineskip}
\end{proof}


\subsubsection{The Divisible hull}

Here we construct the divisible hull, which will be something like the ``ordered semiring tensor product with $\mathbb{Q}^{\geq0}$\,''.

\begin{definition}\label{divhull}
Let $S^{\mathrm{div}}$ be the
quotient of $\mathbb{Q}^{\geq0}\times S$ by the equivalence relation $=^{\mathrm{div}}$ defined by
\[
\bigg(\frac{a}{b},r\bigg)=^{\mathrm{div}}\bigg(\frac{c}{d},s\bigg)\text{ iff }dar=bcs.
\]
By an abuse of notation we let $\left(\frac{a}{b},r\right)$ denote its $=^{\mathrm{div}}$-equivalence class. We define addition $+^{\mathrm{div}}$ by
\[
\bigg(\frac{a}{b},r\bigg)+^{\mathrm{div}}\bigg(\frac{c}{d},s\bigg):=\bigg(\frac{1}{bd},dar+bcs\bigg),
\]
multiplication $\cdot^{\mathrm{div}}$ by
\[
\bigg(\frac{a}{b},r\bigg)\cdot^{\mathrm{div}}\bigg(\frac{c}{d},s\bigg):=\bigg(\frac{ac}{bd},r\cdot s\bigg),
\]
and the ordering $\leq^{\mathrm{div}}$ by
\[
\bigg(\frac{a}{b},r\bigg)\leq^{\mathrm{div}}\bigg(\frac{c}{d},s\bigg)\text{ iff }dar\leq bcs.
\]
Finally, let $0^{\mathrm{div}}:=(0,0)$ and $1^{\mathrm{div}}:=(1,1)$.
By another abuse of notation we write $S^{\mathrm{div}}:=(S^{\mathrm{div}},+^{\mathrm{div}},\cdot^{\mathrm{div}},0^{\mathrm{div}},1^{\mathrm{div}},\leq^{\mathrm{div}})$.
\end{definition}
Observe that, by Lemma~\ref{lem:multiplication.by.n}(i),  for all $c,d\in S$, $c<d$ implies $(1/n, c)<(1/n, d)$. 

\begin{lemma}
Let $S$ be a measuring semiring. Then the operations on $S^{\mathrm{div}}$ given above in \autoref{divhull} are well-defined, and $S^{\mathrm{div}}$ is a measuring semiring. The map $s\longmapsto (1,s)$ gives an embedding of ordered semirings $S \longrightarrow S^{\mathrm{div}}$.
\end{lemma}

\begin{proof}

The proof that the operations are well-defined is straightforward and is left as an exercise. As one case, to see that $\leq^{\mathrm{div}}$ is well-defined, suppose that  $dar\leq bcs$  and 
$\big(\frac{a}{b},r\big)=^{\mathrm{div}}\big(\frac{a'}{b'},r'\big)$. We must check that $da'r'\leq b'cs$, so suppose $da'r'>b'cs$. Multiplying by $ab$ using Lemma~\ref{lem:multiplication.by.n} (i) we obtain $da'r'ab>b'csab$. Hence, as $b'ar=ba'r'$ we have
$darab'>b'csab$, so $dar>bcs$ by \ref{lem:multiplication.by.n}(i), a contradiction. 
To show that $S^{\mathrm{div}}$ satisfies {\bf(MS)}, one simply unpacks the definitions and applies that $S$ satisfies {\bf(MS)}. Details are left to the reader. 
\end{proof}

\noindent {\bf Convention.}\enspace 
In view of this lemma, we shall from now on replace $S$ by its divisible hull $S^{\mathrm{div}}$; that is, we assume that $S$ is divisible in the sense that for all $s\in S$ and $q\in \mathbb{Q}^{>0}$ there is a well-defined element $qs$ of $S$. Observe that if $s\in S$ and $q\in \mathbb{Q}^{>0}$, then $s\sim qs$. 

\subsubsection{The standard part image of $S$}

If two elements of a measuring semiring $S$ have the same $S$-dimension, then we can compare them more finely.

\begin{definition}
Let $S$ be a measuring semiring, and let $x,y\in S$. We write $x\approx y$ if $x\sim y$ and for all $q\in\mathbb{Q}^{>1}$ we have $x\leq qy$ and $y\leq qx$.
\end{definition}

It is apparent that $x\approx y$ is an equivalence relation refining $\sim$, and we denote the quotient of $S$ by $\approx$ by $S^{\rm st}$. 

\begin{lemma}\label{cancel}
Let $S$ be a measuring semiring and $x,y,z\in S$. If $z\neq 0$ and  $xz\approx yz$ then $x\approx y$.
\end{lemma}
\begin{proof}
We may suppose $x\neq 0$, since otherwise $yz=0$, hence $y=0$ by \autoref{orderedsemiring}{\bf (OS7)}, and thus $x \approx y$. Suppose for a contradiction that  there is $q=\frac{m}{n}>1$ with $qx<y$. 
Then $qxz\leq yz$. Since $0<xz$ by {\bf (OS7)}, we have $(2m-1)xz<2mxz$ by \autoref{lem:multiplication.by.n}(iii) --- note that this uses the Measuring Axiom {\bf(MS)}. Then $(2m-1)xz<2mxz\leq 2nyz$. Write $r:=\frac{2m-1}{2n}$. By \autoref{lem:multiplication.by.n} we obtain $rxz<yz$. As $r> 1$, this contradicts $xz\approx yz$.
\end{proof}


Let $a\in S$. We define the function
\[\rho_{a}\colon S\longrightarrow\mathbb{R}^{\geq0}\cup\{\infty\}\]
by
\[
x\longmapsto\left\{\begin{array}{ll}0&\text{ if }d(x)<d(a)\\r&\text{ if }d(x)=d(a)\text{, where }r:=\sup\{q\in \mathbb{Q}:qa\leq x\}\\\infty&\text{ if }d(x)>d(a).\end{array}\right.
\]

\begin{example}\label{mon} 
Using the functions $\rho_a$ and  sacrificing some information, we can work with a measuring semiring in a canonical form.
With $S$ and its quotient $D=S/{\sim}$ as above, and ${Z}$ as an indeterminate, we form another measuring semiring $P=(\mathbb{R}\mon{D},\boxplus,\boxdot,0,1,<)$,
where $\mathbb{R}\mon{D}=\{r{Z}^{d}\mid r\in\mathbb{R}^{>0},d\in D\}\cup\{0{Z}^{0}\}$
and $n:=n{Z}^{0}$ for all $n\in\mathbb{N}$,
which we
equip with `max-plus' addition and the standard multiplication:
\[
    r_{1}{Z}^{d_{1}}\boxplus r_{2}{Z}^{d_{2}}:=\left\{\begin{array}{ll}
    (r_{1}+r_{2}){Z}^{d_{1}}&d_{1}=d_{2}\\
    r_{1}{Z}^{d_{1}}&d_{1}>d_{2}\\
    r_{2}{Z}^{d_{2}}&d_{1}<d_{2}
    \end{array}\right.
\]
and
\begin{align*}
    r_{1}{Z}^{d_{1}}\boxdot r_{2}{Z}^{d_{2}}&:=r_{1}r_{2}{Z}^{d_{1}\oplus d_{2}}.
\end{align*}
We put $r_1{Z}^{d_1}< r_2{Z}^{d_2} \Leftrightarrow (d_1,r_1)< (d_2,r_2)$ lexicographically. The measuring semiring axiom {\bf (MS)} is easily verified.
Again, there is a quotient map $d_{P}\colon P\longrightarrow D$, with $d_P(r{Z}^{d})=d$.
We call a measuring semiring of the form $P$ a {\em monomial semiring}.
\end{example}

\begin{theorem} \label{monomialise}

There is a semiring homomorphism $\phi\colon S \longrightarrow P$ such that $d=d_P \circ \phi$, with the property that for $x,y\in S$, $x\approx y$ if and only if $\phi(x)=\phi(y)$.

\end{theorem}

\begin{proof} 

In the argument below we sometimes write $ra\approx r'a'$, where $r,r'\in \mathbb{R}$ and $a,a'\in S$. This is an abbreviation for the statement
\[(\forall q,q'\in \mathbb{Q}^{>0})[((q<r\wedge r'<q') \rightarrow qa<q'a')\wedge ((q'<r' \wedge r<q)\rightarrow q'a'<qa)].\]

 We define $\phi$ by transfinite induction. For a subsemiring $A$ of $S$ write $d_A$ for the restriction of $d$ to $A$. Suppose $A$ is a subsemiring of $S$
and $\phi_A\colon A \longrightarrow P$ is a homomorphism satisfying  $d_A=d_P \circ \phi_A$, with the property that  for all $a,a'\in A$, $a\approx a'$ if and only if $\phi(a)=\phi(a')$. We write 
$\phi_A(a)=m(a)Z^{d(a)}$ for each $a\in A$. We may extend $\phi_A$ to the divisible hull of $A$ in $S$ by putting  $\phi_A(qa)=q\phi_A(a)$ for all $q\in \mathbb{Q}^{>0}$ and $a\in A$; that is, we may suppose that $A$ is divisible. 

Let $x\in S\setminus A$, and let
$B$ be the subsemiring generated by $A \cup \{x\}$. For the inductive step we must show that $\phi_A$ extends to $\phi_B\colon B \longrightarrow P$, with $d_B=d_P\circ \phi_B$ and the $\approx$-condition preserved. 
Elements of $B$ have the form $p(x)$, where $p(X)\in A[X]$. Let $d_0=d(x)$. We will put $\phi_B(x)=rZ^{d_0}$ for some $r\in \mathbb{R}^{>0}$. The map $\phi_B$ will then extend to the required homomorphism on $B$, provided $r$ can be found satisfying certain constraints. These constraints arise through equations of the form
$p(x)=q(x)$ holding in $B$, where $p(X),q(X)\in A[X]$.

Let $p(X)=\sum_{i=1}^n a_i X^i$ and $q(X)=\sum_{i=1}^l b_iX^i$, and let $p(x)=q(x)$  be such a constraint.  By dropping lower dimension terms, we may replace $p(x)=q(x)$ by $p(x)\approx q(x)$ and assume that all monomials in $p(x)$ and $q(x)$ are either 0 or of the same dimension $e$. Furthermore, by cancelling powers of $x$ using \autoref{cancel}, we may assume that the constant term of $p$ or $q$, say the constant term $a_0$ of $p$, is nonzero. Thus $\rho_{a_0}(c)$ is a positive real number for each $c\in S$ with 
$d(c)=d(a_0)=e$. Hence, for each non-zero term of form $a_ix^i$ of $p(x)$ 
we have an equivalence  $a_ix^i \approx \mu_i a_0$ for some $\mu_i\in \mathbb{R}$, where  
$\rho_{a_0}(a_ix^i)=\mu_i$.  Similarly for each non-zero term 
$b_jx^j$ we obtain the equivalence $b_jx^j\approx \nu_j a_0$. 

An equivalence such as $a_ix^i \approx \mu_i a_0$ will force  $m(x)$ to be defined so that
$m(a_i)m(x)^i=\mu_i m(a_0)$. We have to show that two such equivalences, possibly coming from polynomial equations in different dimensions, do not impose conflicting constraints on $m(x)$. So suppose we have two such constraints
\begin{equation}a_ix^i \approx \mu_i a_0 \label{21}
\end{equation}
\begin{equation}d_jx^j \approx \xi_j c_0 \label{22}
\end{equation}
where $a_i,d_j,a_0,c_0\in A$ and $\mu_i,\xi_j \in \mathbb{R}^{>0}$.
Taking the $j^{{\rm th}}$ power of both sides of \eqref{21} and the $i^{{\rm th}}$ power of both sides of \eqref{22}, and multiplying 
the first equation by $d_j^i$ and the second  by $a_i^j$, we obtain
\begin{equation}
d_j^i a_i^j x^{ij}\approx \mu_i^ja_0^jd_j^i
\end{equation}
and
\begin{equation}a_i^jd_j^ix^{ij}\approx \xi_j^ic_0^ia_i^j.
\end{equation}
Thus
\begin{equation}\mu_i^j a_0^jd_j^i\approx \xi_j^i c_0^ia_i^j.\end{equation}
Since $a_0,a_i,c_0,d_j\in A$, it follows by the inductive hypothesis that
\begin{equation}\label{26}\mu_i^j m(a_0)^jm(d_j)^i=\xi_j^im(c_0)^im(a_i)^j.
\end{equation}
We aim to extend $m$ to $B$ so that $m(x)$ is a solution to all equations of the form 
$m(a_i)X^i=\mu_im(a_0)$.
Since the coefficients are positive real numbers, each such equation has a unique positive real solution.
The problem is to check that two such equations are consistent, and this reduces to showing that the equations
 $m(a_i)X^i=\mu_i m(a_0)$ and $m(d_j)X^j=\xi_j m(c_0)$ have a simultaneous positive real solution. This is  equivalent to the 
following equation in $\mathbb{R}$ holding, which is itself equivalent to  \eqref{26}:
\[\left(\frac{\mu_i \,m(a_0)}{m(a_i)}\right)^{1/i}=\left(\frac{\xi_j\, m(c_0)}{m(d_j)}\right)^{1/j}.\]
Thus, there is no inconsistency.
The final assertion follows from the construction of $\phi$ given above. 
\end{proof}

\subsection{{\rm\em S}-measurable structures}

\begin{definition}\label{def:T-measurable}
Let $S$ be a measuring semiring and let $M$ be an $L$-structure. We say that $M$ is \em $S$-measurable \rm if there is a function $h:\mathrm{Def}(M)\longrightarrow S$ such that
\begin{enumerate}[(i)]
\item (\tt finite sets\rm) $h(X)=|X|$ for finite $X$;
\item (\tt finite additivity\rm) $h$ is finitely additive, that is, if $X,Y\in \mathrm{Def}(M)$ are disjoint then $h(X \cup Y)=h(X)+h(Y)$;
\item (\tt \mac{}~condition\rm) for each $\emptyset$-definable family $\mathcal{X}$ there exists a finite set $F\subseteq S$ such that $h(\mathcal{X})=F$ and for each $f\in F$, $h^{-1}(\{f\})$ is a $\emptyset$-definable family; and
\item (\tt Fubini\rm) suppose that $p\colon X\longrightarrow Y$ is a definable function for which there exists $f\in S$ such that for all $a\in Y$, $h(p^{-1}(\{a\}))=f$; then we have $h(X)=f\cdot h(Y)$.
\end{enumerate}

We refer to $h$ as a \emph{generalised measure}. Also, when $S$ is clear from the context, we just say that $M$ is \em generalised measurable\rm. By an abuse of notation, we also use $d$ to denote the composition $d\circ h:\mathrm{Def}(M)\longrightarrow D$. 

As with MS-measurability, the above definition also has content  if, in~(iii), we drop the definability assumption that each $h^{-1}(\{f\})$ is an $\emptyset$-definable family; without this assumption we say that $M$ is {\em weakly $S$-measurable}.

We say that $M$ is {\em $S$-ring-measurable}, or just {\em ring-measurable}, if there is a commutative ordered unital ring $S$ such that $M$ is $S^{\geq 0}$-measurable. 
\end{definition}

We assemble some simple facts about generalized measurability.

\begin{proposition}\label{easy-facts}
Let $M$ be $S$-measurable with set $D$ of dimensions.
\begin{enumerate}[(i)]
\item Any expansion of $M$ by constants is $S$-measurable. 
\item The generalised measure is monotonic, i.e.\ for all definable $A\subseteq B$ we have $h(A)\leq h(B)$.
\item Definable bijections preserve $h$.
\item Let $A:= \prod_{i=1}^n B_i$ be a cartesian product of definable sets. Then $h(A) = \prod_{i=1}^n h(B_i)$.
\item $M$ is $S'$-measurable, where $S'$ is the monomial semiring $\mathbb{R}\mon{D}$.
\item If $D=\mathbb{N}$,  then $M$ is measurable, {\em i.e.} is MS-measurable.
\end{enumerate}
\end{proposition}
\begin{proof}
Part (i) is an easy analogue of Lemma~\ref{constants}.
For (ii), let $C:=B\setminus A$. Then $h(B)=h(A)+h(C)\geq h(A)$.
Assertion (iii) is trivial using \autoref{def:T-measurable}(iv), and (iv) follows from \autoref{def:T-measurable}(iv) by induction.
Part (v) is a consequence of \autoref{monomialise}. Finally, to see (vi), observe that if $D=\mathbb{N}$ then by Theorem~\ref{monomialise} we may suppose that $S=\mathbb{R}\langle Z \rangle$, and the result then follows from the definitions. 
\end{proof}

\begin{proposition}\label{theory}
Suppose that $S$ is a measuring semiring, $M$ is $S$-measurable with generalised measure $h^{M}$, and $N\equiv M$. Then $N$ is $S$-measurable.
\end{proposition}

\begin{proof} We define a generalised measure $h^{N}$ for $N$. Let $\chi^{N}$ be a definable family of sets in $N$ determined by the formula $\phi(\bar{x},\bar{y})$, and let $\chi^{M}$ be the corresponding definable family in $M$. Let 
$h^{M}(\chi^{M})=F\subset S$, and for each $f\in F$ let $d\phi_f(\bar{y})$ be the formula defining those $\bar{a}$ such that
$h^{M}(\phi(\bar{x},\bar{a}))=f$, as guaranteed by \autoref{def:T-measurable}\,(iii). Since $F$ is finite, for each $\bar{a}\in N^{|\bar{y}|}$ it follows that 
 $d\phi_f(\bar{a})$ holds for exactly one $f\in F$, and for this $f$ we may put 
 $h^{N}(\phi(\bar{x},\bar{a}))=f$. It is easily checked that this is well-defined -- i.e.\ independent of the choice of formula defining $\chi$ -- and that $h^{N}$ satisfies \autoref{def:T-measurable}.
\end{proof}

\begin{remark}
The use of the finiteness of $F$ in \autoref{def:T-measurable}(iii) in the proof of \autoref{theory} is, by compactness, essential. The result justifies the following definition.
\end{remark}

\begin{definition}
If $T$ is a complete theory and $S$ is a measuring semiring, we say that $T$ is {\em $S$-measurable} if some model of $T$ is $S$-measurable.
\end{definition}


\begin{proposition} Let $M$ be $S$-measurable, let $D=d(S)$ and let $\hat{D}$ be the corresponding ordered abelian group as in \autoref{cancel_groupembedding}.
Then ${\rm Th}(M)$ is dimensional in the sense of \cite{wagner}, dimensions taking values in $\hat{D}$.
\end{proposition}

\begin{proof}
We must extend the dimension function on ${\rm Def}(M)$ to the family ${\rm Int}(M)$ of interpretable sets. Suppose $X$ is a definable set with $d(X)=e$, and $E$ is a definable equivalence relation on $X$. Since the $E$-classes are uniformly definable via the formula $Exy$ they take finitely many dimensions $f_1,\ldots, f_r$, with the union $S_i$ of the $E$-classes of dimension $f_i$ having dimension $g_i$, say. We extend $h$ to ${\rm Int}(M)$ by putting $h(S_i/E)=g_i-f_i\in \hat{D}$ and
$h(S/E)={\rm Max}\{g_i-f_i:1\leq i\leq r\}$. It is routine to check that the dimension so defined satisfies Definition 1.1 of \cite{wagner}. 
\end{proof}

In fact, the above proof, in combination with Theorem~\ref{monomialise},  yields that generalised measurability extends to $M^{{\rm eq}}$.
\begin{proposition}\label{meq}
Let $M$ be generalised measurable. Then $M^{{\rm eq}}$ is generalised measurable. 
\end{proposition}
\begin{proof} Using \autoref{monomialise} and \autoref{cancel} we may suppose that $M$ is $S$-measurable, where $S$ is a monomial semiring $S=\mathbb{R}\mon{D}$ with $D$ an ordered abelian group.
We extend the map $h$ in \autoref{def:T-measurable}
to interpretable sets. The key idea is that if $X$ is a definable set and $E$ is a definable equivalence relation on $X$ such that $h(X)=sZ^e$ and $h(C)=rZ^d$ for each $E$-class $C$, then we put $h(X/E)=\frac{s}{r}Z^{e-d}$. We leave the reader to verify the details. 
\end{proof}

The following observation connects to Theorem~\ref{pillaysm} above.

\begin{proposition} \label{smin-MS}
Let $M$ be a generalised measurable strongly minimal set. Then $M$ is measurable, {\em i.e.} is MS-measurable. 
\end{proposition}

\begin{proof}
We may suppose that $M$ is $S$-measurable with generalised measure $h$, where $S=\mathbb{R}\langle Z^D\rangle$ for some divisible ordered abelian group $D$. We now apply Lemma 2.3 of \cite{kestner-pillay} and its proof. Indeed, by the proof of that lemma which works equally well for generalised measurability, if $d(M)=t$, then  for any definable $X\subset M^n$ we have $d(X)=t.{\rm RM}(X)$. It follows that we may replace the measuring function $h$ by $h'$ taking values in $\mathbb{R}\langle Z\rangle$, putting $h'(X)=\mu Z^{{\rm RM}(X)}$ whenever $h(X)=\mu Z^{t{\rm RM}(X)}$ -- the axioms of generalised measurability still apply. Now apply Proposition~\ref{easy-facts}(vi). 
\end{proof}

\subsection{Ultraproducts of \mac{}s}\label{macsmeas}

We show that ultraproducts of \mac{}s are generalized measurable. 

\begin{theorem}\label{macgm}

 Let $\mathcal{C}=\{M_{i}: i\in I\}$ be an $R$-\mac{}, where $R$ is a set of functions 
$\mathcal{C}\longrightarrow\mathbb{R}^{\geq0}$. Assume that $\mathcal{C}$ contains only finitely many structures of any given finite size. Let $\mathcal{U}$ be a non-principal ultrafilter on $I$. Then there is a measuring semiring $S$ such that $M:=\prod_{i\in I}M_{i}\big/\mathcal{U}$ is $S$-measurable.
\end{theorem}

\begin{proof}
 
 
The proof consists of three main steps. We first define $S$ and show that it is a measuring semiring; this is done in Claims~\ref{defineS1}--\ref{defineS3}. Next, we define the generalised measure on definable sets $X=\phi (M; \bar b)$ in the ultraproduct $M$, show that this definition is well-defined, and in \autoref{defineS4} confirm that this definition is independent of the formula $\phi (\bar x; \bar b)$ used to define $X$. We then complete the proof of the theorem in \autoref{defineS5} by showing that $M$ is $S$-measurable.
 
We construct $S$ from $R$ as a quotient. We assume without loss of generality that $R$ is the semiring generated by the functions used to measure definable sets in the models in $\mathcal{C}$. Moreover, we can assume for every formula $\phi(\bar{x};\bar{y})$ that no member $\pi$ of the $\emptyset\/$-definable partition for $\phi(\bar{x};\bar{y})$ satisfies for some $n\in \mathbb N$ that if $(M,\bar{b})\in\pi$ then $|M|\leq n$. For $f,g\in R$, write $f\leq g$ if and only if $\{i: f(M_{i})\leq g(M_{i})\}\in\mathcal{U}$.
 
\begin{claim}\label{defineS1}
$\leq$ is a (weak) total pre-order on $R$, i.e.\ it is reflexive, transitive, and total.
\end{claim}
\begin{proof}[Proof of claim]
Reflexivity and transitivity follow from the fact that $\mathcal{U}$ is a filter. Since $\mathcal{U}$ is an ultrafilter, $\leq$ is total.
\end{proof}



Put 
\[R_{=0}:=\{f\in R : (\forall\epsilon >0)(\exists U\in\mathcal{U})
(\forall i\in U)\, 0\leq |f(M_{i})|<\epsilon\},\]
and 
$R_{>0}:=R\setminus R_{=0}$.  

Define $\approxeq$ on $R_{>0}$ by 
\[f\approxeq g\> \Leftrightarrow\> 
\left|\frac{f}{g}-1\right|\longrightarrow 0 \mbox{\ as $i\longrightarrow \infty$ on\ } { \mathcal U},\]
where the right-hand side means 
\[(\forall\epsilon >0)(\exists U\in\mathcal{U})(\forall i\in U)\bigg|\frac{f(M_{i})}{g(M_{i})}-1\bigg|<\epsilon,\]
that is, 
\[\left|f(M_{i})-g(M_{i})\right|=o(g(M_{i})) \mbox{\ as $i\longrightarrow \infty$ on\ } { \mathcal U}.\]

\begin{claim}\label{defineS2} The relation $\approxeq$ is an equivalence relation 
and $\approxeq$-equivalence classes are $\leq$-convex in $R$. Also, $R_{=0}$ is $\leq$-convex in $R$. 
\end{claim}
\begin{proof} 
That $R_{=0}$ is $\leq$-convex is evident. 

Reflexivity of $\approxeq$ is immediate. To see that $\approxeq$ is symmetric, suppose $f\approxeq g$ and let $\epsilon>0$, with $\epsilon<1$.  Let $U\in \mathcal{U}$ such that $(\forall i\in U)\bigg|\frac{f(M_{i})}{g(M_{i})}-1\bigg|<\epsilon.$ There is $\delta>0$ such that
$\delta <\frac{f(M_{i})}{g(M_{i})}$ for all $i\in U$.  There is $V\in \mathcal{U}$ with $V\subset U$  such that $(\forall i\in V)\bigg|\frac{f(M_{i})}{g(M_{i})}-1\bigg|<\epsilon\delta.$ Then for all $i\in V$ we have
$\bigg|\frac{g(M_{i})}{f(M_{i})}-1\bigg|<\epsilon.$ A similar argument proves transitivity.

To see that $\approxeq$-classes are convex, suppose that $f,g,h\in R_{>0}$ and $f\approxeq h$ and $f<g<h$. Let $\epsilon>0$ and $U\in \mathcal{U}$ satisfy $(\forall i\in U)\bigg|\frac{f(M_{i})}{h(M_{i})}-1\bigg|<\epsilon.$ We may suppose that
$f(M_i)<g(M_i)<h(M_i)$ for all $i\in U$, so $\bigg|\frac{f(M_{i})}{g(M_{i})}-1\bigg| <\bigg|\frac{f(M_{i})}{h(M_{i})}-1\bigg|$ for all $i\in U$. It follows that $\bigg|\frac{f(M_{i})}{g(M_{i})}-1\bigg|<\epsilon,$ as required. 
\end{proof} 

By the claim, $R_{>0}/{\approxeq}$ inherits a total ordering from $\leq$. We treat $R_{=0}$ as 
an additional equivalence class that we denote by $[0]$. Likewise we write $[f]$ for the $\approxeq$-class of $f\in R_{>0}$.
Let $S:=(R_{>0}/{\approxeq})\, \cup \,\{[0]\}$. It is evident that every function in $R$  is $\approxeq$-equivalent to a $\{+,\times\}$-term in functions in $R$ that measure a definable set in the models in $\mathcal{C}$. There are now several assertions to be established. 

\begin{claim}\label{defineS3} The operations $+$ and $\times$ are well-defined on $S$, and with these operations $S$ is a measuring semiring. 
\end{claim}

\begin{proof} 
First we handle the assertion that $+$ and $\times$ are well-defined on $S$, omitting many details. 
To see that $+$ is well-defined, suppose $f,f',g,g'\in R_{>0}$ with $f\approxeq f'$ and $g\approxeq g'$, and let 
$\epsilon>0$. There are $U,V\in \mathcal{U}$ such that for all $i\in U$ we have $\bigg|\frac{f(M_{i})}{f'(M_{i})}-1\bigg| <\epsilon$, that is, $|f(M_i)-f'(M_i)|<\epsilon|f'(M_i)|$), and likewise for all $i\in V$ we have 
$\bigg|\frac{g(M_{i})}{g'(M_{i})}-1\bigg| <\epsilon$,  
so $|g(M_i)-g'(M_i)|<\epsilon |g'(M_i)|$. 
Let $W=U\cap V\in \mathcal{U}$. 
For all $i\in W$ we have
$$|(f(M_i)+g(M_i))-(f'(M_i)+g'(M_i))|\\ \leq |f(M_i)-f'(M_i)|+|g(M_i)-g'(M_i)| $$
 $$ <  \epsilon(|f'(M_i)|+|g'(M_i)|)=\epsilon|f'(M_i)+g'(M_i)|. $$ 
Similar arguments show that $\times$ is well-defined and $S$ is an ordered semiring. 

To see that $S$ is a measuring semiring, define $\sim$ and the dimension function $d\colon S\longrightarrow D$ as usual.  Suppose that $f, g, h, \in S$ are such that $[f]< [g]$ and $d([g])=d([h])$.  We further suppose that $[g]\leq [h]\leq n [g]$, the other case being similar. There is $\epsilon>0$ and $U\in \mathcal{U}$ such that for all $i\in U$ we have $f(M_i)<g(M_i)$ and $1-\frac{f(M_i)}{g(M_i)}>\epsilon$. We may assume additionally  that for all $i\in U$ we have
$g(M_i)<h(M_i)<ng(M_i)$. Let $\epsilon':=\frac{\epsilon}{n+1}$. 
Then for $i\in U$ we have 
\[1-\frac{f(M_i)+h(M_i)}{g(M_i)+h(M_i)}\geq  1-\frac{f(M_i)+ng(M_i)}{(n+1)g(M_i)}\geq \epsilon' ,\]
hence $[f]+[h]<[g]+[h]$, as required. 
\end{proof}

We now define the generalised measure in the ultraproduct $M$ and show that this definition is well-defined. Let $X$ be definable in $M$ with parameters via the formula $\phi(\bar{x};\bar{b})$. Let 
$\Pi$ and $H_\Pi$ be as given in \autoref{def:mac} of an $R$-\mac{} applied to the formula $\phi(\bar{x};\bar{y})$. Write $\bar{b}=[(\bar{b}_{i})_{i\in I}]_{\mathcal{U}}$. Let $\pi\in\Pi$ be such that $\{i\in I : (M_{i},\bar{b}_{i})\in\pi\}\in\mathcal{U}$. Since $\mathcal{U}$ is an ultrafilter and $\Pi$ is finite, this choice is well-defined and unique. Define $h(X):=[M_{i}\longmapsto h_{\pi}(M_{i})]_\approxeq$,  where we abuse notation slightly by viewing 
$[0]$ as an $\approxeq$-class. Note that $h(X)\in S$. 

This assignment uses the definition of $X$ by an instance of a particular formula $\phi$. We must show the following:
\begin{claim}\label{defineS4}
The definition of $h(X)$ above is independent of the formula defining $X$. 
\end{claim} 
\begin{proof} 
To this end, suppose that $\phi(\bar{x};\bar{b})$ is as above and 
$\psi(\bar{x};\bar{c})$ also defines $X$. Let the partition and measuring functions associated with  
$\psi(\bar{x};\bar{y})$ be $\Lambda$ and $H_{\Lambda}=\{h_\lambda:\lambda\in \Lambda\}$, and 
$\lambda\in \Lambda$ be such that 
\[\left||\psi(M_{i}^{|\bar{x}|},\bar{c}_{i})|-h_{\lambda}(M_{i})\right|=o(|h_{\lambda}(M_{i})|)\]
as $i\longrightarrow\infty$ on $\mathcal{U}$. Let $U\in\mathcal{U}$ be such that 
$M_{i}\models\forall \bar{x} \, (\psi(\bar{x},\bar{c}_{i})\leftrightarrow\phi(\bar{x},\bar{b}_{i}))$ for all $i\in U$. 
Modulo the ultrafilter, we may assume without loss of generality that $h_{\pi}\leq h_{\lambda}$. 
By the triangle inequality 
\[\left|h_\pi(M_{i})-h_\lambda(M_{i})\right|=o(h_{\lambda}(M_{i}))\]
as $i\longrightarrow\infty$ on $\mathcal{U}$. Therefore $h_{\pi}\approxeq h_{\lambda}$, and 
thus the definition of $h(X)$ does not depend on the defining formula, as desired. Note that the argument must be modified slightly if $X=\emptyset$, since 
$X$ is measured by $[0]$. 
\end{proof}

For $n\in\mathbb{N}_{>0}$ we define the $\approxeq$-equivalence class $[n]$ as follows. Let 
$\varphi_n(x; y_1,\ldots ,y_n)$ be the formula stating that $y_1,\ldots ,y_n$ are distinct and $x$ is equal to one of them. Let $h_n$ be the measuring function for $\varphi_n$ as defined above, and put 
$[n]:=[h_n]_{\approxeq}$. 
The mapping $n\longmapsto [n]$ is an embedding of $(\mathbb{N}, +, \cdot)$ into $S$. We note, but do not use, that it is easy to show that for each $n\in\mathbb{N}$ there is no $\approxeq$-equivalence class between $[n]$ and 
$[n+1]$. 

With $0$ interpreted by $[0]$ and $1$ interpreted by $[1]$ in $S$, we come to our last claim:

\begin{claim}\label{defineS5}
The axioms in \autoref{def:T-measurable} are satisfied.  
\end{claim}
\begin{proof}
To see that \autoref{def:T-measurable}(i) holds, let $X$ be a finite set in $M$ of size $n$. Then $X$ is defined via a formula $\varphi_n$ as above, and it follows that $h(X)=[n]$. 

For \autoref{def:T-measurable}(ii), suppose that $X$ and $Y$ are disjoint definable sets, defined by formulas $\phi(\bar{x},\bar{a})$ and $\psi(\bar{x},\bar{b})$, say. Let $\chi(\bar{x},\bar{y}\bar{z})$ be the formula $\phi(\bar{x},\bar{y})\vee \psi(\bar{x},\bar{z})$ (so $|\bar{y}|=|\bar{a}|$ and $|\bar{z}|=|\bar{b}|$). Thus $\chi(\bar{x},\bar{a}\bar{b})$ defines $X$. There is $U\in \mathcal{U}$ such that for $i\in U$ the sets defined by $\phi(\bar{x},\bar{a}_i)$ and $\psi(\bar{x},\bar{b}_i)$ are disjoint. Let $\Pi$ be the partition associated with $\phi(\bar{x},\bar{y})$ with $H_\Pi \subset R$, and likewise let $\Theta$ and $H_\Theta$ be associated with $\psi(\bar{x},\bar{z})$, and $\Lambda$ and $H_\Lambda$ with $\chi(\bar{x},\bar{y}\bar{z})$.  We may further suppose there are $\pi \in \Pi$, $\theta \in \Theta$ and $\lambda \in \Lambda$ such that $(M_i,\bar{a}_i)\in \pi$, $(M_i,\bar{b}_i)\in \theta$ and $(M_i,\bar{a}_i\bar{b}_i)\in \lambda$ for all $i\in U$. Let $h_\pi$, $h_\theta$ and $h_\lambda$ be the corresponding elements of $R$.  

Suppose first that there is $V \in \mathcal{U}$ such that  for $i\in V$ we have $h_\lambda(M_i)\geq 3(h_\pi(M_i)+h_\theta(M_i))$. Replacing $U$ by a suitable subset of $U \cap V$, we may then arrange that for all $i\in U$ we have
\begin{align*}
\bigg|h_\pi(M_i)-|\phi(M_i,\bar{a}_i)|\bigg|&<\frac{1}{2}h_\pi(M_i),\\
\bigg|h_\theta(M_i)-|\psi(M_i,\bar{b}_i)|\bigg|&<\frac{1}{2}h_\theta(M_i),\\
\text{and}\quad\bigg|h_\lambda(M_i)-|\chi(M_i,\bar{a}_i\bar{b}_i)|\bigg|&<\frac{1}{2}h_\lambda(M_i).
\end{align*}
Since $|\chi(M_i,\bar{a}_i\bar{b}_i)|=|\phi(M,\bar{a}_i)|+|\psi(M,\bar{b}_i)|$ for $i\in U$, it follows by the triangle inequality that
$h_\lambda(M_i)-(h_\pi(M_i)+h_\theta(M_i))<\frac{1}{2}(h_\lambda(M_i)+h_\theta(M_i)+h_\pi(M_i))$, and it is easily checked that this is incompatible with $h_\lambda(M_i)\geq 3(h_\pi(M_i)+h_\theta(M_i))$.

Thus, we now assume for all $i\in U$ that $h_\lambda(M_i)<3(h_\pi(M_i)+h_\theta(M_i))$. Let $\epsilon>0$, and put $\epsilon'=\frac{\epsilon}{4}$. Adjusting $U$, we may suppose for all $i\in U$ that
\begin{align*}
\bigg|h_\pi(M_i)-|\phi(M_i,\bar{a}_i)|\bigg|&<\epsilon' h_\pi(M_i),\\
\bigg|h_\theta(M_i)-|\psi(M_i,\bar{b}_i)|\bigg|&<\epsilon' h_\theta(M_i),\\
\text{and}\quad\bigg|h_\lambda(M_i)-|\chi(M_i,\bar{a}_i\bar{b}_i)|\bigg|&<\epsilon' h_\lambda(M_i).
\end{align*}
Thus $|h_\lambda(M_i)-(h_\pi(M_i)+h_\theta(M_i))|<\epsilon' (h_\lambda(M_i)+h_\pi(M_i)+h_\theta(M_i)),$ so
\[\bigg|\frac{h_\lambda(M_i)}{h_\pi(M_i)+h_\theta(M_i)}-1\bigg|<\epsilon'\, \frac{h_\lambda(M_i)+h_\pi(M_i)+h_\theta(M_i)}{h_\pi(M_i)+h_\theta(M_i)}<\epsilon.\] It follows that
$h_\lambda \approxeq h_\pi+h_\theta$, as required for \autoref{def:T-measurable}(ii). 

\autoref{def:T-measurable}(iii) follows almost immediately from how we assigned values in $S$ to definable sets. \autoref{def:T-measurable}(iv) is proved by similar arguments; we omit the details. 
\end{proof}

With this, the proof of \autoref{macgm} is complete. \end{proof}

\begin{remark} \label{mac-nsop} 
The above argument shows that if $\mathcal{C}$ is a weak \mac{} then any ultraproduct is weakly generalised measurable, and hence, by \autoref{NSOP} and \autoref{unimod} below, does not have the strict order property and has functionally unimodular theory. 
\end{remark}

The construction of $S$ is much simpler if $\mathcal{C}$ is a \mec{}, the essential point being that because of exactness, \autoref{def:mac} holds with values taken in a ring.  


\begin{theorem} \label{ring} Suppose that $\mathcal{C}$ is a \mec{} and let $\mathcal{U}$ be an ultrafilter on $\mathcal{C}$ and $M$ the corresponding ultraproduct of $\mathcal{C}$. 
 Then there is an ordered commutative ring $S$ (an integral domain) such that $M$ is $S^{\geq 0}$-measurable.
\end{theorem}

\begin{proof} Suppose that $\mathcal{C}$ is an $R$-\mec{} and let $\bar{S}$ be the ring of functions $\mathcal{C} \longrightarrow \mathbb{Z}$ generated by $R$. We adopt the notation $h_{\pi}$, etc.\ from \autoref{def:mac}. For each definable family $\Phi$ given by the formula
$\phi(\bar{x},\bar{y})$ there is a corresponding finite set $H_{\Phi}\subset R$ and for each $r\in H_{\Phi}$ a corresponding formula
$d\phi_r(\bar{y})$ defining in each $N\in \mathcal{C}$ the set of $\bar{a}$ such that $|\phi(\bar{x},\bar{a})|=r(N)$.

Now define $\tilde{h}\colon{\rm Def}(M)\longrightarrow \bar{S}$ in the obvious way: If $X:=\phi(\bar{x},\bar{a})$, then put $\tilde{h}(X)=r$ if and only if 
$M\models d\phi_r(\bar{a})$. It can be checked that this is well-defined and that conditions (i)--(iv) of \autoref{def:T-measurable} hold for $M$, but with the ring $\bar{S}$ in place of a semiring. There is a natural pre-total order $\preceq$ on $\bar{S}$, where
$s_1\preceq s_2$ if and only if
\[\{N\in \mathcal{C}: s_1(N)\leq s_2(N)\}\in \mathcal{U}.\]
Define
\[J:=\{s\in \bar{S}: \{N\in \mathcal{C}: s(N)=0\}\in \mathcal{U}\}.\]
Then $J$ is a convex ideal of $\bar{S}$ and the pre-total order $\preceq$ on $\bar{S}$ induces a total order on $S:=\bar{S}/J$. If $\eta$ is the natural map $\bar{S} \longrightarrow S$ and $h=\eta \circ \tilde{h}$, then $h$ witnesses the statement of the theorem; since $S$ is an ordered ring, $S^{\geq 0}$ certainly satisfies the measuring semiring axiom. As $S$ is an ordered ring it is immediate that it is an integral domain. 
\end{proof}

We next explore the implications of \autoref{macgm} and \autoref{ring} for {\em polynomial} m.a.c.s and m.e.c.s.
In the context of \autoref{macgm} and its proof, we may suppose that if $f\in R$ and $r\in \mathbb{R}^{>0}$, then $rf\in R$, where $(rf)(M)=r(f(M))$ for each $M\in \mathcal{C}$. Given an ultrafilter $\mathcal{U}$ on $I$, we extend  to $R$ our earlier definition of $\sim$ on a measuring semiring. We define $\sim$ on $R$ by  putting $f\sim g$ if there is $V\in \mathcal{U}$ and $n\in \mathbb{N}^{>0}$ such that
\[\big(f(M_k)\leq g(M_k)\leq nf(M_k)\big) \vee \big( g(M_k)\leq f(M_k)\leq ng(M_k)\big) \mbox{~~for all~} k \in V.\]
Then $\sim$ is an equivalence relation on $R$. Also, for $f,g\in R$, write $f\ll g$ if $f\not\sim g$ and there is $V\in \mathcal{U}$ such that $f(M_i)<g(M_i)$ for all $i\in V$. 

\begin{lemma} \label{structureofS}
With the assumptions and notation  of \autoref{macgm}, the following hold.
\begin{enumerate}[(i)]
\item If $f,g\in R\setminus R_{=0}$ with $f \sim g$, then there is $r\in \mathbb{R}^{>0}$ such that $f \approxeq rg$.
\item If $f, g\in R$ with $f\ll g$, then $f+g \approxeq g$.
\item If $f,g\in R$ with $f \sim g$ and $f \approxeq rg$, then $f+g \approxeq (r+1)g$.
\end{enumerate}
\end{lemma}

\begin{proof}
(i) Let $X=\{q\in \mathbb{Q}: (\exists V\in \mathcal{U})( \forall k\in V)\, f(M_k)\leq qg(M_k)\}$. Then $X$ is non-empty and bounded below, so we may put $r=\inf(X)$. For every $\epsilon\in \mathbb{R}^{>0}$ there is $W \in \mathcal{U}$ such that 
for all $k\in W$ we have
\[r(1-\epsilon)g(M_k)\leq f(M_k)\leq r(1+\epsilon)g(M_k).\]
Thus $\big|\frac{f(M_k)}{rg(M_k)}-1\big|\leq \epsilon$ for all $k\in W$, so $f \approxeq rg$ as required.

(ii) Let $\epsilon>0$ and choose $n \in \mathbb{N}$ with $n>1/\epsilon$. There is $V \in \mathcal{U}$ with $nf(M_k)<g(M_k)$ for all $k \in V$. Then $|(f+g)(M_k)-g(M_k)|\leq f(M_k)\leq \epsilon g(M_k)$, as required.

(iii) The proof is similar and left to the reader.
\end{proof}

We next consider strengthenings of Theorems~\ref{macgm} and \ref{ring} for polynomial \mac{}s and \mec{}s. The main purpose of the next proposition is to support Lemmas~\ref{polymacwellorder} and \ref{polymacsimple}, which give supersimplicity of ultraproducts of polynomial m.a.c.s. 
\begin{proposition} \label{ultrapoly} Suppose that $\mathcal{C}=\{M_i:i\in I\}$ is a polynomial \mac{} containing just finitely many members of each finite size, and let $M$ be an  ultraproduct of $\mathcal{C}$ with respect to a non-principal ultrafilter $\mathcal{U}$ on $I$. Let $R$, $S$  and $\approxeq$ be as in \autoref{macgm}. Then there is a set $P$ of monomials in $R$ of form $rZ_1^{d_1}\ldots Z_k^{d_k}$, where $r\in \mathbb{R}^{>0}, d_1,\ldots,d_k\in \mathbb{N}$ and $Z_1,\ldots,Z_k$ are indeterminates, such that: 
\begin{enumerate}[(i)]
\item each element of $R\setminus R_{=0}$ lies in the $\approxeq$-class of a unique element of $P$; and
\item the ordered semiring structure on $S$ induces a total order $<$ on $P \cup \{0\}$ with least element $0$ and ordered semiring operations $\boxplus$ and $\boxdot$ on $P$ as follows, where $f,g\in P$:
\begin{align*}
    f\boxplus g&=\left\{\begin{array}{ll}
    g & \text{\rm if $f\ll g$}\\
    f & \text{\rm if $g\ll f$}\\
    \text{\rm the unique $h\in P$ with $h\approxeq (r+1)g$}& \text{\rm if $f\sim g$ and $f \approxeq rg$}
    \end{array}\right.
\end{align*}
\[
rZ_1^{d_1}\cdots Z_k^{d_k} \boxdot sZ_1^{e_1}\cdots Z_k^{e_k}\approxeq rsZ_1^{d_1+e_1} \cdots Z_k^{d_k+e_k}.
\]
\end{enumerate}
\end{proposition}

\begin{proof}
We just sketch the proof. We suppose that $\mathcal{C}$ is a polynomial \mac{} with respect to balanced formulas 
\[\delta_1(\bar{x}_1,\bar{y}_1),\ldots, \delta_k(\bar{x}_k,\bar{y}_k)\mbox{\ and\ } N_1, \ldots, N_k\]
as in \autoref{polynomialmac}. Thus, every element $h\in R$ may be identified with a polynomial $P_h(Z_1,\ldots,Z_k)\in \mathbb{R}[Z_1,\ldots,Z_k]$ and a definable set in a structure $M \in \mathcal{C}$ that has cardinality approximately
$P_h\left(|\delta_1(M^{|\bar{x}_1|},\bar{a}_1)|^{\frac{1}{N_1}},\ldots, |\delta_k(M^{|\bar{x_k}|},\bar{a}_k)|^{\frac{1}{N_k}}\right)$ for suitable $\bar{a}_1,\ldots,\bar{a}_k$. Thus, every element of $S\setminus\{0\}$ in the proof of \autoref{macgm} is an $\approxeq$-class of such polynomials.

Choose $P$ to contain a unique monomial $rZ_1^{e_1}\ldots Z_k^{e_k}$ from each $\approxeq$-class restricted to  the set of monomials. By \autoref{structureofS}, every polynomial $P_h\in R\setminus R_{=0}$ is equivalent modulo $\approxeq$ to a unique element of $P$. The ordering on $S$ defined in the proof of \autoref{macgm} induces an ordering on $P$, and the semiring operations of $S$ induce operations $\boxplus$ and $\boxdot$ on $P$ with the stated properties.
\end{proof}

\begin{proposition}\label{naturalorder}
Suppose that $\mathcal{C}$ is a polynomial \mec{} in the sense of \autoref{polynomialmac}, and let $M$ be an ultraproduct of $\mathcal{C}$. Then $M$ is $S$-measurable for some finitely generated ordered ring $S$.
\end{proposition}

\begin{proof} We follow the proof of \autoref{ring}. The ring $S$ will have the form $\mathbb{Z}[X_1,\ldots,X_f]\big/I$ for some $f$ and some convex ideal $I$. Here each $X_i$ corresponds in $\mathcal{C}$ to a certain set defined by $\delta(\bar{x}_i,\bar{a}_i)$. 
\end{proof}

\begin{example}\label{exgenmeas}
\phantom{done soon}
\begin{enumerate}[(i)]
\item Let $M$ be a measurable structure, as defined at the beginning of this section. Then $M$ is generalised measurable, with monomial measuring semiring $\mathbb{R}\mon{}$.
A definable set assigned the measure-dimension pair 
$(\mu,d)$ will be assigned the value $\mu Z^d\in \mathbb{R}\mon{}$. 
\item Consider a structure of the form $(K,W_1,\ldots,W_k)$, where $K$ is a pseudofinite field in the language of rings and 
$W_1<\ldots<W_k$ are infinite-dimensional vector spaces over $K$, with $W_{i+1}/W_i$ infinite dimensional for each $i$; these are ultraproducts of the class from \autoref{chainsubspace}. This structure is generalised measurable in a monomial semiring of the form $\mathbb{R}\mon{D}$, where $D\cong \mathbb{Z}^{k+1}$ (lexicographically ordered).
\item By Theorem~\ref{bilinear-granger}, the class of finite vector spaces equipped with a symplectic form over a finite field is an $R$-m.a.c., where $R=\mathbb{Q}(\mathbf{F})[\mathbf{V}]$. It follows from Theorem~\ref{macgm} that an ultraproduct consisting of an infinite-dimensional vector space with a symplectic form over an infinite pseudofinite field is $S$-measurable for some measuring semiring $S$. Inspection of the proof of Theorem~\ref{macgm} shows that we may take $S$ to be the set of monomials $\mu \mathbf{V}^r\mathbf{F}^s$ where $\mathbf{V}$
and $\mathbf{F}$ are indeterminates, and the pairs $(r,s)\in \mathbb{N} \times \mathbb{Z}$ are lexicographically ordered. The semiring operations are as in Proposition~\ref{ultrapoly}.

\item Let $M$ be a pseudofinite strongly minimal set. By \autoref{pillaysm}, $M$ is elementarily equivalent to an ultraproduct of a \mec{} and, in fact, \autoref{pillaysm} and \autoref{ring} ensure that $M$ is ring-measurable in the ring $\mathbb{Z}[Z]$, in which $1\ll Z\ll\ldots\ll Z^{n}\ll\ldots$.
\item If $M$ is a smoothly approximable structure, then $M$ is ring-measurable in a finitely generated ordered ring, where the number of generators is at most the number of geometries in a `standard system of geometries' in the sense of \cite[Definition 2.5.6]{Cherlin-Hrushovski03}. This follows from Proposition~\ref{naturalorder} and Theorem~\ref{wolf-thesis}.
\end{enumerate}
\end{example}


\subsection{Some basic model-theoretic properties of generalized measurable structures}

Recall that a theory $T$ has the {\em strict order property} if some model  of $T$ has an interpretable partial order containing an infinite totally ordered subset. A theory is said to have NSOP if it does not have the strict order property. 

\begin{proposition} \label{NSOP}
If $M$ is a weakly generalised measurable structure then its theory has NSOP. 
\end{proposition}
\begin{proof}
Let $S$ be a measuring semiring and let $M$ be weakly $S$-measurable. Suppose there exists a formula $\phi(\bar{x}_{1};\bar{x}_{2})$ that defines a (strict) preorder on the power $M^t$ of $M$ with arbitrarily long finite chains.
Let $\psi(\bar{z};\bar{x}_{1},\bar{x}_{2})$ be the formula
\[\phi(\bar{x}_{1};\bar{z})\wedge\phi(\bar{z};\bar{x}_{2}).\]
By Ramsey's theorem and clause~(iii) of weak $S$-measurability, there exists $s\in S$ and $\bar{a}_1,\bar{a}_2,\bar{a}_3$ such that $\phi(\bar{a}_i,\bar{a}_j)$ and $h\big(\psi(\bar{z};\bar{a}_{i},\bar{a}_{j})\big)=s$ for each $i,j\in \{1,2,3\}$ with $i<j$. Now $\psi(M;\bar{a}_{1},\bar{a}_{3})$ contains the disjoint union of $\psi(M;\bar{a}_1,\bar{a}_{2})$ and $\psi(M;\bar{a}_{2},\bar{a}_{3})$, together with the singleton $\{\bar{a}_{2}\}$. Thus
\[s\geq s+s+1\]
which contradicts \autoref{lem:multiplication.by.n}(iii).
\end{proof}

The following definition is essentially from \cite{Elwes07}, although note that there is some confusion in the literature around related notions, as clarified in \cite{kestner-pillay} and \cite{garcia-wagner}.

\begin{definition}\label{func-un}
A structure $M$ is {\em functionally unimodular} if, whenever $f_1,f_2\colon A \longrightarrow B$ are definable maps and there are integers $k_1,k_2$ such that $|f_1^{-1}(b)|=k_1$ and $|f_2^{-1}(b)|=k_2$ for all $b\in B$, then $k_1=k_2$.
\end{definition}

\begin{lemma} \label{unimod}
Let $M$ be a weakly $S$-measurable structure for some measuring semiring $S$. Then $M$ is functionally unimodular.
\end{lemma}

\begin{proof} Let $h$ be the corresponding measuring function, and adopt the notation and assumptions of Definition~\ref{func-un}.  By \autoref{def:T-measurable}(i) and (iv), we have $h(A)=k_1h(B)=k_2h(B)$, and it follows by \autoref{lem:multiplication.by.n}(iii) that $k_1=k_2$.
\end{proof}


We consider briefly the additional model-theoretic implications of {\em ring}-measurability as compared with generalised measurability. 

\begin{proposition} \label{injsur}
Suppose that the structure $M$ is $S$-ring-measurable and let $X\subset M^n$ be definable. Then
\begin{enumerate}[(i)]
\item any definable injection $f\colon X \longrightarrow X$ is surjective; and
\item any definable surjection $f\colon X \longrightarrow X$ is injective.
\end{enumerate}
\end{proposition}

\begin{proof} Let $h\colon{\rm Def}(M) \longrightarrow S$ denote the function given by $S$-measurability.

(i) Let $Y:=X\setminus f(X)$. Then $h(X)=h(f(X))+h(Y)= h(X)+h(Y)$, as $h(X)=h(f(X))$ by injectivity. Hence $h(Y)=0$, so $Y=\emptyset$.

(ii) Suppose that $X=Y_1 \cup Y_2$, where $Y_1:=\{x\in X:|f^{-1}(f(x))|=1\}$ and $Y_2=X\setminus Y_1$. Put $Z_1=f(Y_1)$ and $Z_2:=f(Y_2)$. Then $h(Y_1)+h(Y_2)=h(X)=h(Z_1)+h(Z_2)=h(Y_1)+h(Z_2)$, in view of the definable partitions $X=Y_1\cup Y_2=Z_1 \cup Z_2$ and of the fact that $h\restriction_{Y_1}$ is a bijection onto $Z_1$. As 
$h\restriction_{Y_2}$ is at least 2-to-1, we have  $h(Y_2) \geq 2h(Z_2)$. Combining these observations, 
$h(Z_2)=h(Y_2)\geq 2h(Z_2)$, forcing  $h(Z_2)=0$ and hence $Y_2=Z_2=\emptyset$.
\end{proof}

\vspace{.1\baselineskip}

\begin{remark}~
\begin{enumerate}[(i)]

\item Both parts of \autoref{injsur} require {\em ring\/}-measurability, not just generalised measurability. For example, the structure
$(\mathbb{N},S)$ (with $S$ the successor function) is strongly minimal and unimodular in the sense of 
\cite{kestner-pillay}, hence measurable by \cite[Proposition 3.1]{kestner-pillay}. Here, the definable function $S$ is injective but not surjective. 

\item For an example of a measurable structure with a definable surjection $M\longrightarrow M$ that is not injective, let $\{c_n:n\in \mathbb{Z}\setminus \mathbb{N}\}$ be disjoint from $\mathbb{Z}$ and let $M$ have domain $\mathbb{Z}\cup C$. Define $f\colon M\longrightarrow M$ by putting $f(n)=n+1$ for $n \in \mathbb{Z}$, $f(c_n)=c_{n+1}$ for $n<-1$, and $f(c_{-1})=0$. Then $(M,f)$ is again strongly minimal and unimodular, hence measurable, but $f$ is surjective but not injective. 

\item \autoref{injsur}(i) also follows from \cite[Theorem 3.1]{krajicek}.

\item If $M$ is $S$-ring-measurable where $S$ is generated as a ring by $\mathrm{Im}(h)$, then the ring $S$ is a quotient of the Grothendieck ring of $M$, as defined for example in \cite{scanlon}.

\item Just as generalised measurability extends to $M^{{\rm eq}}$ (see Proposition~\ref{meq}), also if $M$ is ring-measurable then so is $M^{{\rm eq}}$. This follows essentially because the quotient field of an ordered integral domain is an ordered field, and thus we may extend a measuring function $h$ to quotients (as in the proof of Proposition~\ref{meq}) so that it takes values in an ordered field. 

\end{enumerate}
\end{remark}



\section{Simplicity and supersimplicity}

As shown in \cite[Corollary 3.7]{EM08}, (MS-)measurable structures (defined at the beginning of Section~\ref{genmeassect}) are supersimple of finite rank,
so it is natural to ask what model-theoretic properties generalised measurable structures have.
In this section we give criteria on measuring rings $S$ which are sufficient for the supersimplicity/simplicity of $S$-measurable structures.

\subsection{Ordered polynomial rings with well-ordered variables}

\subsubsection{Well-ordered semigroups}

Let $(S,+,<)$ be a commutative totally ordered semigroup (written additively) and let $\langle A\rangle$ denote the semigroup generated by a subset $A\subseteq S$.

We require the following fact proved by Higman \cite{higman}; a simpler proof was given by Nash-Williams \cite{nash-williams}.


\begin{lemma}\label{lem:generated}
Let $A\subseteq S$ be well-ordered, positive and such that $0\in A$. Then $\langle A\rangle$ is well-ordered.
\end{lemma}

\subsubsection{Well-ordered dimensions}

Let $\{X_i:i\in I\}$ be a set of variables, let $R:=\mathbb{R}[(X_i)_{i\in I}]$ be the commutative ring generated over $\mathbb{R}$ by the  $X_i$ and let $<$ be a ring ordering on $R$. Let $M$ be an $(R^{\geq0},<)$-measurable structure with measuring function $h\colon\mathrm{Def}(M)\longrightarrow R^{\geq 0}$. Let $\sim$ be the usual equivalence relation on $R$ given by: $x\sim y$ if $x$ is in the convex hull of $\{ry\;|\;r\in\mathbb{R}^{>0}\}$. Let $\langle X\rangle$ be the semigroup generated multiplicatively by the variables $\{X_i:i\in I\}$.

\begin{lemma}
Assume $\{X_i/{\sim} : i\in I\}$ is well-ordered by the ordering induced from $R$. Then $R^{\geq0}/{\sim}$ is well-ordered.
\end{lemma}
\begin{proof}
The $\sim$-equivalence class of a given $h\in R$ is equal to the maximum of the equivalence classes of the monomials that make up $h$. Thus it suffices to show that the set of $\sim$-equivalence classes of monomials is well-ordered. This follows from \autoref{lem:generated}, applied multiplicatively.
\end{proof}

Similarly, we have:

\begin{lemma}\label{polymacwellorder}
Let $\mathcal{C}$ be a polynomial \mac{}, let $M$ be an infinite ultraproduct of members of $\mathcal{C}$ and let $S$ be obtained as in \autoref{macgm} (so $M$ is $S$-measurable). Then $S/{\sim}$ is well-ordered.
\end{lemma}

\begin{proof} By  \autoref{ultrapoly}, we may identify $S$ with a certain set $P$ of monomials in $X_1,\ldots,X_k$, with the natural multiplication. By \autoref{lem:generated} the multiplicative submonoid of $S$ generated by $\{1,X_1,\ldots,X_k\}$  is well-ordered. Since every element of $S$ is $\sim$-equivalent to an element in this monoid, the result follows.
\end{proof}
\subsubsection{Dimension versus D-rank}

For a definable set $X$ in a structure we write $D(X)$ for the (ordinal-valued) $D$-rank of $X$ (see e.g.\ \cite[Definition 2.5.6]{kim}). This should not be confused with the set $D$ of dimensions. Below, we write $d(X)$ for $d(h(X))$, for any definable set $X$. 

\begin{theorem}\label{supersimple}
Let $M$ be $S$-measurable, let $d\colon S\longrightarrow D$ be the corresponding dimension function and let $D_0=\{d(X): X\in {\rm Def}(M)\}$, the set of dimensions taken by definable sets. Suppose that $D_0$ is well-ordered and identify it with the corresponding ordinal. Then 
\begin{enumerate}[(i)]
\item $D(X)\leq d(X)$ for all definable sets $X$; and
\item $T={\rm Th}(M)$ is supersimple.
\end{enumerate}
\end{theorem}

\begin{proof} 
We first note that~(ii) follows immediately from (i) and \cite[Proposition 2.5.11]{kim}. For the proof of~(i), it suffices to show for all definable sets $X$ and all $\alpha\in\mathbf{Ord}$ that $D(X)\geq\alpha$ implies $d(X)\geq\alpha$. We prove this by transfinite induction.

Suppose that $D(X)\geq\alpha$. If $\alpha=0$ then the result is trivial since $d$ always takes non-negative values. Likewise the case when $\alpha$ is a limit ordinal is immediate. 

Suppose now that $\alpha=\beta+1$. There exists a formula $\psi(\bar{x};\bar{y})$ and $(\bar{c}_{i})_{i\in\omega}$ such that

\begin{enumerate}[(1)]
\item for all $i\in\omega$, $\psi(M;\bar{c}_{i})\subseteq X$;
\item for all $i\in\omega$, $D(\psi(M;\bar{c}_{i}))\geq\beta$; and
\item there exists $k\in\omega$ such that $\{\psi(\bar{x};\bar{c}_{i})\;|\;i\in\omega\}$ is $k$-inconsistent.
\end{enumerate}

We aim to show that $d(X)\geq\alpha$, equivalently $d(X)>\beta$. Write $X_{i}:=\psi(M;\bar{c}_{i})$. Since $X_i\subseteq X$, by induction we have $d(X)\geq \beta$, so we assume for a contradiction that $d(X)=\beta$. Let $s=h(X)$ and for each definable set $Y$ write $\rho(Y)=\rho_s(h(Y))$, the measure of $Y$ relative to $s$ (i.e.\ to $X$), as defined just before \autoref{mon}. 

Let $\rho^*:=\mathrm{min}\{\rho(X_{i})\;|\;i\in\omega\}$, which exists by clause~(iii) of \autoref{def:T-measurable}. Let $N\in\mathbb{N}$ and choose $m\geq Nk^{3}$. We restrict our attention to $X_{i}$ where $i\leq m$, and put $\mathcal{X}:=\{X_{i}\;|\; i\leq m\}$. We aim to approximate the $\rho$-measure of $\bigcup_{X_{i}\in\mathcal{X}}X_{i}$. We cannot use finite additivity directly since there can be some non-trivial intersection between different $X_{i}$. 
 
 For $j< k$, let $X^{j}$ denote those elements of $\bigcup\mathcal{X}$ which are members of precisely $j$-many $X_{i}$, and let $X^{j}_{i}:=X_{i}\cap X^{j}$. We thus have a partition $X_{i}=\bigsqcup_{j< k}X_{i}^{j}$ for each $i\leq m$. 
For $j< k$, put $I^{j}:=\{i\leq m\;|\;\rho(X^{j}_{i})\geq \rho^*/k\}$.
\begin{claim}
$\bigcup\{I^{j}\;|\;j< k\}=\{1,\ldots,m\}$.
\end{claim}
\begin{proof}[Proof of Claim]
Let $i\leq m$. Since $\rho(X_{i})\geq \rho^*$ and the partition above of $X_i$ has $k$-many elements, there must be at least one, say $X_{i}^{j}$, with measure $\geq \rho^*/k$, hence $i\in I^{j}$.
\end{proof}

Applying the claim there exists $j_0< k$ such that $|I^{j_0}|\geq m/k$. Elements of $X_{i}^{j_0}$ are included in at most $j_0$ members of $\mathcal{X}$, hence 
\[
j_0\cdot \rho(X)
\geq j_0\cdot \rho(\bigcup_{i\in I^{j_0}}X^{j_0}_{i})
\geq\sum_{i\in I^{j_0}}\rho(X_{i}^{j_0}).
\]
Therefore $\rho(X)\geq\frac{1}{j_0}\frac{m}{k}\frac{\rho^*}{k}\geq N\rho^*$. Since $N\in\mathbb{N}$ is arbitrary, $d(X)>\beta$, as required.
\end{proof}

 Under the assumptions of Theorem~\ref{supersimple} and in particular that the set $D_0$ of dimensions of definable sets is well-ordered, for a type $p$ we can define $d(p)=\Min\{d(\phi(\bar{x})): \phi(\bar{x})\in p\}$ and put $d(\bar{a}/C)=d(\tp(\bar{a}/C))$. Now write
$\bar{a}\downarrow^d_C B$ if $d(\bar{a}/B \cup C) < d(\bar{a}/C)$. 

\begin{proposition} Let $M$ be $S$-measurable, and satisfy the assumptions and notation of Theorem~\ref{supersimple}.
If $\tp(\bar{a}/B \cup C)$ forks over $C$ then $\bar{a}\not\downarrow^d_C B$.
\end{proposition}

\begin{proof} Suppose that $\tp(\bar{a}/B \cup C)$ forks over $C$.
By our assumption that $D_0$ is well-ordered,  there is a $C$-definable set $X$ containing all realisations of $\tp(\bar{a}/C)$ with $d(X)=d(\bar{a}/C)$.
Some formula $\phi(\bar{x}/\bar{b})$ in $\tp(\bar{a}/B\cup C)$ forks over $C$, so implies
$\phi_1(\bar{x},\bar{b}_1) \vee \ldots \vee \phi_r(\bar{x},\bar{b}_r)$  for some $r$, where each $\phi_i(\bar{x},\bar{b}_i)$ divides over $C$. We may suppose that $\phi(\bar{x},\bar{b})$ and the $\phi_i(\bar{x},\bar{b}_i)$ each imply that $\bar{x}\in X$. 

We claim that for each $i\leq r$, $d(\phi_i(\bar{x},\bar{b}_i))<d(X)$. To see this, fix $i\leq r$. 
There is an indiscernible sequence $(\bar{c}_j:j\in \omega)$ realising $\tp(\bar{b}_i/C)$ and some $k\in \omega$  such that the formulas 
$\phi_i(\bar{x},\bar{c}_j)$ are $k$-inconsistent. Let $X_j$ be the set of realisations of $\phi_i(\bar{x},\bar{c}_j)$. Then the $X_j$ are $k$-inconsistent subsets of $X$ with $d(X_j)=d(X_l)$ for all $j,l\in \omega$, and the proof of 
\autoref{supersimple}(i) above forces $d(X_j)<d(X)$ for all $j$. Thus $d(\phi_i(x,\bar{b}_i))<d(X)$,  as claimed. 

Given the claim, it follows that $d(\phi_1(\bar{x},\bar{b}_1) \vee \ldots\vee \phi_r(\bar{x},\bar{b}_r))< d(X)$, and so
$\phi(\bar{x},\bar{b})<d(X)$. The proposition follows. 
\end{proof}

\begin{corollary}\label{polymacsimple}
Let $\mathcal{C}$ be a polynomial \mac{}. Then any infinite ultraproduct of members of $\mathcal{C}$ has supersimple theory.
\end{corollary}

\begin{proof} This follows from \autoref{ultrapoly},  \autoref{polymacwellorder}, and \autoref{supersimple}.
\end{proof}


\begin{example} In each of the cases in \autoref{exgenmeas} the semiring is a polynomial ring or monomial semiring and the dimension function $d$ takes values in some well-ordered set. Thus all of these examples are supersimple by \autoref{supersimple}. 

\end{example}




\subsection{Locally well-ordered dimensions}

Next, we localise the results in the last subsection to instances of a formula. 

Let $\phi(\bar{x};\bar{y})$ be an $L$-formula and let $L_{\phi}$ denote the set of instances of $\phi$, i.e.\ formulas 
$\phi(\bar{x};\bar{c})$ for a $\bar{y}$-tuple of parameters $\bar{c}$. Let $\Delta_{\phi}$ be the set of finite positive 
$\phi$-types, i.e.\ conjunctions of instances of $\phi$, and put $h(\Delta_\phi)=\{h(p):p\in \Delta_\phi\}$. For an $L$-structure $M$ we write $\mathrm{Def}_{\phi}(M)$ for the collection of sets in $M$ defined by formulas from $\Delta_{\phi}$.

Now let $M$ be an $S$-measurable structure with generalised measure  $h$, and $d\colon S\longrightarrow D$  the corresponding dimension function.
Suppose for each $\phi$ that $h(\Delta_{\phi})/{\sim}$ is well-ordered. We may identify this set with its corresponding ordinal and define $d_\phi(p):={\rm Min}\{d(h(\psi)):\psi\in p\}$ for all $p\in \Delta_\phi$. Likewise, if $X\in \mathrm{Def}_{\phi}(M)$ is defined by $p\in \Delta_{\phi}$, then put $d_\phi(X):=d_\phi(p)$. The rank $D(p,\phi,k)$ (for partial types $p$) is defined in \cite[Definition 2.3.4]{wagner2}

\begin{fact}\cite[Theorem 2.4.7]{wagner2}
A theory $T$ is simple if and only if, for all  formulas $\phi$ and all $k<\omega$, $D(\bar{x}=\bar{x},\phi,k)<\omega$.
\end{fact}
\begin{theorem}
Suppose that $h(\Delta_{\phi})/{\sim}$ is well-ordered for each $\phi$. Let $X\in\mathrm{Def}_{\phi}(M)$. Then $D(X,\phi,k)\leq d_{\phi}(X)$.
\end{theorem}
\begin{proof} The proof is the same as that for \autoref{supersimple}(i).
\end{proof}
\begin{corollary}\label{simplecor}
Suppose that $h(\Delta_{\phi})/{\sim}$ is well-ordered for each $\phi$. Then $M$ is simple.
\end{corollary}

We also state a characterisation of stability for a formula, under the assumption that an appropriate set of dimensions is well-ordered. For an $L$-formula $\phi(\bar{x};\bar{y})$, let $\Delta_\phi^*$ be the set of finite $\phi$-types; that is, the set of finite conjunctions of instances of $\phi$ and their negations. 

\begin{theorem}\label{stablethm}
Let $M$ be generalised measurable, let $\phi(\bar{x},\bar{y})$ be a formula and suppose that $h(\Delta^*_{\phi})/{\sim}$ is well-ordered. Then the following are equivalent:
\begin{enumerate}[(i)]
\item $\phi(\bar{x},\bar{y})$ is unstable.
\item There is an $N\equiv M$ and for some $A\subset N$ an $A$-definable set $D\subset N^{|\bar{x}|}$ and sequence $\{\bar{a}_i:i\in \omega\} \subset N^{|\bar{y}|}$ indiscernible over $A$ such that $d(D)=d(D \wedge \bigwedge_{i<k}\phi(\bar{x},\bar{a}_i))$ for all $k \in \omega$ and $\rho_{h(D)}(\phi(\bar{x},\bar{a}_i)\wedge \phi(\bar{x},\bar{a}_j))<\rho_{h(D)}(\phi(\bar{x},\bar{a}_i))$ for all $i<j$.
\item $\phi(\bar{x},\bar{y})$ has the independence property.

\end{enumerate}
\end{theorem}
\begin{proof} The proof is essentially the same as that of \cite[Proposition 3.3]{GMS15}, with the well-ordering assumption in place of $(A_\phi^*)$; we omit the details. 
\end{proof}

\begin{remark} \rm
We refer back to Propositions~\ref{abelianexact} and \ref{homocyclic} above for the \mec{} example consisting of the collection $\mathcal{C}$ of all finite homocyclic groups $(\mathbb{Z}/p^n\mathbb{Z})^m$ where $m,n \in \mathbb{N}^{>0}$ and $p$ is prime. Stability and simplicity properties of ultraproducts can be read off from the formulas for cardinalities of definable sets.  In particular, Proposition~\ref{homocyclic}(ii), in combination with Theorems~\ref{modules}(ii) and \ref{ring}, gives examples of ring-measurable stable unsuperstable pseudofinite groups. 
\end{remark}

\begin{remark} \rm
An interesting example of a generalised measurable structure  is discussed in \cite{berenstein}.
Let $M$ be a structure which is {\em coherent measurable} in the sense of \cite[Definition 5.3]{berenstein}. This means that $M$ is measurable of SU-rank 1, dimension coincides with SU-rank on definable sets, and $M$ is `nowhere trivial', a condition which typically follows if $M$ has some algebraic structure. Consider the structure $(M, H(M))$ where $H$ is a predicate for a `generic' independent subset of $M$. Then, by \cite[Theorem 5.16]{berenstein}, $(M,H(M))$ is generalised measurable and the set $D$ of dimensions has order type $\mathbb{N} \times \mathbb{N}$,  lexicographically ordered. In particular the set of dimensions is well-ordered, so $(M,H(M))$ is supersimple, as already follows from \cite{vass}. 
\end{remark}



\subsection{Quasifinite fields}

In this section we adapt the main result of \cite{Scanlon00} to the context of generalised measurable fields.

\begin{definition}
A field $K$ is \em quasifinite \rm if it is perfect and $G_{K}:=\Aut(K^{{\rm alg}}/K)\cong\widehat{\mathbb{Z}}$.
\end{definition}

\begin{definition}
A \em strong ordered Euler characteristic \rm on a field $K$ is a function 
\[\chi\colon\mathrm{Def}(K)\longrightarrow R_{\geq0}\]
into a partially-ordered ring $R$ such that for all $X,Y\in\mathrm{Def}(K)$:
\begin{enumerate}[(i)]
\item $\chi(X)=\chi(Y)$ if $X$ and $Y$ are definably isomorphic;
\item $\chi(X\times Y)=\chi(X)\cdot\chi(Y)$;
\item $\chi(X\cup Y)=\chi(X)+\chi(Y)$ if $X$ and $Y$ are disjoint; 
\item if $f\colon X\longrightarrow Y$  is a definable function such that $c=\chi(f^{-1}(\{y\}))$ for every $y\in Y$, then 
$\chi(X)=c\cdot\chi(Y)$. 
\end{enumerate}
The Euler characteristic is nontrivial if $0<1$ in $R$ and the image of $\chi$ is not $\{0\}$.
\end{definition}

\begin{theorem}\rm(Theorem 1, \cite{Scanlon00}) \em
Any field admitting a nontrivial strong ordered Euler characteristic is quasifinite.
\end{theorem}

Note that a generalised measurable field need not admit a nontrivial strong ordered Euler characteristic: In the definition above $R$ is a ring whereas the definition of a generalized measurable structure uses a semiring. 
However, we can adjust the proof of Theorem~5.18 of \cite{MS08}  to the generalised measurable context to obtain the following theorem. 

\begin{theorem}
Let $K$ be a  generalised measurable field. Then $K$ is quasifinite.
\end{theorem}
\begin{proof}
For convenience, we may assume by \autoref{monomialise} that the measuring function for $K$ takes values in a monomial semiring (see \autoref{mon}).  We note that $K$ is perfect by a simple measure argument. For the rest, the proof of Theorem~5.18 of \cite{MS08} easily adapts: Dimension in $D$ serves as a proxy for $S_1$ rank, and for disjoint sets of the same dimension, (relative) measure is additive, as is apparent from the definition of a monomial semiring. We omit further details.
\end{proof}


\section{Further Observations}

The examples of \mec{}s in Section 4 have ultraproducts which are smoothly approximable, or are modules, or are graphs of bounded degree (or at least have Gaifman graph of bounded degree), or are uncountably categorical pseudofinite. All such structures are simple and one-based in the sense of  \cite[Definition 5.5.14]{kim}. This suggests the following question.

\begin{problem} \label{simplemec} Find an example of a \mec{} with an ultraproduct which does not have simple theory. Must it have NSOP${}_1$ theory? If the ultraproduct has simple theory, must it be one-based?
\end{problem}

Related to the distinction between \mac{}s and \mec{}s, we ask:

\begin{question} Are there model-theoretic consequences of adjusting the error term $o(h_\pi(M))$ in the definition of an $R$-\mac{} for certain $R$ (e.g.\ for polynomial $R$-\mac{}s)?
\end{question}


We find the conjecture below from Section~\ref{multiexactsection} particularly appealing.

\vspace{.5\baselineskip}

\begin{conjecture}~ \label{mainconj}
\begin{enumerate}[(i)]
\item Let $M$ be a homogeneous structure over a finite relational language $L$. Then there is an \mec{} with ultraproduct elementarily equivalent to $M$ if and only if $M$ is stable.

\item Let $\mathcal{C}$ be an \mec{} and let $M$ be an unstable homogeneous structure over a finite relational language. Then $M$ is not elementarily equivalent to any structure interpretable in an ultraproduct of $\mathcal{C}$.
\end{enumerate}
\end{conjecture}

The above conjecture leads to natural questions, for a given finite relational language $L$, concerning finite $L$-structures without any symmetry assumption but with arbitrarily high levels of combinatorial regularity; for example with the 5-regularity and 3-tuple regularity  mentioned  respectively in  parts~(ii) and~(v) of the proof of \autoref{homog2}. Such issues are considered also in \cite{heinrich}, with connections mentioned to the Weisfeiler--Leman algorithm for graph isomorphism. It would be interesting to explore such notions further, for arbitrary finite relational languages, in the spirit of Lachlan's shrinking and stretching theory for finite homogeneous structures, cf.\ \cite{lachlan2}. (The structure theory in the latter depends on the existence of a bound -- dependent only on $L$ -- on a certain rank for finite homogeneous $L$-structures; the existence of this bound was verified in \cite{cherlin-lachlan}, using substantial finite group theory, though a somewhat shorter proof is given in the binary case in \cite{lachlan-shelah}.) More precisely:

\begin{problem}
Let $L$ be a finite relational language. We say that the class $\mathcal{C}$ of finite $L$-structures is {\em eventually regular} if the following hold:
\begin{enumerate}[(i)]
\item For each $L$-formula $\phi(\bar{y})$ there is a quantifier-free $L$-formula $\psi(\bar{y})$ such that all but finitely many $M \in \mathcal{C}$ satisfy $M\models \forall \bar{y}(\phi(\bar{y})\leftrightarrow \psi(\bar{y}))$.
\item There is a function $f_L\colon \mathbb{N}\longrightarrow \mathbb{N}$ such that for any $n\in \mathbb{N}$, any quantifier-free formula $\phi(x,\bar{y})$ with $|\bar{y}|=n$, and any $\bar{a}\in M^n$ where $M\in \mathcal{C}$ has size at least $f_L(n)$, the size
of $\phi(M,\bar{a})$ depends only on the quantifier-free type of $\bar{a}$.
\end{enumerate}
Show that Lachlan's shrinking and stretching theory from \cite{lachlan2} applies to eventually regular classes. In particular, show that all but finitely many members of such $\mathcal{C}$ are homogeneous.
\end{problem}

Cameron's result at the end of \cite{cameron}, cited above in the proof of Theorem~\ref{homog2}(ii), yields that this holds for finite graphs even without assumption (i). Likewise, by \cite{heinrich}, it holds (without assumption (i) above) for finite graphs expanded by unary predicates. Ainslie \cite{ainslie} has extensive partial results in this direction for finite structures in  a language with three symmetric irreflexive binary relations such that every pair of distinct vertices satisfies exactly one relation. Ainslie also proves Conjecture~\ref{mainconj}(i) for several other homogeneous structures, namely the universal metrically homogeneous graph of any fixed finite diameter, the universal homogeneous two-graph, and the `semifree' binary structures listed by Cherlin  in the appendix of 
\cite{cherlin_98}.

\begin{question}\label{randomstr}~
\begin{enumerate}[(i)]
\item Is there a weak \mec{} with an ultraproduct elementarily equivalent to the random graph?

\item Is there a \mec{} with an ultraproduct elementarily equivalent to the random (i.e.\ universal homogeneous) digraph? Or to the random 3-uniform hypergraph?
\end{enumerate}
\end{question}

\begin{question}
 Clarify the model-theoretic implications of ring-measurability (as a strengthening of generalised measurability). For example, does it imply simplicity of the theory? Also, is there a ring-measurable structure that is not pseudofinite? Is there an infinite ring-measurable field? Is every ring-measurable group soluble-by-finite?
\end{question}

\begin{problem}
Clarify the connections between \mac{}s and \mec{}s and the Hrushovski--Wagner notion of pseudofinite dimension (\cite{hrushovski-wagner, hrush-pseud}). In particular, clarify the connections to the key concepts of \cite{GMS15}. For example, what natural conditions on a \mac{} ensure that any ultraproduct (in the appropriate extended language) satisfies conditions such as  (SA), (DC), and (FMV) from \cite{GMS15}. There are some results of this kind in \cite{vanabel}, building on \autoref{abel}. 
\end{problem}

In this direction, we note the following result. For the definitions in (i), see \cite{GMS15}.
\begin{proposition}
Let $\mathcal{C}$ be a polynomial m.a.c..
\begin{enumerate}
\item[(i)] Any infinite ultraproduct of members of $\mathcal{C}$ satisfies (SA), (DC${}_L$), (MD${}_L$).
\item[(ii)] Tao's Algebraic Regularity Lemma (as expressed in \cite[Theorem 6.4]{GMS15}) holds for graphs uniformly definable in members of $\mathcal{C}$.
 \end{enumerate}
\end{proposition}

\begin{proof} (i) (SA) holds essentially by Lemma~\ref{polymacwellorder}, and (DC${}_L$) and  (MD${}_L$) follow from definability of generalised measure.

(ii) This follows from \cite[Theorem 6.4]{GMS15}.
\end{proof}

By Proposition 6.5 of \cite{Elwes07} (see \cite{kestner-pillay}) every stable measurable structure is one-based. This suggests the following question. 

\begin{question}
Is every stable generalised measurable structure one-based? 
\end{question}

Note that every generalised measurable strongly minimal set is measurable by Proposition~\ref{smin-MS}, and so is one-based.

We say that a homogeneous structure is {\em free homogeneous} if its age is a free amalgamation class. The first author has shown in \cite{A16} that any free homogeneous structure is generalised measurable; in particular the universal homogeneous triangle-free graph is generalised measurable even though it does not have supersimple theory so is not MS-measurable. As another example, let $M$ be the universal homogeneous tetrahedron-free 3-hypergraph (so determined by the minimal forbidden configuration of 4-set all of whose 3-subsets are edges). Then $M$ is supersimple of SU-rank 1 and even one-based, and not MS-measurable (see \cite[Theorem 7.3.9]{marimon}), but is generalised measurable as it is free-homogeneous. The corresponding measuring semiring has infinite descending chains of dimensions. In light of Theorem~\ref{supersimple}, we ask

\begin{question} Is the universal homogeneous tetrahedron-free 3-hypergraph generalised measurable with a measuring semiring having well-ordered dimensions?
\end{question}

Section 6.3 suggests the following question. 

\begin{question} Must  generalised measurable fields be PAC?
\end{question}

Finally, we remark that Evans \cite{evans} and Marimon \cite{marimon2} have recently shed further light on the content of MS-measurability, finding connections to $n$-amalgamation, and developing new methods for showing structures are {\em not} MS-measurable. It would be interesting to explore whether these methods apply to generalised measurability.




\section*{Acknowledgements}

The authors would like to thank Charlotte Kestner for a useful conversation regarding the proof of  \autoref{modules}, and Dar\'io Garc\'ia for several helpful conversations. They would also like to thank 
Bethany Marsh and William Crawley-Boevey for very helpful discussions on the representation theory of quivers of finite representation type. 

\def\bibfont{\footnotesize}
\bibliographystyle{plain}

\end{document}